\documentclass[11pt]{article}
\usepackage{smile}

\usepackage{fullpage}
\usepackage{lscape}
\usepackage{bigints}
\usepackage{framed}
\usepackage{mdframed}
\usepackage{enumerate}
\usepackage[inline]{enumitem}
\usepackage[T1]{fontenc}
\usepackage{moresize}
\usepackage{bm}
\usepackage{bbm}
\usepackage{dsfont}
\usepackage{amsmath}
\usepackage{amssymb}
\usepackage{amsthm}
\usepackage{amsfonts}
\usepackage{stmaryrd}
\usepackage{array}
\usepackage{mathrsfs}
\usepackage{mathtools} 
\usepackage{extarrows}
\usepackage{stackrel}
\usepackage{relsize,exscale}
\usepackage{scalerel}
\usepackage[nodisplayskipstretch]{setspace}
\usepackage{color}
\usepackage[usenames,dvipsnames]{xcolor}
\usepackage{cancel}
\usepackage{soul}
\usepackage{undertilde}
\usepackage{xfrac}
\usepackage{siunitx}
\usepackage{graphicx}
\usepackage{float}
\usepackage{rotating}
\usepackage{subcaption}
\usepackage{overpic}
\usepackage[all]{xy}
\usepackage{tikz}
\usetikzlibrary{arrows,matrix,positioning,calc,automata,patterns}
\usepackage{booktabs}
\usepackage{dcolumn}
\usepackage{multirow}
\usepackage{diagbox}
\usepackage{tabularx}
\usepackage{verbatim}
\usepackage{listings}
\usepackage[ruled,vlined]{algorithm2e}
\usepackage{fancyvrb}
\usepackage{hyperref}
\usepackage[round]{natbib}
\usepackage{sectsty}

\hypersetup{
    bookmarks=true,         
    unicode=false,          
    pdftoolbar=true,        
    pdfmenubar=true,        
    pdffitwindow=false,     
    pdfstartview={FitH},    
    pdftitle={My title},    
    pdfauthor={Author},     
    pdfsubject={Subject},   
    pdfcreator={Creator},   
    pdfproducer={Producer}, 
    pdfkeywords={key1, key2}, 
    pdfnewwindow=true,      
    colorlinks=true,        
    linkcolor=blue,         
    citecolor=blue,         
    filecolor=blue,         
    urlcolor=cyan           
}

\usepackage{stackengine}
\stackMath
\newcommand\tenq[2][1]{%
\def\useanchorwidth{T}%
\ifnum#1>1%
\stackunder[0pt]{\tenq[\numexpr#1-1\relax]{#2}}{\!\scriptscriptstyle\thicksim}%
\else%
\stackunder[1pt]{#2}{\!\scriptstyle\thicksim}%
\fi%
}

\makeatletter
\DeclareRobustCommand\widecheck[1]{{\mathpalette\@widecheck{#1}}}
\def\@widecheck#1#2{%
    \setbox\z@\hbox{\m@th$#1#2$}%
    \setbox\tw@\hbox{\m@th$#1%
       \widehat{%
          \vrule\@width\z@\@height\ht\z@
          \vrule\@height\z@\@width\wd\z@}$}%
    \dp\tw@-\ht\z@
    \@tempdima\ht\z@ \advance\@tempdima2\ht\tw@ \divide\@tempdima\thr@@
    \setbox\tw@\hbox{%
       \raise\@tempdima\hbox{\scalebox{1}[-1]{\lower\@tempdima\box
\tw@}}}%
    {\ooalign{\box\tw@ \cr \box\z@}}}
\makeatother

\def\given{\,|\,}

\def\Biggiven{\,\Big{|}\,}
\def\tr{\mathop{\text{tr}}\kern.2ex}
\def\tZ{{\tilde Z}}
\def\tX{{\tilde X}}

\def\P{{\mathrm P}}

\def\E{{\mathrm E}}

\def\d{{\mathrm d}}
\def\cI{{\mathcal I}}

\newcommand{\zahl}[1]{\llbracket #1\rrbracket}
\newcommand\yestag{\addtocounter{equation}{1}\tag{\theequation}}
\newcolumntype{L}[1]{>{\raggedright\let\newline\\\arraybackslash\hspace{0pt}}m{#1}}
\newcolumntype{C}[1]{>{  \centering\let\newline\\\arraybackslash\hspace{0pt}}m{#1}}
\newcolumntype{R}[1]{>{ \raggedleft\let\newline\\\arraybackslash\hspace{0pt}}m{#1}}
\newcolumntype{d}[1]{D{.}{.}{#1}}
\newcolumntype{H}{>{\setbox0=\hbox\bgroup}c<{\egroup}@{}}
\newcolumntype{Z}{>{\setbox0=\hbox\bgroup}c<{\egroup}@{\hspace*{-\tabcolsep}}}
\newcolumntype{b}{X}
\newcolumntype{s}{>{\hsize=.5\hsize}X}

\numberwithin{equation}{section}

\newtheorem{theorem}{Theorem}[section]
\newtheorem{lemma}{Lemma}[section]
\newtheorem{proposition}{Proposition}[section]
\newtheorem{assumption}{Assumption}[section]
\newtheorem{corollary}{Corollary}[section]
\newtheorem{innercustomgeneric}{\customgenericname}
\providecommand{\customgenericname}{}
\newcommand{\newcustomtheorem}[2]{%
  \newenvironment{#1}[1]
  {%
   \renewcommand\customgenericname{#2}%
   \renewcommand\theinnercustomgeneric{##1}%
   \innercustomgeneric
  }
  {\endinnercustomgeneric}
}
\newcustomtheorem{customdefinition}{Definition}
\newcustomtheorem{customdefinitions}{Definitions}
\newcustomtheorem{customtheorem}{Theorem}
\newcustomtheorem{customassumption}{Assumption}
\newcustomtheorem{customlemma}{Lemma}
\newcustomtheorem{customexample}{Example}
\theoremstyle{definition}

\newtheorem{remark}{Remark}[section]

\usepackage{enumitem}
\makeatletter
\newcommand{\mylabel}[2]{#2\def\@currentlabel{#2}\label{#1}}
\makeatother

\graphicspath{{./fig3/}}

\allowdisplaybreaks

\begin{document}

\setlength{\abovedisplayskip}{5pt}
\setlength{\belowdisplayskip}{5pt}
\setlength{\abovedisplayshortskip}{5pt}
\setlength{\belowdisplayshortskip}{5pt}
\hypersetup{colorlinks,breaklinks,urlcolor=blue,linkcolor=blue}

\title{\LARGE On regression-adjusted imputation estimators of the average treatment effect}

\author{Zhexiao Lin\thanks{Department of Statistics, University of California, Berkeley, CA 94720, USA; e-mail: {\tt zhexiaolin@berkeley.edu}}~~~and~
Fang Han\thanks{Department of Statistics, University of Washington, Seattle, WA 98195, USA; e-mail: {\tt fanghan@uw.edu}}
}

\date{}

\maketitle

\vspace{-1em}

\begin{abstract}
Imputing missing potential outcomes using an estimated regression function is a natural idea for estimating causal effects. In the literature, estimators that combine imputation and regression adjustments are believed to be comparable to augmented inverse probability weighting. Accordingly, people for a long time conjectured that such estimators, while avoiding directly constructing the weights, are also doubly robust \citep{imbens2004nonparametric,stuart2010matching}. Generalizing an earlier result of the authors \citep{lin2021estimation}, this paper formalizes this conjecture, showing that a large class of regression-adjusted imputation methods are indeed doubly robust for estimating the average treatment effect. In addition, they are provably semiparametrically efficient as long as both the density and regression models are correctly specified. Notable examples of imputation methods covered by our theory include kernel matching, (weighted) nearest neighbor matching, local linear matching, and (honest) random forests.
\end{abstract}


\section{Introduction}

The problem of estimating the average effect of a binary treatment on a scalar outcome under unconfoundedness and overlap conditions has had a long and rich history \citep{rosenbaum1983central,imbens2015causal}. While nowadays a large literature focuses on propensity score-based methods, alternatives that are based on regression \citep{heckman1997matching,heckman1998matching,heckman1998characterizing,hahn1998role,athey2016recursive,wager2018estimation} and matching \citep{rubin1973matching,abadie2006large,abadie2011bias} still receive persistent attention. 

Regression and matching methods 
relate causal inference to the imputation methods prevalent in the statistical missing value literature \citep{rubin2004multiple,tsiatis2006semiparametric,little2019statistical}. Indeed, as Guido Imbens and others (cf. \citet[Section IIIB]{imbens2004nonparametric} and \citet[Page 241]{abadie2006large}) have pointed out, both the regression and matching methods are intrinsically imputing the missing potential outcomes using, e.g., kernel matching, local linear matching, random forests, or the nearest neighbor matching. Accordingly, to be aligned with the missing value terminology, we call both of them the {\it imputation methods}.

Employing imputation methods alone can be either inefficient or lacking precision. This was discussions by \cite{robins1995semiparametric} in the missing value, \citet[Section IIID]{imbens2004nonparametric}  and \cite{abadie2006large} in the causal inference, and \cite{cassel1976some} and \cite{sarndal2003model} in the survey literature. It stimulates a surge in combining imputation methods with different types of adjustments --- including the celebrated augmented inverse probability weighted (AIPW) estimators \citep{robins1994estimation,scharfstein1999adjusting} as well as its much more recent cousin, the double machine learning estimators \citep{chernozhukov2018double}--- partly in order to encourage more efficient and robust estimators. 

This paper is interested in exploring the {\it double robustness} \citep{robins1994estimation,robins1997toward,scharfstein1999adjusting,bang2005doubly,kang2007demystifying} and {\it semiparametric efficiency} properties of the imputation methods when combined with {\it regression adjustments} for {\it correcting the bias}. While being proposed and studied in prominent works \citep{rubin1973use,abadie2011bias}, unlike its counterpart that integrates imputation with weighting --- e.g., propensity score \citep{robins1994estimation,hirano2003efficient} or covariate balancing \citep{chan2016globally,ben2021balancing} --- theoretical results on regression-adjusted imputation methods are extremely scarce. This may be partly explained by the fact that they are fully outcome model driven, and hence it was unclear which part is playing the role of propensity score weighting. 

More specifically, in the literature, people have been long time conjecturing that combining imputation with regression adjustments (for the purpose of bias correction) would yield doubly robust estimators. This was made explicit in, e.g., \citet[Section IIID]{imbens2004nonparametric} that ``the benefit associated with combining methods is made explicit in the notion developed by Robins and Ritov (1997) of double robustness'' as well as \citet[Section 5]{stuart2010matching} that  ``[matching and regression] have been shown to work best in combination... [t]his is similar to the idea of double robustness''. However, a mathematical formulation of double robustness for regression-adjusted imputation methods is still absent in the literature.

In addition to double robustness, statistical efficiency is vital for justifying any developed estimator. In a landmark paper, \cite{heckman1998matching} underpinned theoretical studies of (bias-uncorrected) imputation methods and showed that 
imputation based on covariate kernel matching yields a semiparametrically efficient estimator. Nevertheless,  \cite{heckman1998matching}'s result only focuses on estimating the average treatment effect on the treated (ATT). Later, \cite{abadie2006large,abadie2011bias} studied the limit theorems of NN matching for estimating both the ATT and the average treatment effect (ATE). However, the conveyed message therein is mixed, suggesting that NN matching-based imputation --- no matter bias correction is made or not --- is not semiparametrically efficient in estimating either the ATT or ATE. Except for the aforementioned two special cases, efficiency theory on (regression-adjusted) imputation methods is still largely lacking.

This paper aims to offer a general theory towards demystifying the efficiency and robustness properties of regression-adjusted imputation methods. For imputing the missing potential outcomes, we are concerned with a class of nonparametric regression methods called {\it linear smoothers} \citep{buja1989linear,fan2018local,wasserman2006all}, which include all the aforementioned examples (kernel matching, local linear matching, nearest neighbor matching, and random forests). Building on an earlier result of the authors that focuses on the nearest neighbor matching \citep{lin2021estimation}, the new theory shows:
\begin{itemize}
\item[(P1)] a linear smoother can implicitly give rise to a density ratio estimator;
\item[(P2)] imputation methods with regression adjustments in the form of \cite{rubin1973use} and \cite{abadie2011bias} constitute AIPW estimators;
\item[(P3)] these imputation methods are consistent as long as either the density model or the outcome model is correctly specified, and thus {\it doubly robust};
\item[(P4)] they further constitute asymptotically normal estimators of the ATE with the asymptotic variance attaining the semiparametric efficiency lower bound \citep{hahn1998role} if both the density and outcome models are correctly specified, and are thus {\it semiparametrically efficient};
\item[(P5)] the double machine learning \citep{chernozhukov2018double} versions of regression-adjusted imputations --- those that estimate the imputation function and the corrected bias via sample splitting and cross fitting --- can attain the properties in (P3) and (P4) while weakening some conditions.
\end{itemize}
Our results thus provide necessary theoretical support for using regression-adjusted imputation methods and establish them as useful alternatives to the weighting-based ones.

Notably speaking, the results of this paper are built on an earlier work of the authors \citep{lin2021estimation}, who established the double robustness and semiparametrical efficiency theory for \cite{abadie2011bias}'s NN matching-based ATE estimator by allowing the number of matches to diverge with the sample size. Their Lemma 5.1 reveals that \cite{abadie2011bias}'s bias-corrected NN matching estimator can be formulated as an AIPW one, which stimulates us to explore more cases. This leads to the general theory established in Section \ref{sec:general} and the study of more imputation methods elaborated on in Sections \ref{sec:example} and \ref{sec:RF}. Due to the richness of newly obtained results, we feel compelled to disseminate them to peers by writing a second manuscript.

\vspace{0.2cm}

\noindent {\bf Paper organization.} Section \ref{sec:prelim} introduces necessary notation, the preliminary setup, and those regression-adjusted imputation ATE estimators that will be analyzed in subsequent sections. Section \ref{sec:general} lays out our general theory, with examples provided in Sections \ref{sec:example} and \ref{sec:RF}. Specifically, Section \ref{sec:example} concerns imputation using kernel matching, weighted NN, and local linear matching while Section \ref{sec:RF} is focused on imputing the missing potential outcomes using random forests.

\section{Preliminary}\label{sec:prelim}

In the following, for any integers $n,d\ge 1$, we write $\zahl{n}:= \{1,2,\ldots,n\}$, and $\bR^d$ to represent the $d$-dimensional real space. A set consisting of distinct elements $x_1,\dots,x_n$ is written as either $\{x_1,\dots,x_n\}$ or $\{x_i\}_{i=1}^{n}$, and the corresponding sequence is denoted by $[x_1,\dots,x_n]$ or $[x_i]_{i=1}^{n}$. 


Consider $n$ observations, categorized to two groups, the treated and control, separately with $D_1,\ldots,D_n$ indexing the treatment statuses. More specifically, for each unit $i \in \zahl{n}$, we observe $D_i=1$ if in the treated group and $D_i=0$ if in the control group. Let $n_0:=\sum_{i=1}^n (1-D_i)$ and $n_1:=\sum_{i=1}^n D_i$ be the numbers of control and treated units, respectively. Adopting the Neyman-Rubin potential outcome framework \citep{neyman1923applications,rubin1974estimating}, the unit $i$ has two potential outcomes, $Y_i(1)$ and $Y_i(0)$, but we observe only one of them:
\[
    Y_i = \begin{cases}
        Y_i(0), & \mbox{ if } D_i=0,\\
        Y_i(1), & \mbox{ if } D_i=1.
    \end{cases}
\]
Let $X_i$ represent the pretreatment covariates of the $i$-th unit. 

The data we observe are $[(X_i,D_i,Y_i)]_{i=1}^n$, which are assumed to be independently drawn from the triple $(X,D,Y)$, where $D\in\{0,1\}$ is a binary variable, $X \in \bR^d$, and $Y \in \bR$. Our goal of interest is to estimate the following population ATE, 
\begin{align*}
    \tau := \E\Big[Y_i(1)-Y_i(0)\Big],
\end{align*}
based on $[(X_i, D_i, Y_i)]_{i=1}^n$.

As stated in the introduction section, this paper is interested in studying the imputation-based ATE estimators. To this end, we consider imputing the missing potential outcomes by regressing the data points in the opposite group against it:
\begin{align*}
    \hat{Y}_i^{\rm imp}(0) := \begin{cases}
        Y_i, & \mbox{ if } D_i=0,\\
        \displaystyle\sum_{j:D_j=0} w_{i\leftarrow j} Y_j, & \mbox{ if } D_i=1,     
    \end{cases}  
\end{align*}
and
\begin{align*}
    \hat{Y}_i^{\rm imp}(1) := \begin{cases}
   \displaystyle   \sum_{j:D_j=1} w_{i\leftarrow j} Y_j, & \mbox{ if } D_i=0,\\   
      Y_i, & \mbox{ if } D_i=1.
    \end{cases}    
\end{align*}
Here the $[w_{i\leftarrow j}]_{i,j}$ constitutes the {\it smoothing matrix}, where each entry $w_{i\leftarrow j}$ --- called the {smoothing parameter} --- is learnt from the covariates $X_i$ and those $X_j$'s in the opposite group, i.e., those with $D_j=1-D_i$. 
Nonparametric regressors taking the above form are called the {\it linear smoothers} \citep{buja1989linear}. Note that all imputation methods considered in Sections \ref{sec:example} and \ref{sec:RF}, including the kernel regression and local linear regression estimators \citep{heckman1997matching,heckman1998characterizing,heckman1998matching}, the (weighted) NN regression \citep{abadie2006large,abadie2011bias,lin2021estimation}, and the (honest) random forests \citep{athey2016recursive,wager2018estimation,athey2019estimating}, admit such a form. 

Unfortunately, imputing the missing potential outcomes alone is often not sufficient for attaining efficiency or even merely root-$n$ consistency. To remedy it, we are interested in correcting the bias via regression adjustments as proposed in \cite{rubin1973use} and \cite{abadie2011bias}. In detail, let's write 
\[
\hat{\mu}_0(x)~~~{\rm and}~~~\hat{\mu}_1(x) 
\]
to represent the mappings from $\bR^d$ to $\bR$ that estimate the conditional means of the outcomes
\[
\mu_0(x) := \E [Y \given X=x,D=0]~~ {\rm and}~~ \mu_1(x) := \E [Y \given X=x,D=1], 
\]
respectively. Of note, in the literature, $\hat{\mu}_0(x)$ and $\hat{\mu}_1(x)$ may differ from the regression imputation methods used in calculating $\hat{Y}_i^{\rm imp}(0)$'s and $\hat{Y}_i^{\rm imp}(1)$'s. For example, \cite{abadie2011bias} used NN regression to impute the missing potential outcomes, but series regressions to correct the bias.

We are then ready to define the regression-adjusted imputed values as
\begin{align*}
    \hat{Y}_i(0) := \begin{cases}
        Y_i, & \mbox{ if } D_i=0,\\
        \displaystyle \sum_{j:D_j=0} w_{i\leftarrow j} (Y_j + \hat{\mu}_0(X_i) - \hat{\mu}_0(X_j)), & \mbox{ if } D_i=1,     
    \end{cases}  
\end{align*}
and
\begin{align*}
    \hat{Y}_i(1) := \begin{cases}
    \displaystyle  \sum_{j:D_j=1} w_{i\leftarrow j} (Y_j + \hat{\mu}_1(X_i) - \hat{\mu}_1(X_j)), & \mbox{ if } D_i=0,\\   
      Y_i, & \mbox{ if } D_i=1.
    \end{cases}    
\end{align*}
The according regression-adjusted imputation-based ATE estimator is 
\begin{align*}
   \hat\tau_w  := \frac{1}{n} \sum_{i=1}^n \Big[\hat{Y}_i(1) -\hat{Y}_i(0)\Big].
\end{align*}

The estimator $\hat\tau_w$ has the appealing property of being fully outcome model driven, i.e.,  both the imputation and the bias correction steps are regression-based. It is conceptually easy to parse. The first goal of this paper is to show that $\hat\tau_w$, while avoiding directly modeling the propensity score, can be formulated as an AIPW one, and the regression imputation is intrinsically estimating the propensity score. The second goal of this paper is to establish a general theory, formulating conditions under which $\hat\tau_w$ is doubly robust and semiparametrically efficient.  Examples covered by our general theory shall occupy the rest two sections of this paper.

\section{The general theory}\label{sec:general}

This section lays out the general theory on the regression-adjusted imputation estimator $\hat\tau_w$. Recall the conditional mean estimators $\hat\mu_0$ and $\hat\mu_1$ introduced in the last section. Let the residuals from fitting the outcome models be
\[
\hat{R}_i := Y_i - \hat{\mu}_{D_i}(X_i), ~~i\in\zahl{n},
\]
and the estimator based on the outcome models be
\[
\hat{\tau}^{\rm reg}:= n^{-1} \sum_{i=1}^n \Big[\hat{\mu}_1(X_i) - \hat{\mu}_0(X_i)\Big].
\]

\subsection{A key lemma}

Results in Section \ref{sec:general} are all built on the following key lemma, which gives an AIPW formulation of the ATE estimator $\hat\tau_w$.

\begin{lemma}\label{lemma:mbc}  The regression-adjusted imputation estimator $\hat\tau_w$ can be rewritten as
\begin{align}\label{eq:mbc}
    \hat\tau_w  = \hat{\tau}^{\rm reg} + \frac{1}{n} \sum_{i=1,D_i=1}^n \Big(1 + \sum_{j:D_j=1-D_i} w_{j\leftarrow i}\Big) \hat{R}_i- \frac{1}{n} \sum_{i=1,D_i=0}^n \Big(1 + \sum_{j:D_j=1-D_i} w_{j\leftarrow i}\Big) \hat{R}_i  \notag\\
    + \frac{1}{n} \sum_{i=1}^n (2D_i-1) \Big(1 - \sum_{j:D_j=1-D_i} w_{i\leftarrow j} \Big) \hat{\mu}_{1-D_i}(X_i).
\end{align}
\end{lemma}

The sum of the first three terms in \eqref{eq:mbc} has the same form as an AIPW estimator that was studied in \cite{scharfstein1999adjusting} and \cite{bang2005doubly}, among many others.  The last is an additional bias term that was induced by those unnormalized $w_{i\leftarrow j}$'s such that
\[
\sum_{j:D_j=1-D_i} w_{i\leftarrow j} \ne 1.
\]
Accordingly, Equation \eqref{eq:mbc} favors a normalized smoothing matrix such that $\sum_{j:D_j=1-D_i} w_{i\leftarrow j}$ adds up to 1. This is an observation interestingly related to the classic arguments in nonparametric regressions; cf. \citet[Section 3.4]{fan2018local} and \citet[Remark 5.23]{wasserman2006all}. 

Note that the relation between regression-adjusted imputation and AIPW estimators was for the first time disclosed in \cite{lin2021estimation}, stated as Lemma 5.1 therein and with a focus on NN regression-based imputation. Lemma \ref{lemma:mbc}, on the other hand, delivers the general form that applies to an arbitrary linear smoother.

\subsection{Double robustness}

For presenting the general theory, let us first introduce some additional notation. In the sequel, for any two real sequences $\{a_n\}$ and $\{b_n\}$, we write $a_n = O(b_n)$ if $\lvert a_n \rvert / \lvert b_n \rvert $ is bounded and $a_n = o(b_n)$ if $\lvert a_n \rvert / \lvert b_n \rvert  \to 0$. We use $\stackrel{\sf d}{\longrightarrow}$ and $\stackrel{\sf p}{\longrightarrow}$ to denote convergence in distribution and in probability, respectively. For any sequence of random variables $[X_n]$, write $X_n = o_\P(1)$ if $X_n \stackrel{\sf p}{\longrightarrow} 0$ and $X_n = O_\P(1)$ if $X_n$ is bounded in probability. For any vector $x$, we use $\lVert x \rVert$ to denote its Euclidean norm. For any $0<p\le \infty$ and function $f$, let $\lVert f(Z) \rVert_p$, or simply $\lVert f \rVert_p$ if no confusion is possible, to represent $(\int \lvert f(\omega) \rvert^p \d \P_Z(\omega))^{1/p}$, where $\P_Z$ represents the law of a certain random variable $Z$. 

In the following, let $U_\omega := Y(\omega) - \mu_{\omega}(X)$ for $\omega \in \{0,1\}$ be the residuals of $Y(0)$ and $Y(1)$ projected on $X$ and let $\cS$ be the support of $X$. 
The first set of assumptions concerns the data generating distribution.

\begin{assumption}  \phantomsection \label{asp:dr} 
    \begin{enumerate}[itemsep=-.5ex,label=(\roman*)]
      \item\label{asp:dr-1} For almost all $x \in \cS$, $D$ is independent of $(Y(0),Y(1))$ conditional on $X=x$, and there exists some constant $\eta > 0$ such that $\eta < \P(D=1 \given X=x) < 1-\eta$.
      \item $[(X_i,D_i,Y_i)]_{i=1}^n$ are independent and identically distributed (i.i.d.) following the joint distribution of $(X,D,Y)$.
      \item $\E [U^2_\omega \given X=x] $ is uniformly bounded for almost all $x \in \cS$ and $\omega \in \{0,1\}$.
      \item $\E [\mu^2_\omega(X)]$ is bounded for $\omega \in \{0,1\}$.
    \end{enumerate}
\end{assumption}

Assumption~\ref{asp:dr}\ref{asp:dr-1} is the unconfoundedness and overlap assumptions commonly assumed in the literature. In particular, $e(x):=\P(D=1\given X=x)$ is the propensity score \citep{rosenbaum1983central}. The rest conditions in Assumption \ref{asp:dr} constitute standard i.i.d. assumptions and the moment assumptions on the residuals.

The next set of assumptions concerns the smoothing matrix used in the imputation step.

\begin{assumption} \phantomsection \label{asp:weight} 
    \begin{enumerate}[itemsep=-.5ex,label=(\roman*)]
        \item\label{asp:weight-1} Let $\pi: \zahl{n} \to \zahl{n}$ be any permutation. For samples $[(X_i,D_i,Y_i)]_{i=1}^n$ given, let $[w_{i\leftarrow j}]_{D_i+D_j=1}$ be the weights constructed by $[(X_i,D_i,Y_i)]_{i=1}^n$, and $[w^\pi_{i\leftarrow j}]_{D_i+D_j=1}$ be the weights constructed by $[(X_{\pi(i)},D_{\pi(i)},Y_{\pi(i)})]_{i=1}^n$. Then for any $i,j \in \zahl{n}$ such that $D_i+D_j=1$ and any permutation $\pi$, we have $w_{i\leftarrow j} = w^\pi_{\pi(i)\leftarrow \pi(j)}$.
        \item \label{asp:weight-2} The weights satisfy
        \begin{align*}
            \lim_{n \to \infty }\E \Big[ 
                \sum_{j:D_j=1-D_1} w_{1\leftarrow j} - 1
            \Big]^2 = 0.
        \end{align*}
    \end{enumerate}
\end{assumption}

Assumption~\ref{asp:weight} is to our knowledge new and is added for aiding the general theory to be presented later. There Assumption~\ref{asp:weight}\ref{asp:weight-1} ensures that the regression smoothing matrix is invariant to the feeding order of sample points, and Assumption~\ref{asp:weight}\ref{asp:weight-2} ensures that the bias term in Lemma~\ref{lemma:mbc} is asymptotically ignorable, which will be automatically satisfied if the smoother preserves the constant curve \citep[Remark 5.23]{wasserman2006all}. 

The next set of assumptions quantifies estimation accuracy of the ``density models''.

\begin{assumption} \phantomsection \label{asp:dr1}
    \begin{enumerate}[itemsep=-.5ex,label=(\roman*)]
        \item\label{asp:dr1,o} For $\omega \in \{0,1\}$, there exists a deterministic function $\bar{\mu}_\omega(\cdot):\bR^d \to \bR$  such that $\E [\bar{\mu}^2_\omega(X)]$ is bounded and the estimator $\hat{\mu}_\omega(x)$ satisfies
        \[
            \lVert \hat{\mu}_\omega - \bar{\mu}_\omega \rVert_\infty = o_\P(1).
        \]
        \item\label{asp:dr1,w} The weights satisfy
        \begin{align*}
            \lim_{n \to \infty }\E \Big[ 
                \sum_{j:D_j=1-D_1} w_{j\leftarrow 1} - \Big(D_1 \frac{1-e(X_1)}{e(X_1)} + (1-D_1) \frac{e(X_1)}{1-e(X_1)} \Big)
            \Big]^2 = 0.
        \end{align*}
    \end{enumerate}
  \end{assumption}
  
Assumption \ref{asp:dr1} allows for outcome model misspecification. Here Assumption \ref{asp:dr1}\ref{asp:dr1,o} is a regression misspecification assumption that is Assumption 5.3 in \cite{lin2021estimation}. Assumption \ref{asp:dr1}\ref{asp:dr1,w} is the key assumption that relates regression imputation/linear smoothers to the estimation of density ratios, in the form of $(1-e(x))/e(x)$ and its inverse; in Sections \ref{sec:example} and \ref{sec:RF} we will verify its validity for a variety of regression imputation methods.
  
 In parallel to Assumption \ref{asp:dr1}, the following conditions quantify estimation accuracy of the ``outcome models''. 
 
 \begin{assumption} \phantomsection \label{asp:dr2}
    \begin{enumerate}[itemsep=-.5ex,label=(\roman*)]
        \item\label{asp:dr2,o} For $\omega \in \{0,1\}$, the estimator $\hat{\mu}_\omega(x)$ satisfies
        \[
            \lVert \hat{\mu}_\omega - \mu_\omega \rVert_\infty = o_\P(1).
        \]
        \item\label{asp:dr2,w} The weights $[w_{1\leftarrow j}]_{D_j=1-D_1}$ are constructed by $[(X_i,D_i)]_{i=1}^n$ only without using the outcome information $[Y_i]_{i=1}^n$. 
        \item\label{asp:dr2,w2}  The weights satisfy
        \begin{align*}
            \E \Big[ \Big\lvert \sum_{j:D_j=1-D_1} w_{j\leftarrow 1} \Big\rvert \Big] = O(1).
        \end{align*}
    \end{enumerate}
\end{assumption}

Assumption \ref{asp:dr2} allows for density model misspecification. Here Assumption \ref{asp:dr2}\ref{asp:dr2,o} is Assumption 5.4 in \cite{lin2021estimation}; \cite{chen2015optimal} and \cite{chen2018optimal} verified such conditions for various nonparametric regressors. Assumption~\ref{asp:dr2}\ref{asp:dr2,w} ensures that the responses are not used in the construction of weights, and is satisfied by all examples to be introduced in Sections \ref{sec:example} and \ref{sec:RF}. This assumption is also related to the sample splitting procedures used in the context of double machine learning \citep{chernozhukov2018double} and honest random forests \citep{wager2018estimation}, shown to help avoid overfitting. Given Assumptions \ref{asp:dr} and \ref{asp:weight}, Assumption~\ref{asp:dr2}\ref{asp:dr2,w2} holds automatically as long as all the weights $w_{i\leftarrow j}$'s are nonnegative, or when Assumption~\ref{asp:dr1}\ref{asp:dr1,w} holds.
We would also like to highlight that Assumption~\ref{asp:dr2}\ref{asp:dr2,w2} is only needed for proving double robustness properties.

With the above assumptions, we are now ready to formalize the double robustness property of the regression-adjusted imputation estimator $\hat\tau_w$.

\begin{theorem}[Double robustness of $\hat\tau_w$]  \label{thm:dr} 
Suppose Assumptions \ref{asp:dr} and \ref{asp:weight} hold, and either Assumption~\ref{asp:dr1} or Assumption~\ref{asp:dr2} is true. We then have
\begin{align*}
\hat\tau_w  - \tau \stackrel{\sf p}{\longrightarrow} 0.
\end{align*}
\end{theorem}

Theorem \ref{thm:dr} unveils an interesting phenomenon that, although regression-adjusted imputation methods are {\it fully outcome model driven}, they are doubly robust and an intrinsic statistic coming from imputation captures the role of the propensity score; cf. Assumption \ref{asp:dr1}\ref{asp:dr1,w}. To the authors' knowledge, both the missing value and causal inference literature is largely silent about this phenomena. The most related result to Theorem \ref{thm:dr} resides in simple parametric models.

In detail, the fact that ordinary least square (OLS) is intrinsically a weighted estimator is very well known; cf. \cite{angrist2009mostly} and \cite{imbens2015matching}. In two very interesting papers, \citet[Section 3]{robins2007comment} and \cite{kline2011oaxaca} showed that OLS is also able to offer double robustness guarantee for estimating either a population mean with incomplete data or the ATT. This was developed more sophistically in a recent work of \cite{chattopadhyay2021implied} and other interesting research along this line includes \cite{guo2021generalized} and \cite{cohen2020no}. In the high level, they all bear a similar flavor to Theorem \ref{thm:dr} that a regression/imputation approach, without designing a set of weights (propensity score-based or not) on purpose, automatically satisfies the double robustness property. The difference with ours, on the other hand, is self-explanatory.


\subsection{Semiparametric efficiency}

This section establishes the semiparametric efficiency theory of $\hat\tau_w$. To this end, it appears that we have to put more assumptions on the moments of $U_\omega$, the regression adjustments $\hat\mu_w(\cdot)$, and the smoothing parameters $w_{i\leftarrow j}$'s.

\begin{assumption}  \phantomsection \label{asp:se1} 
\begin{enumerate}[itemsep=-.5ex,label=(\roman*)]
    \item $\E [U^2_\omega \given X=x]$ is uniformly bounded away from zero for almost all $x \in \cS$ and $\omega \in \{0,1\}$.
    \item There exists some constant $\kappa>0$ such that $\E [\lvert U_\omega \rvert ^{2+\kappa} \given X=x]$ is uniformly bounded for almost all $x \in \cS$ and $\omega \in \{0,1\}$.
\end{enumerate}
\end{assumption}

\begin{assumption}  \phantomsection \label{asp:se2} 
There exists a positive integer $k$ such that
\begin{enumerate}[itemsep=-.5ex,label=(\roman*)]
    \item\label{asp:se2,o1} $\max_{t \in \Lambda_{k}} \lVert \partial^t \mu_{\omega} \rVert_\infty$ is bounded, where for any positive integer $k$, $\Lambda_k$ is the set of all $d$-dimensional vectors of nonnegative integers $t=(t_1,\ldots,t_d)$ such that $\sum_{i=1}^d t_i = k$;
    \item\label{asp:se2,o2} For $\omega \in \{0,1\}$, the estimator $\hat{\mu}_\omega(x)$ satisfies
    \[
    \max_{t \in \Lambda_{k}} \lVert \partial^t \hat{\mu}_{\omega} \rVert_\infty = O_\P(1)~~~{\rm and}~~~ 
    \max_{t \in \Lambda_\ell} \lVert \partial^t \hat{\mu}_{\omega} - \partial^t \mu_{\omega} \rVert_\infty = O_\P(n^{-\gamma_\ell}) ~~\mbox{\rm for all}~~ \ell \in \zahl{k-1},
    \]
    with some constants $\gamma_\ell$'s for $\ell=1,2,\ldots,k-1$;
    \item\label{asp:se2,d} The discrepancy satisfies
    \begin{align*}
        & \E \Big[ \sum_{j:D_j=1-D_1} \lvert w_{1\leftarrow j} \rvert \cdot \lVert X_j - X_1 \rVert^k \Big] = o(n^{-1/2}),\\
        & \E \Big[ \sum_{j:D_j=1-D_1} \lvert w_{1\leftarrow j} \rvert \cdot \lVert X_j - X_1 \rVert^\ell \Big] = o(n^{-1/2+ \gamma_\ell} )  ~\mbox{\rm for all}~ \ell \in \zahl{k-1};
    \end{align*}
    \item\label{asp:se2,w1} The weights satisfy
    \begin{align*}
        \E \Big[ 
            \sum_{j:D_j=1-D_1} w_{1\leftarrow j} - 1
        \Big]^2 = o(n^{-1}).
    \end{align*}
\end{enumerate}
\end{assumption}

Assumptions~\ref{asp:se1} and \ref{asp:se2}\ref{asp:se2,o1}-\ref{asp:se2,o2} are Assumptions 5.6 and 5.7 in \cite{lin2021estimation}; check \cite{abadie2011bias} and \cite{chen2018optimal} for results on verifying these requirements. Assumption~\ref{asp:se2}\ref{asp:se2,d} assumes that the linear smoother used in imputing the missing values is a local method, i.e., it will put larger values on the closer ones and smaller values on the farther ones. Lastly, Assumption \ref{asp:se2}\ref{asp:se2,w1}, as a counterpart of Assumption \ref{asp:weight}\ref{asp:weight-2}, requires the bias term in \eqref{eq:mbc} to be root-$n$ ignorable.

We then introduce the semiparametric efficiency lower bound for estimating the ATE  \citep{hahn1998role},
\begin{align}\label{eq:hahn}
  \sigma^2:= \E \Big[\mu_1(X) - \mu_0(X) + \frac{D(Y-\mu_1(X))}{e(X)} - \frac{(1-D)(Y-\mu_0(X))}{1-e(X)} - \tau \Big]^2.
\end{align}
The following theorem then shows that the asymptotic variance of $\hat\tau_w$ can attain $\sigma^2$.

\begin{theorem}[Semiparametric efficiency of $\hat\tau_w$]\label{thm:mbc}
Suppose Assumptions~\ref{asp:dr}-\ref{asp:se2} hold. We then have
    \begin{align*}
        \sqrt{n} (\hat\tau_w  - \tau) \stackrel{\sf d}{\longrightarrow} N(0,\sigma^2).
    \end{align*}
    In addition, the variance estimator 
  \begin{align*}
  \hat{\sigma}^2:= \frac{1}{n} \sum_{i=1}^n \Big[\hat{\mu}_1(X_i) - \hat{\mu}_0(X_i) + (2D_i-1)\Big(1 + \sum_{j:D_j=1-D_i} w_{j\leftarrow i}\Big) \hat{R}_i - \hat\tau_w \Big]^2
\end{align*}
is a consistent estimator of  $\sigma^2$ in \eqref{eq:hahn}. 
\end{theorem}

\subsection{Double machine learning}

Assumptions \ref{asp:se2}\ref{asp:se2,o1}-\ref{asp:se2,o2} are arguably strong regularity conditions for the outcome model $\mu_\omega(\cdot)$. Partly in order to alleviate such requirements, \cite{chernozhukov2018double} introduced the idea of double machine learning via sample splitting and cross fitting. Similar ideas have also been studied in nonparametric statistics; cf. \citet{bickel1982adaptive}, \cite{efromovich1996nonparametric}, and \cite{zheng2010asymptotic}.  In the following, let's introduce $\tilde{\tau}_{w,N}$ as a counterpart of $\hat{\tau}_w$ based on \citet[Definition 3.1]{chernozhukov2018double}. 

In detail, let $N \ge 2$ represent a fixed number of partitions. For presentation simplicity and also without much loss of generality, assume $n$ to be divisible by $N$. Let $[I_k]_{k=1}^N$ be an $N$-fold random partition of $\zahl{n}$, with each of size equal to $n' = n/N$. For each $k \in \zahl{N}$ and $\omega \in \{0,1\}$, construct $\hat{\mu}_{\omega,k}(\cdot)$ using data $[(X_i,D_i,Y_i)]_{i=1,i \notin I_k}^n$. Similarly, for regression imputation, we impute each unit's value by regressing it against all units in the opposite group outside the $k$-th fold. More specifically, we calculate the smoothing matrix entries as follows: for any $i,j \in \zahl{n}$ with $D_i+D_j=1,i \in I_k,j \notin I_k$, let $w_{j\leftarrow i,k}$ be the weights constructed using data $(X_i,D_i,Y_i) \cup [(X_j,D_j,Y_j)]_{j=1,j \notin I_k}^n$.

We are then ready to define the double machine learning version of $\hat\tau_w$ as follows:
\begin{align*}
   \widecheck{\tau}_{w,k} :=& \frac{1}{n'} \sum_{i=1,i \in I_k}^n \Big[\hat{\mu}_{1,k}(X_i) - \hat{\mu}_{0,k}(X_i)\Big] \\
  &+ \frac{1}{n'} \sum_{i=1,i \in I_k}^n (2D_i-1)\Big(1 + \sum_{j:D_j=1-D_i,j \notin I_k} w_{j\leftarrow i,k}\Big) \Big(Y_i - \hat{\mu}_{D_i,k}(X_i)\Big)
\end{align*}
and
\begin{align*}
    \tilde{\tau}_{w,N} := \frac{1}{N} \sum_{k=1}^N \widecheck{\tau}_{w,k}.
\end{align*}

For establishing the efficiency theory of $\tilde{\tau}_{w,N} $, the following two sets of assumptions are needed.

\begin{assumption}  \phantomsection \label{asp:dml1} 
    \begin{enumerate}[itemsep=-.5ex,label=(\roman*)]
      \item $\E [U^2_\omega]$ is bounded away from zero for $\omega \in \{0,1\}$.
      \item There exists some constant $\kappa>0$ such that $\E [\lvert Y \rvert^{2+\kappa}]$ is bounded.
    \end{enumerate}
\end{assumption}

\begin{assumption}\label{asp:dml2}
    There exist two positive integers $1 \le p_1,p_2 \le \infty$ with $p_1^{-1} + p_2^{-1} = 1$, two positive real-valued sequences $[r_1] = [r_1]_n, [r_2] = [r_2]_n$ with $r_1r_2 = o(n^{-1/2})$ such that
    \begin{enumerate}[itemsep=-.5ex,label=(\roman*)]
        \item\label{asp:dml2,o} for $\omega \in \{0,1\}$, the estimator $\hat{\mu}_\omega(x)$ satisfies
        \[
            \lVert \tilde{\mu}_\omega - \mu_\omega \rVert_{p_1} = O_\P(r_1);
        \]
        \item\label{asp:dml2,w} the weights $[w_{i\leftarrow j}]_{D_i+D_j=1}$ satisfy
        \begin{align*}
            & \Big\{\E \Big[ \Big\lvert \sum_{j:D_j=1-D_1} w_{j\leftarrow 1} - \Big(D_1 \frac{1-e(X_1)}{e(X_1)} + (1-D_1) \frac{e(X_1)}{1-e(X_1)} \Big) \Big\rvert^{p_2} \Big] \Big\}^{1/p_2} = O(r_2),\\
            {\rm and}\quad & \E \Big[ \sum_{j:D_j=1-D_1} w_{j\leftarrow 1} \Big]^\kappa = O(1), ~~ {\rm for~any~} \kappa>0.
        \end{align*}
    \end{enumerate}
\end{assumption}

We are now ready to introduce the general theory on the double machine learning-based regression-adjusted imputation estimators.

\begin{theorem}  \phantomsection \label{thm:dml} 
    \begin{enumerate}[itemsep=-.5ex,label=(\roman*)]
     \item\label{thm:dml1} (Double robustness of $\tilde{\tau}_{w,N}$) Under the same conditions as those in Theorem~\ref{thm:dr}, we have
      \begin{align*}
        \tilde{\tau}_{w,N} - \tau \stackrel{\sf p}{\longrightarrow} 0.
      \end{align*}
      \item\label{thm:dml2} (Semiparametric efficiency of $\tilde{\tau}_{w,N}$) Under Assumptions~\ref{asp:dr}-\ref{asp:dr2} and \ref{asp:dml1}-\ref{asp:dml2}, we have
      \begin{align*}
       \sqrt{n} (\tilde{\tau}_{w,N} - \tau) \stackrel{\sf d}{\longrightarrow} N(0,\sigma^2).
      \end{align*}
      In addition, the variance estimator 
  \begin{align*}
  \hat{\sigma}^2_N:= \frac{1}{n} \sum_{i=1}^n \Big[\hat{\mu}_1(X_i) - \hat{\mu}_0(X_i) + (2D_i-1)\Big(1 + \sum_{j:D_j=1-D_i} w_{j\leftarrow i}\Big) \hat{R}_i - \tilde{\tau}_{w,N} \Big]^2
\end{align*}
is a consistent estimator for  $\sigma^2$ in \eqref{eq:hahn}. 
    \end{enumerate}
\end{theorem}

\section{Examples}\label{sec:example}


This section aims to provide examples so to put the general theory introduced in Section \ref{sec:general} on a solid ground. In the sequel, write $\ind(\cdot)$ to represent the indicator function and $a_n \asymp b_n$ if both $a_n = O(b_n)$ and $b_n = O(a_n)$ holds. For any matrix $A$, we use $\lvert A \rvert$ and $\lVert A \rVert_2$ to denote its determinant and spectral norm. For any set $\cS$, let ${\rm diam}(\cS):=\sup_{x,y\in \cS}\lVert x-y \rVert$ be its diameter. 

\subsection{Kernel matching}

We first consider the kernel matching that has been advocated in various settings \citep{heckman1997matching, heckman1998characterizing, heckman1998matching, frolich2004finite, frolich2005matching, huber2013performance}. It leverages the local constant regression (Nadaraya–Watson estimator) to impute the missing values  \citep{nadaraya1964estimating, watson1964smooth}.

More specifically, let $H = H_n \in \bR^{d \times d}$ be the {\it bandwidth matrix}  and  $K(\cdot): \bR^d \to \bR$ be the {\it multivariate kernel function} on $\bR^d$. For any $x \in \bR^d$, define 
\[
K_H(x) := \lvert H \rvert^{-1/2} K(H^{-1/2}x). 
\]
For any $i,j \in \zahl{n}$ such that $D_i+D_j=1$, one can then verify that the weight $w_{i\leftarrow j}$ corresponding to kernel matching is 
\begin{align*}
    w_{i\leftarrow j}:= \frac{K_H(X_i-X_j)}{\sum_{k:D_k=1-D_i} K_H(X_i-X_k)}.
\end{align*}

Denote the corresponding kernel matching estimator using the above smoothing matrix  as well as the double machine learning version of it by
\[
\hat{\tau}_{\rm K}~~{\rm and}~~ \tilde{\tau}_{{\rm K},N}. 
\]
Assumptions in Section \ref{sec:general} can then be shown to hold under the following sufficient conditions.


\begin{assumption}\phantomsection  \label{asp:kernel} Assume that (i)
$H$ is symmetric and positive definite, and (ii) $K$ constitutes a multivariate symmetric density function.
\end{assumption}

\begin{assumption} \phantomsection \label{asp:kernel,dr}
    \begin{enumerate}[itemsep=-.5ex,label=(\roman*)]
      \item The density of $X$ is bounded and bounded away from zero. The densities of $X \given D=1$ and $X \given D=0$ are continuous almost everywhere.
      \item $K$ is bounded with a compact support such that $\lVert H^{1/2} \rVert_2 \to 0$ and $n \lvert H^{1/2} \rvert \to \infty$.
    \end{enumerate}
\end{assumption}

\begin{assumption} \phantomsection \label{asp:kernel,se} There exists a positive integer $k$ such that  
\begin{itemize}
\item[(i)] Assumptions~\ref{asp:se2}\ref{asp:se2,o1},\ref{asp:se2,o2} hold; 
\item[(ii)] we further have $\lVert H^{1/2} \rVert_2^k = o(n^{-1/2})$ and $\lVert H^{1/2} \rVert_2^\ell = o(n^{-1/2+ \gamma_\ell})$ for all $\ell \in \zahl{k-1}$. 
\end{itemize}
\end{assumption}

\begin{assumption}[double machine learning] \phantomsection \label{asp:kernel,dml}
    \begin{enumerate}[itemsep=-.5ex,label=(\roman*)]
        \item The densities of $X \given D=1$ and $X \given D=0$ are Lipchitz on $\cS$. The diameter and the surface area (Hausdorff measure, \citet[Section 3.3]{evans2018measure}) of $\cS$ are bounded. 
        \item There exist two positive real-valued sequences $[r_1] = [r_1]_n, [r_2] = [r_2]_n$ with $r_1r_2 = o(n^{-1/2})$ such that 
        \[
        \lVert \tilde{\mu}_\omega - \mu_\omega \rVert_\infty = O_\P(r_1)~~ {\rm for} ~~\omega \in \{0,1\}, 
        \]
        and $(n \lvert H^{1/2} \rvert)^{-1/2} + \lVert H^{1/2} \rVert_2 = O(r_2)$.
    \end{enumerate}
\end{assumption}

Assumption~\ref{asp:kernel,dr} is standard for establishing consistency of the Nadaraya-Watson estimator. Assumption~\ref{asp:kernel,se} ensures that the discrepancy level in Assumption~\ref{asp:se2} is small. The regularity condition on the support and the smoothness condition on the density function are standard in nonparametric statistics \cite[Section 2]{MR2724359}. 

\begin{remark}
A specific common choice of $H$ is $h_n^2 I_d$, where $I_d$ is the $d$-dimensional identity matrix. The bandwidth selection condition in Assumption~\ref{asp:kernel,dr} then reduces to
\[
h_n \to 0\quad {\rm and}\quad  nh_n^d \to \infty, 
\]
and Assumption~\ref{asp:kernel,se} reduces to 
\[
h_n/n^{-1/(2k)} \to 0\quad {\rm and}\quad h_n/n^{(-1/2+\gamma_\ell)/\ell} \to 0 ~~{\rm for}~~ \ell \in \zahl{k-1}, 
\]
suggesting that the bandwidth cannot be too large; this echos the NN matching case where the number of NNs incorporated also has to be controlled (Theorem 5.2 in \cite{lin2021estimation}). The convergence rate in Assumption~\ref{asp:kernel,dml} reduces to $(nh^d)^{-1/2}+h$, and is the minimax rate of the density estimation over Lipchitz class $n^{-1/(2+d)}$ \cite[Section 2]{MR2724359} by taking $h_n \asymp n^{-1/(2+d)}$.
\end{remark}

The following theorem then verifies the general conditions presented in Section \ref{sec:general} when kernel matching is used for imputing the missing potential outcomes.

\begin{theorem}\label{thm:kernel}
Assume Assumptions \ref{asp:dr} and \ref{asp:kernel}  hold. We then have the following four are true.
 \begin{enumerate}[itemsep=-.5ex,label=(\roman*)]
\item\label{thm:kernel1} Assumptions  \ref{asp:weight}, \ref{asp:dr2}\ref{asp:dr2,w}\ref{asp:dr2,w2}, and \ref{asp:se2}\ref{asp:se2,w1} hold;
\item\label{thm:kernel2} Under Assumption~\ref{asp:kernel,dr}, Assumption~\ref{asp:dr1}\ref{asp:dr1,w} holds; 
\item\label{thm:kernel3} Under Assumptions~\ref{asp:kernel,dr} and \ref{asp:kernel,se}, Assumption~\ref{asp:se2} holds; 
\item\label{thm:kernel4} Under Assumptions~\ref{asp:kernel,dr} and \ref{asp:kernel,dml}, Assumption~\ref{asp:dml2} holds with $p_1,p_2$ chosen to be $\infty$ and $1$.
\end{enumerate}
\end{theorem}

Theorem \ref{thm:kernel} directly yields the following corollary, which establishes the double robustness and semiparametric efficiency properties of $\hat\tau_{\rm K}$ and $\tilde{\tau}_{{\rm K},N} $.


\begin{corollary}  \phantomsection \label{crl:kernel} 
    \begin{enumerate}[itemsep=-.5ex,label=(\roman*)]
     \item(Double robustness of $ \hat{\tau}_{{\rm K}}$) Suppose Assumptions~\ref{asp:dr} and \ref{asp:kernel} hold and either Assumptions~\ref{asp:dr1}\ref{asp:dr1,o}, \ref{asp:kernel,dr} or Assumption~\ref{asp:dr2}\ref{asp:dr2,o} is true. We then have
     \begin{align*}
        \hat{\tau}_{{\rm K}} - \tau \stackrel{\sf p}{\longrightarrow} 0.
      \end{align*}
      \item(Semiparametric efficiency of $\hat{\tau}_{{\rm K}}$) Under Assumptions~\ref{asp:dr}, \ref{asp:dr1}\ref{asp:dr1,o}, \ref{asp:dr2}\ref{asp:dr2,o}, \ref{asp:se1}, \ref{asp:kernel}-\ref{asp:kernel,se}, we have
      \begin{align*}
       \sqrt{n} (\hat{\tau}_{{\rm K}} - \tau) \stackrel{\sf d}{\longrightarrow} N(0,\sigma^2).
      \end{align*}
     \item(Double robustness of $\tilde{\tau}_{{\rm K},N}$) Suppose Assumptions~\ref{asp:dr} and \ref{asp:kernel} hold and either Assumptions~\ref{asp:dr1}\ref{asp:dr1,o}, \ref{asp:kernel,dr} or Assumption~\ref{asp:dr2}\ref{asp:dr2,o} is true. We then have
      \begin{align*}
        \tilde{\tau}_{{\rm K},N} - \tau \stackrel{\sf p}{\longrightarrow} 0.
      \end{align*}
      \item(Semiparametric efficiency of $\tilde{\tau}_{{\rm K},N}$) Under Assumptions~\ref{asp:dr}, \ref{asp:dr1}\ref{asp:dr1,o}, \ref{asp:dr2}\ref{asp:dr2,o}, \ref{asp:dml1}, \ref{asp:kernel}, \ref{asp:kernel,dr}, \ref{asp:kernel,dml}, it holds true that
      \begin{align*}
        \sqrt{n} (\tilde{\tau}_{{\rm K},N} - \tau) \stackrel{\sf d}{\longrightarrow} N(0,\sigma^2).
      \end{align*}
    \end{enumerate}
\end{corollary}

\subsection{Weighted NNs}

NN matching \citep{rubin1973matching,abadie2006large,stuart2010matching} is a popular imputation method that imputes the missing potential outcomes by a NN regression. In the nonparametric statistics literature, it is well known that NN regression, which assigns equal weights to all NNs, can be less efficient. This motivates the development of weighted NNs as useful alternatives to NN regression for boosting statistical efficiency \citep{royall1966class,samworth2012optimal}. The theoretical properties of WNNs for conducting nonparametric regression have been studied in, among many others, \cite{stone1977consistent}, \cite{samworth2012optimal}, and \citet[Chapter 5]{biau2015lectures}.

Consider the $M$-NN that restricts attention to the first $M$ NNs. The weighted nearest neighbor  (WNN) regression imputes the missing potential outcomes using a set of preassigned weights $[\gamma_{M,m}]_{m=1}^M$ satisfying  
\[
\gamma_{M,m}\geq 0~~~{\rm and}~~~\sum_{m=1}^M \gamma_{M,m} = 1. 
\]
The corresponding imputed outcome values are then 
\begin{align*}
    \hat{Y}_i^{\rm WNN}(0) := \begin{cases}
        Y_i, & \mbox{ if } D_i=0,\\
        \sum_{m=1}^M \gamma_{M,m} Y_{j_m(i)}, & \mbox{ if } D_i=1,     
    \end{cases}
~~{\rm and}~~    \hat{Y}_i^{\rm WNN}(1) := \begin{cases}
        \sum_{m=1}^M \gamma_{M,m} Y_{j_m(i)}, & \mbox{ if } D_i=0,\\   
        Y_i, & \mbox{ if } D_i=1.
    \end{cases}    
\end{align*}
Here $j_m(i)$ represents the index of $m$-th nearest neighbor (NN) of $X_i$ in $\{X_j:D_j=1-D_i\}_{j=1}^n$, i.e., the index $j \in \zahl{n}$ such that $D_j=1-D_i$ and 
\[
    \sum_{\ell=1, D_\ell=1-D_i}^n \ind\Big(\lVert X_\ell -X_i \rVert \le \lVert X_j - X_i \rVert\Big) = m.
\]

For any $i,j \in \zahl{n}$ with $D_i+D_j=1$, the corresponding weight $w_{i\leftarrow j}$ is then defined to be
\begin{align*}
    w_{i\leftarrow j}:= \sum_{m=1}^M \gamma_{M,m} \ind(j_m(i)=j)
\end{align*}
and the WNN-based ATE estimator and its double machine learning version are then denoted by
\[
\hat{\tau}_{\rm WNN}~~~{\rm and}~~~\tilde{\tau}_{{\rm WNN},N}.
\]
Notably speaking, when $\gamma_{M,m}=1/M$ for all $m \in \zahl{M}$, $\hat\tau_{\rm WNN}$ reduces to the standard bias-corrected NN matching that was studied in \cite{abadie2006large,abadie2011bias} and \cite{lin2021estimation}.

\begin{assumption} \phantomsection \label{asp:wnn,dr}
    \begin{enumerate}[itemsep=-.5ex,label=(\roman*)]
      \item\label{asp:wnn,dr,1} The density of $X$ is bounded and bounded away from zero. The densities of $X \given D=1$ and $X \given D=0$ are continuous almost everywhere. The diameters and the surface area (Hausdorff measure, \citet[Section 3.3]{evans2018measure}) of $\cS$ are bounded. There exists a constant $a \in (0,1)$ such that for any $\delta \in (0,{\rm diam}(\cS)]$ and $z \in \cS$,
      \[
        \lambda(B_{z,\delta} \cap S) \ge a \lambda(B_{z,\delta}),
      \]
      where $B_{z,\delta}$ represents the closed ball in $\bR^d$ with center at $z$ and radius $\delta$.
      \item\label{asp:wnn,dr,2} Assume $M\log n/n \to 0$, $\sum_{m=1}^M \gamma_{M,m}^2 \to 0$, and
      \begin{align*}
        \limsup_{n \to \infty} n \int_0^\infty \Big[\sum_{m=1}^M \gamma_{M,m}^2 \P\Big( U_{(m-1)} \le t \le U_{(m)} \Big) \Big]^{1/2} \d t \le 1,
      \end{align*}
      where $(U_{(1)},\ldots,U_{(M)})$ are the first $M$ order statistics of $n$ i.i.d random variables from the uniform distribution on $[0,1]$.
    \end{enumerate}
\end{assumption}

\begin{assumption} \phantomsection \label{asp:wnn,se}
    Assume that there exists a positive integer $k$ such that 
    \begin{itemize}
\item[(i)]    Assumptions~\ref{asp:se2}\ref{asp:se2,o1},\ref{asp:se2,o2} hold; 
\item[(ii)] we further have 
    \[
    \sum_{m=1}^M \gamma_{M,m} (m/n)^{k/d} = o(n^{-1/2})~~~{\rm and}~~~\sum_{m=1}^M \gamma_{M,m} (m/n)^{\ell/d} = o(n^{-1/2+ \gamma_\ell})
    \]
     for all $\ell \in \zahl{k-1}$. 
     \end{itemize}
\end{assumption}

\begin{assumption}[double machine learning] \phantomsection \label{asp:wnn,dml}
    \begin{enumerate}[itemsep=-.5ex,label=(\roman*)]
        \item The densities of $X \given D=1$ and $X \given D=0$ are Lipchitz on $\cS$.
        \item Assume $M/\log n \to \infty$ and the weights satisfy
        \begin{align*}
            M\max_{m \in \zahl{M}}\gamma_{M,m} = O(1)
        \end{align*}
        and there exists a positive sequence $[r_3]=[r_3]_n$ such that
        \begin{align*}
            n \int_0^\infty \Big[\sum_{m=1}^M \Big(\gamma_{M,m} - \frac{1}{M}\Big)^2 \P\Big( U_{(m-1)} \le t \le U_{(m)} \Big) \Big]^{1/2} \d t = O(r_3).
          \end{align*}
        Further assume that there exist two positive sequences $[r_1] = [r_1]_n, [r_2] = [r_2]_n$ with $r_1r_2 = o(n^{-1/2})$ such that $\lVert \tilde{\mu}_\omega - \mu_\omega \rVert_\infty = O_\P(r_1)$ for $\omega \in \{0,1\}$, and 
        \[
        (M/n)^{1/d} + M^{-1/2} + \Big(\sum_{m=1}^M \gamma_{M,m}^2\Big)^{1/2} + r_3 = O(r_2).
        \]
    \end{enumerate}
\end{assumption}

\begin{remark}
    Assumption \ref{asp:wnn,dr}\ref{asp:wnn,dr,1} is Assumption 4.1 in \cite{lin2021estimation}. When $\gamma_{M,m}=1/M$ for all $m \in \zahl{M}$, Assumption~\ref{asp:wnn,dr}\ref{asp:wnn,dr,2} is satisfied as long as $M \to \infty$, and the inequality in Assumption~\ref{asp:wnn,dr}\ref{asp:wnn,dr,2} can be automatically satisfied by using Chernoff's inequality, which recovers Theorems 5.1 and 5.2 in \cite{lin2021estimation}. 
    Similar discussions also apply to Assumption \ref{asp:wnn,dml}.
\end{remark}

\begin{theorem}\label{thm:wnn}
Assume Assumption \ref{asp:dr} holds. We then have the following four are true.
 \begin{enumerate}[itemsep=-.5ex,label=(\roman*)]
 \item\label{thm:wnn1} Assumptions \ref{asp:weight}, \ref{asp:dr2}\ref{asp:dr2,w}\ref{asp:dr2,w2}, and \ref{asp:se2}\ref{asp:se2,w1} hold;
 \item\label{thm:wnn2} Under Assumption~\ref{asp:wnn,dr}, Assumption~\ref{asp:dr1}\ref{asp:dr1,w} holds; 
 \item\label{thm:wnn3} Under Assumptions~\ref{asp:wnn,dr} and \ref{asp:wnn,se},  Assumption~\ref{asp:se2} holds; 
 \item\label{thm:wnn4} Under Assumptions~\ref{asp:wnn,dr} and \ref{asp:wnn,dml}, Assumption~\ref{asp:dml2} holds with $p_1,p_2$ chosen to be $\infty$ and $1$.
 \end{enumerate}
\end{theorem}


\begin{corollary}  \phantomsection \label{crl,wnn} 
    \begin{enumerate}[itemsep=-.5ex,label=(\roman*)]
     \item(Double robustness of $\hat\tau_{\rm WNN}$) Suppose Assumption~\ref{asp:dr} holds, and either Assumptions~\ref{asp:dr1}\ref{asp:dr1,o} and \ref{asp:wnn,dr} or Assumption~\ref{asp:dr2}\ref{asp:dr2,o} is true. We then have
     \begin{align*}
        \hat{\tau}_{{\rm WNN}} - \tau \stackrel{\sf p}{\longrightarrow} 0.
      \end{align*}
      \item(Semiparametric efficiency of $\hat\tau_{\rm WNN}$) Under Assumptions~\ref{asp:dr}, \ref{asp:dr1}\ref{asp:dr1,o}, \ref{asp:dr2}\ref{asp:dr2,o}, \ref{asp:se1}, \ref{asp:wnn,dr}, \ref{asp:wnn,se}, we have
      \begin{align*}
       \sqrt{n} (\hat{\tau}_{{\rm WNN}} - \tau) \stackrel{\sf d}{\longrightarrow} N(0,\sigma^2).
      \end{align*}
     \item(Double robustness of $\tilde{\tau}_{{\rm WNN},N} $) Suppose Assumption~\ref{asp:dr} holds, and either Assumptions~\ref{asp:dr1}\ref{asp:dr1,o} and \ref{asp:wnn,dr} or Assumption~\ref{asp:dr2}\ref{asp:dr2,o} is true. We then have
      \begin{align*}
        \tilde{\tau}_{{\rm WNN},N} - \tau \stackrel{\sf p}{\longrightarrow} 0.
      \end{align*}
      \item(Semiparametric efficiency of $\tilde{\tau}_{{\rm WNN},N}$) Under Assumptions~\ref{asp:dr}, \ref{asp:dr1}\ref{asp:dr1,o}, \ref{asp:dr2}\ref{asp:dr2,o}, \ref{asp:dml1}, \ref{asp:wnn,dr}, \ref{asp:wnn,dml}, it holds true that
      \begin{align*}
        \sqrt{n} (\tilde{\tau}_{{\rm WNN},N} - \tau) \stackrel{\sf d}{\longrightarrow} N(0,\sigma^2).
      \end{align*}
    \end{enumerate}
\end{corollary}

\subsection{Local linear matching}

In nonparametric statistics, local linear regression has been a prominent alternative to local constant regression, proving to be more efficient than the latter, especially along the boundary \citep{fan1992design,fan1993local}. This approach has also been heavily used in ATE estimation for imputing the missing potential outcomes, which is often called ``local linear matching''; cf. \cite{heckman1997matching}, \cite{heckman1998characterizing}, \cite{heckman1998matching}, and \cite{frolich2005matching}. 

In detail, for any unit $i \in \zahl{n}$, local linear matching uses the local linear regression \citep{fan2018local} to minimize
\begin{align}\label{eq:llr}
    \sum_{j:D_j=1-D_i} \Big[Y_j - \beta_0 - \beta^\top (X_j - X_i) \Big]^2 K_H(X_j-X_i),
\end{align}
and then $Y_i(1-D_i)$ is imputed by the solution to the above objective function.

Let $\mB_i \in \bR^{n_{1-D_i} \times (1+d)}$ be the design matrix with the row corresponding to unit $j$ with $D_i + D_j=1$ to be $(1,(X_j-X_i)^\top):=b_{ij}^\top$. Let $\mW_i \in \bR^{n_{1-D_i} \times n_{1-D_i}}$ be the diagonal matrix with the diagonal element corresponding to unit $j$ with $D_i + D_j=1$ to be $K_H(X_j-X_i)$. It is well known that the solution to the minimization problem \eqref{eq:llr} is:
\[
    \hat{Y}_i^{\rm LL}(1-D_i) = e_1^\top (\mB_i^\top \mW_i \mB_i)^{-1} \mB_i^\top \mW_i \mY_{1-D_i},
\]
where $e_1 \in \bR^{1+d}$ is the vector with the first element to be 1 and all the rest 0 and $\mY_{\omega} \in \bR^{n_\omega}$ for $\omega \in \{0,1\}$ represents the vector containing entries $Y_j$'s with $D_j=\omega$. For any $i,j \in \zahl{n}$ with $D_i+D_j=1$, one could calculate the corresponding weight $w_{i\leftarrow j}$ as
\begin{align*}
    w_{i\leftarrow j}:= e_1^\top (\mB_i^\top \mW_i \mB_i)^{-1} b_{ij} K_H(X_j-X_i).
\end{align*}
Denote the corresponding local linear estimator and its double machine learning version by 
\[
\hat{\tau}_{\rm LL}~~ {\rm and}~~ \tilde{\tau}_{{\rm LL,}N}.
\]

For analyzing $\hat\tau_{\rm LL}$ and $\tilde{\tau}_{{\rm LL,}N}$, we need to regulate the kernel function $K(\cdot)$ a little bit more. The following assumption is standard in multivariate local linear regression literature (cf. Assumption A1 in \cite{ruppert1994multivariate}). It can be satisfied by many kernels, e.g., the spherically symmetric kernels and product kernels based on symmetric univariate kernels \citep[Chapter 4]{simonoff2012smoothing}.

\begin{assumption} \phantomsection \label{asp:lp}
    Assume $\int z K(z) \d z = 0$ and $\int z z^\top K(z) \d z = \mu_2(K) I_d$ with $\mu_2(K)>0$ as a positive real-valued constant that captures the second-order property of $K$. 
\end{assumption}

\begin{theorem}\label{thm:lp}
Assume Assumptions \ref{asp:dr} and \ref{asp:kernel} hold. We then have the following four are true.
 \begin{enumerate}[itemsep=-.5ex,label=(\roman*)]
 \item\label{thm:lp1} Assumptions~\ref{asp:weight}, \ref{asp:dr2}\ref{asp:dr2,w}, \ref{asp:se2}\ref{asp:se2,w1} hold; if $K(\cdot)$ is bounded with a compact support and is bounded away from zero, then Assumption \ref{asp:dr2}\ref{asp:dr2,w2} holds;
 \item\label{thm:lp2} Under Assumptions~\ref{asp:kernel,dr}, \ref{asp:lp}, Assumption~\ref{asp:dr1}\ref{asp:dr1,w} holds; 
 \item\label{thm:lp3} Under Assumptions~\ref{asp:kernel,dr}, \ref{asp:kernel,se}, \ref{asp:lp}, Assumption~\ref{asp:se2} holds; 
 \item\label{thm:lp4} Under Assumptions~\ref{asp:kernel,dr}, \ref{asp:kernel,dml}, \ref{asp:lp}, Assumption~\ref{asp:dml2} holds with $p_1,p_2$ chosen to be $\infty$ and $1$.
 \end{enumerate}
\end{theorem}


\begin{corollary}  \phantomsection \label{crl:lp} 
    \begin{enumerate}[itemsep=-.5ex,label=(\roman*)]
     \item(Double robustness of $\hat{\tau}_{\rm LL}$) Suppose Assumptions~\ref{asp:dr} and \ref{asp:kernel} hold and either Assumptions~\ref{asp:dr1}\ref{asp:dr1,o}, \ref{asp:kernel,dr}, \ref{asp:lp} hold or Assumption~\ref{asp:dr2}\ref{asp:dr2,o} is true and $K(\cdot)$ is bounded with a compact support and is bounded away from zero. We then have
      \begin{align*}
        \hat{\tau}_{\rm LL} - \tau \stackrel{\sf p}{\longrightarrow} 0.
      \end{align*}
      \item(Semiparametric efficiency of $\hat{\tau}_{\rm LL}$) Under Assumptions~\ref{asp:dr}, \ref{asp:dr1}\ref{asp:dr1,o}, \ref{asp:dr2}\ref{asp:dr2,o}, \ref{asp:se1}, \ref{asp:kernel}, \ref{asp:kernel,dr}, \ref{asp:kernel,se}, \ref{asp:lp}, we have
      \begin{align*}
       \sqrt{n} (\hat{\tau}_{\rm LL} - \tau) \stackrel{\sf d}{\longrightarrow} N(0,\sigma^2).
      \end{align*}
     \item(Double robustness of $\tilde{\tau}_{{\rm LL,}N}$) Suppose Assumptions~\ref{asp:dr} and \ref{asp:kernel} hold and either Assumptions~\ref{asp:dr1}\ref{asp:dr1,o}, \ref{asp:kernel,dr}, \ref{asp:lp} hold or Assumption~\ref{asp:dr2}\ref{asp:dr2,o} is true and $K(\cdot)$ is bounded with a compact support and is bounded away from zero. We then have
      \begin{align*}
        \tilde{\tau}_{{\rm LL,}N} - \tau \stackrel{\sf p}{\longrightarrow} 0.
      \end{align*}
      \item(Semiparametric efficiency of $\tilde{\tau}_{{\rm LL,}N}$) Under Assumptions~\ref{asp:dr}, \ref{asp:dr1}\ref{asp:dr1,o}, \ref{asp:dr2}\ref{asp:dr2,o}, \ref{asp:dml1}, \ref{asp:kernel}, \ref{asp:kernel,dr}, \ref{asp:kernel,dml}, \ref{asp:lp}, it holds true that
      \begin{align*}
       \sqrt{n} (\tilde{\tau}_{{\rm LL,}N} - \tau) \stackrel{\sf d}{\longrightarrow} N(0,\sigma^2).
      \end{align*}
    \end{enumerate}
\end{corollary}

\section{Random forests}\label{sec:RF}


This section studies random forests as an imputation method to estimate the ATE. Since being invented by Leo Breiman \citep{breiman2001random}, random forests have proven to be practically powerful in conducting regression and classification tasks; cf. the survey of \cite{biau2016random}. However, it is not until very recent that some major advances were made towards using random forests for inferring causal effect \citep{athey2019machine}; notable works include \cite{hill2011bayesian}, \cite{athey2016recursive}, \cite{athey2019machine}, \cite{athey2019generalized}, among many others. Our results in this section aim to contribute to this growing literature, while being focused on the original regression-adjusted imputation estimator without doing sample splitting and cross fitting.

\subsection{Set up}

In the sequel, for any set $A$ with finite elements, let $\lvert A \rvert$ stand for its cardinality. For introducing the random forests to impute the missing potential outcomes, some additional notation is needed and we also adopt some common terms used in the random forests and regression trees literature \citep{breiman2017classification}. 

Let's first introduce the causal tree. Let $T^1$ be a generic {\it tree} built on the treated group $\{(X_i,Y_i)\}_{i=1,D_i=1}^n$ and $T^0$ be another generic tree built on the control group $\{(X_i,Y_i)\}_{i=1,D_i=0}^n$. The two trees $T^1$ and $T^0$ accordingly partition the covariates space $\cS\subset \bR^d$ into a set of leaves $L^1$ and $L^0$, respectively. For any test point $x \in \bR^d$, let $L^1(x)$ and $L^0(x)$ be the {\it leaves} of $T^1$ and $T^0$ containing $x$. One could then impute the missing potential outcomes as follows:
\begin{align*}
    \hat{Y}_i^{\rm Tree}(0) := \begin{cases}
        Y_i, & \mbox{ if } D_i=0,\\
       \displaystyle \Big(\Big\lvert \Big\{j:D_j=0,X_j \in L^0(X_i)\Big\} \Big\rvert\Big)^{-1}\sum_{j:D_j=0,X_j \in L^0(X_i)} Y_j, & \mbox{ if } D_i=1,    
    \end{cases}
\end{align*}
and
\begin{align*}
    \hat{Y}_i^{\rm Tree}(1) := \begin{cases}
       \displaystyle \Big(\Big\lvert \Big\{j:D_j=1,X_j \in L^1(X_i)\Big\} \Big\rvert\Big)^{-1}\sum_{j:D_j=1,X_j \in L^1(X_i)} Y_j, & \mbox{ if } D_i=0,\\
        Y_i, & \mbox{ if } D_i=1. 
    \end{cases}
\end{align*}

To aggregate many individual causal trees into a {\it causal forest}, we consider {\it subsampling}. In detail, let $B$ be the number of trees and $s$ be the subsample size, which for presentation simplicity are assumed to be identical for the two groups of samples. In the $b$-th round, for building the tree, we sample without replacement the following two size-$s$ subsets
\[
\cI^1_b  ~~~{\rm and}~~~ \cI^0_b
\]
from $\{i:D_i=1\}$ and $\{i:D_i=0\}$, respectively. Of note, for any $b,b'\in\zahl{B}$ and any $\omega\in\{0,1\}$, $\cI^\omega_b$ and $\cI^\omega_{b'}$ could have a nonempty overlap. 

Let $T^1_b$ be the  tree built on the data $\{(X_i,Y_i)\}_{i \in \cI^1_b}$ and $T^0_b$ be the tree built on the data $\{(X_i,Y_i)\}_{i \in \cI^0_b}$. All trees are assumed to be constructed using the same base learner. For any test point $x$, let $L^1_b(x)$ and $L^0_b(x)$ be the leaves of $T^1_b$ and $T^0_b$ that contain $x$. The according random forest then imputes the missing potential outcomes as follows:
\begin{align*}
    \hat{Y}_i^{\rm RF}(0) := \begin{cases}
        Y_i, & \mbox{ if } D_i=0,\\
       \displaystyle \frac1B \sum_{b=1}^B \Big[\Big(\Big\lvert \Big\{j \in \cI^0_b:X_j \in L^0_b(X_i)\Big\} \Big\rvert\Big)^{-1}\sum_{j \in \cI^0_b:X_j \in L^0_b(X_i)} Y_j\Big], & \mbox{ if } D_i=1,    
    \end{cases}
\end{align*}
and
\begin{align*}
    \hat{Y}_i^{\rm RF}(1) := \begin{cases}
   \displaystyle     \frac1B \sum_{b=1}^B \Big[\Big(\Big\lvert \Big\{j \in \cI^1_b:X_j \in L^1_b(X_i)\Big\} \Big\rvert\Big)^{-1}\sum_{j \in \cI^1_b:X_j \in L^1_b(X_i)} Y_j\Big], & \mbox{ if } D_i=0,\\
        Y_i, & \mbox{ if } D_i=1. 
    \end{cases}
\end{align*}

It is well known that random forests, formulable as a special type of weighted NN regressions, constitute linear smoothers \citep{lin2006random,biau2010layered}. In particular, for any $i,j \in \zahl{n}$ with $D_i+D_j=1$, one could verify that the weight $w_{i\leftarrow j}$ corresponding to the above random forests imputation method is
\begin{align*}
    w_{i\leftarrow j}:= \frac1B \sum_{b=1}^B \frac{\ind\Big(j \in \cI^{1-D_i}_b:X_j \in L^{1-D_i}_b(X_i)\Big)}{\Big\lvert \Big\{k \in \cI^{1-D_i}_b:X_k \in L^{1-D_i}_b(X_i)\Big\} \Big\rvert}.
\end{align*}
We then denote the corresponding regression-adjusted random forest-based imputation ATE estimator by $\hat{\tau}_{\rm RF}$.

\subsection{Inference theory}

In order to verify the conditions in Section \ref{sec:general}, the following assumptions are needed and were intentionally designed to be general.

\begin{assumption} \phantomsection \label{asp:rf,dr}
\begin{enumerate}[itemsep=-.5ex,label=(\roman*)]
    \item\label{asp:rf,dr,d} The density of $X$ is bounded and bounded away from zero, the densities of $X \given D=1$ and $X \given D=0$ are continuous almost everywhere, and the support $\cS$ is compact.
    \item\label{asp:rf,dr,t} We assume $s=s_n=O(n^{1/2})$ and $n/B = O(1)$. In addition, assume that for the tree $T$ built on $s$ i.i.d. sampled points from $(X,Y)\given D=1$ or $(X,Y)\given D=0$ with leaves $\{L_t\}_{t \ge 1}$, it holds true that
    \begin{align}\label{assump:RF1}
        \lim_{n \to \infty} \E \Big[\Big(\min_{t\ge1}\Big\lvert L_t \Big\rvert\Big)^{-1}\Big] = 0~~~{\rm and}~~~ \lim_{n \to \infty} \int_S \E\Big[{\rm diam}\Big(L_t(x) \cap \cS\Big)\Big] \d x = 0,
    \end{align}
    where $\lvert L_t \rvert$ represents the number of samples in the leaf $L_t$ for $t\ge1$ and $L_t(x)$ stands for the leaf that contains $x$.
    \end{enumerate}
\end{assumption}

\begin{assumption} \phantomsection \label{asp:rf,honest}
The tree is honest, that is, the tree does not use the responses $Y_i$'s to choose the place to split.
\end{assumption}

\begin{assumption} \phantomsection \label{asp:rf,se}
There exists a positive integer $k$ such that 
\begin{enumerate}[itemsep=-.5ex,label=(\roman*)]
\item Assumptions~\ref{asp:se2}\ref{asp:se2,o1}, \ref{asp:se2,o2} hold; 
\item\label{asp:rf,se,2} for a tree $T$ built on $s$ independent observations from $(X,Y)\given D=1$ or $(X,Y)\given D=0$ with leaves $\{L_t\}_{t \ge 1}$, it holds true that 
\begin{align*}
&\int_S \E\Big[{\rm diam}^k\Big(L_t(x) \cap \cS\Big)\Big] \d x = o(n^{-1/2})~~\\
{\rm and}~~&\int_S \E\Big[{\rm diam}^\ell\Big(L_t(x) \cap \cS\Big)\Big] \d x = o(n^{-1/2+ \gamma_\ell}) \text{ for all }\ell \in \zahl{k-1}.
\end{align*}
\end{enumerate}
\end{assumption}

\begin{remark}\label{remark:RF1}
Assumption~\ref{asp:rf,dr} requires $s_n=O(n^{1/2})$. In the literature, \cite{mentch2016quantifying} required a similar condition, $s_n = o(n^{1/2})$, for establishing asymptotic normality of random forests. 
\cite{wager2018estimation} allowed $s_n \asymp n^\beta$ for some $\beta$ that can be close to 1 (cf. Equation (14) therein); we cannot recover their setting due to the extra difficulty in estimating the ATE compared to estimating the conditional ATE. Assumption~\ref{asp:rf,dr} also requires $n/B=O(1)$, which echoes \cite{wager2014confidence}, where the authors recommended a similar $B \asymp n$ condition. Conditions similar to the two leaf size conditions in \eqref{assump:RF1} have been discussed in multiple places. There the first requirement in \eqref{assump:RF1} regulates the smallest size of the terminal leaves, which echoes the discussions in \citet[Section 3]{lin2006random}. The second requirement in \eqref{assump:RF1} is very related to \citet[Lemma 1]{wager2018estimation}; we defer more discussions on it as well as those on Assumption \ref{asp:rf,se}\ref{asp:rf,se,2} to Lemma \ref{lemma:rf,dist} and Proposition \ref{prop:rf,dist} ahead.
\end{remark}


\begin{remark}
The ``honesty'' condition, Assumption~\ref{asp:rf,honest}, corresponds to Definition 2 in \cite{wager2018estimation}. This condition is usually achieved by implementing sample splitting as was suggested and also analyzed in \cite{wager2018estimation}. It is also satisfied by a variety of alternatives to Breiman's original random forests, including the centered forest \citep{biau2008consistency,scornet2016asymptotics} and the purely uniform random forests \citep{genuer2012variance}. Theoretical analysis of the trees constructed using the responses in the same training data is believed to be much more involved, but was managed in several impressive works including \cite{scornet2015consistency}, \cite{chi2020asymptotic}, and \cite{kulowski2022}. Unfortunately, our analysis hinges on a control of the leaf sizes that is seemingly hard to pursue without Assumption \ref{asp:rf,honest}. 
\end{remark}

Under the above assumptions, we are then ready to present our main theory on $\hat\tau_{\rm RF}$. Note that, in the following, Theorem \ref{thm:rf}\ref{thm:rf-dr} also gives rise to a consistent random forests-based density ratio estimator, which can be of independent interest.

\begin{theorem}\label{thm:rf}
Assume Assumption~\ref{asp:dr} holds. We then have the following four are true.
 \begin{enumerate}[itemsep=-.5ex,label=(\roman*)]
\item\label{thm:rf1} Assumptions~\ref{asp:weight}, \ref{asp:dr2}\ref{asp:dr2,w2}, \ref{asp:se2}\ref{asp:se2,w1} hold. 
\item\label{thm:rf-dr} Under Assumptions~\ref{asp:rf,dr} and \ref{asp:rf,honest},  Assumption~\ref{asp:dr1}\ref{asp:dr1,w} holds. 
\item\label{thm:rf2} Under Assumption~\ref{asp:rf,honest}, Assumption \ref{asp:dr2}\ref{asp:dr2,w} holds. 
\item\label{thm:rf3} Under Assumptions~\ref{asp:rf,dr}-\ref{asp:rf,se}, Assumption~\ref{asp:se2} holds.
\end{enumerate}
\end{theorem}

\begin{corollary}  \phantomsection \label{crl,rf} 
    \begin{enumerate}[itemsep=-.5ex,label=(\roman*)]
     \item(Double robustness of $\hat\tau_{\rm RF}$) Suppose Assumption~\ref{asp:dr} holds and either Assumptions \ref{asp:dr1}\ref{asp:dr1,o}, \ref{asp:rf,dr}, \ref{asp:rf,honest} or Assumptions~\ref{asp:dr2}\ref{asp:dr2,o}, \ref{asp:rf,honest} hold. We then have
     \begin{align*}
        \hat{\tau}_{{\rm RF}} - \tau \stackrel{\sf p}{\longrightarrow} 0.
      \end{align*}
      \item(Semiparametric efficiency of $\hat\tau_{\rm RF}$) Under Assumptions~\ref{asp:dr}, \ref{asp:dr1}\ref{asp:dr1,o}, \ref{asp:dr2}\ref{asp:dr2,o}, \ref{asp:se1}, \ref{asp:rf,dr}-\ref{asp:rf,se}, we have
      \begin{align*}
       \sqrt{n} (\hat{\tau}_{{\rm RF}} - \tau) \stackrel{\sf d}{\longrightarrow} N(0,\sigma^2).
      \end{align*}
    \end{enumerate}
\end{corollary}

\subsection{Balanced and regular random forests}

The goal here is to decipher the second part of \eqref{assump:RF1} and  Assumption \ref{asp:rf,se}\ref{asp:rf,se,2}; cf. the discussions in Remark \ref{remark:RF1}. To this end, we leverage the technical proofs of \citet[Lemma 1]{wager2018estimation} and \citet[Lemma 2]{meinshausen2006quantile}, and provide the convergence rates of the diameters of leaves for some particular trees. 

To this end, we introduce the following regularity conditions on the tree growing patter.

\begin{assumption}\label{def:tree}
    We consider the following type of trees:
    \begin{enumerate}[itemsep=-.5ex,label=(\roman*)]
        \item The tree is $\phi$-balanced, i.e., for each terminal leaf, the proportion of splits along the $j$-th axis for each $j\in\zahl{d}$ is lower bounded by $\phi/d$ for some $\phi \in (0,1)$ and the splitting directions (i.e., picking which feature to split) are independent of the data;
        \item The tree is $(\alpha,\theta)$-regular for some $\alpha \in (0,0.5]$ and some positive integer $\theta$, i.e., at each step of growing the tree, the split leaves at least $\alpha$ of the samples on each side of the split, and the terminal leaves are all of size in $[\theta,\lfloor \theta/\alpha \rfloor]$, where $\lfloor 
        \cdot \rfloor$ is the floor function.
    \end{enumerate}
\end{assumption}

Notably speaking, Assumption 3 in \cite{meinshausen2006quantile} and Definitions 3 and 4 in \cite{wager2018estimation} considered regular and random-split conditions that are similar to Assumption \ref{def:tree}.  In practice, Assumption \ref{def:tree} can always be satisfied by controlling how tree grows in the implementation.

For those trees that satisfy Assumption~\ref{def:tree}, we have the following lemma, which controls arbitrary finite moment of the diameter of the terminal leaves.

\begin{lemma}\label{lemma:rf,dist}
    Let $\epsilon \in (0,1)$, $p \in \zahl{d}$, $\cS = [0,1]^d$, and $T$ be a tree constructed based on $s$ i.i.d. observations from the uniform distribution on $\cS$. As long as $T$ is $(\alpha,\theta)$-regular and $\phi$-balanced, we have for any $x \in \cS$ and any positive integer $k$,
    \begin{align*}
        \E\Big[{\rm diam}_p^k\Big(L_t(x) \cap \cS\Big)\Big] \le \Big(\frac{s}{\alpha^{-1}\theta}\Big)^{k \frac{\log(1-(1-\epsilon)\alpha)}{\log(\alpha^{-1})} \frac{\phi}{d} }+ \frac{\log(s/\theta)}{\log(\alpha^{-1})} \exp\Big[-\theta\alpha\Big(\log\Big(\frac{1}{1-\epsilon}\Big)-\epsilon\Big)\Big],
    \end{align*}
    where ${\rm diam}_p(\cdot)$ stands for the diameter along the $p$-th axis.
\end{lemma}

Lemma~\ref{lemma:rf,dist} then yields sufficient conditions guaranteeing the validity of the second part of \eqref{assump:RF1} and  Assumption \ref{asp:rf,se}\ref{asp:rf,se,2}.

\begin{proposition}[Sufficient conditions on the leaf sizes]\label{prop:rf,dist}
Assume $\cS$ to be a compact subset of $\bR^d$ and $T$ to be a tree constructed based on $s$ i.i.d. observations following a distribution with density bounded and bounded away from zero on $\cS$. Assume further that $T$ is both regular and balanced. We then have, if $s/\theta \to \infty$ and $\log \log (s/\theta)/\theta \to 0$, 
    \begin{align}\label{eq:RF-db}
        \lim_{n \to \infty} \int_S \E\Big[{\rm diam}\Big(L_t(x) \cap \cS\Big)\Big] \d x = 0.
    \end{align}
If it further holds that $(s/\theta)/n^\epsilon \to \infty$ for some $\epsilon>0$ and $\theta/\log n \to \infty$, we then have, for any sufficiently large $k$,
    \begin{align}\label{eq:RF-se}
        \int_S \E\Big[{\rm diam}^k\Big(L_t(x) \cap \cS\Big)\Big] \d x = o(n^{-1/2}).
    \end{align}
\end{proposition}

Of note, in Proposition~\ref{prop:rf,dist} the requirements about $s$ and $\theta$ are much weaker for double robustness (corresponding to \eqref{eq:RF-db}) than for semiparametric efficiency (corresponding to \eqref{eq:RF-se}).

\begin{remark}
Lemma~\ref{lemma:rf,dist} is key to our analysis and is a stronger version of Lemma 1 in \cite{wager2018estimation}. In detail, Lemma 1 in \cite{wager2018estimation} or the proof of Theorem 3 therein can imply that the $k$-th moment of the diameter will always be dominated by 
\[
(s/\theta)^{-0.5[\log((1-\alpha)^{-1})/\log(\alpha^{-1})](\phi/d)}, 
\]
which, however, can not be faster than $n^{-1/2}$ for any positive integer $k$. In contrast, Lemma~\ref{lemma:rf,dist} establishes that we can reach the order $o(n^{-1/2})$ by taking $k$ large enough. This is viable by replacing the random-split condition in \cite{wager2018estimation} with Assumption \ref{def:tree}.
\end{remark}

\begin{remark}
It is worth noting that Lemma~\ref{lemma:rf,dist} and Proposition~\ref{prop:rf,dist} do not require the tree to be honest. This is in line with Lemma 2 in \cite{meinshausen2006quantile} for quantile regression tree using the responses and Lemma 1 in \cite{wager2018estimation} without assuming honesty. It indicates that the results in Lemma~\ref{lemma:rf,dist} and Proposition~\ref{prop:rf,dist} can be applied to more general random forests, e.g., the tree based on CART criteria \citep{breiman2017classification} with consistency analyzed in \cite{scornet2015consistency}. However, the ``regular'' and ``random-split'' conditions enforced in Assumption \ref{def:tree} seem inevitable to our analysis. Later, we require honesty for the double robustness and semiparametric efficiency of $\hat\tau_{\rm RF}$. 
\end{remark}

\section*{Acknowledgement}

We thank helpful discussions with Peng Ding, Kevin Guo, and Elizabeth Stuart on the matching procedure, and Yingying Fan on the random forest.

{
\bibliographystyle{apalike}
\bibliography{AMS}
}

\newpage{}

\appendix






\section{Proofs of the main results}\label{sec:main-proof}

{\bf Additional notation.} For any integer $n$, let $n!$ be the factorial of $n$. We use $\mD,\mX,\mX_0,\mX_1$ to denote $[D_i]_{i=1}^n, [X_i]_{i=1}^n, [X_i]_{i=1,D_i=0}^n, [X_i]_{i=1,D_i=1}^n$, respectively. For any $a,b \in \bR$, write $a \vee b = \max\{a,b\}$ and $a \wedge b = \min\{a,b\}$. For any two real sequences $\{a_n\}$and $\{b_n\}$, write $a_n \lesssim b_n$ (or equivalently, $b_n \gtrsim a_n$) if $a_n = O(b_n)$.

\subsection{Proof of Lemma~\ref{lemma:mbc}}

\begin{proof}[Proof of Lemma~\ref{lemma:mbc}]
From the definitions of $\hat{Y}_i(0)$, $\hat{Y}_i(1)$ for $i \in \zahl{n}$ and $\hat\tau_w$,
\begin{align*}
    &\hat\tau_w = \frac{1}{n} \sum_{i=1}^n \Big[\hat{Y}_i(1) -\hat{Y}_i(0)\Big] \\
    =& \frac{1}{n} \sum_{i=1,D_i = 1}^n \Big[Y_i - \sum_{j:D_j=0} w_{i\leftarrow j} (Y_j + \hat{\mu}_0(X_i) - \hat{\mu}_0(X_j)) \Big] \\
    & + \frac{1}{n} \sum_{i=1,D_i = 0}^n \Big[\sum_{j:D_j=1} w_{i\leftarrow j} (Y_j + \hat{\mu}_1(X_i) - \hat{\mu}_1(X_j)) - Y_i \Big]\\
    =& \frac{1}{n} \sum_{i=1,D_i = 1}^n \Big[\hat{R}_i + \hat{\mu}_1(X_i) - \hat{\mu}_0(X_i) - \sum_{j:D_j=0} w_{i\leftarrow j} \hat{R}_j + \Big(1 - \sum_{j:D_j=0} w_{i\leftarrow j} \Big) \hat{\mu}_0(X_i) \Big] \\
    & + \frac{1}{n} \sum_{i=1,D_i = 0}^n \Big[- \hat{R}_i + \hat{\mu}_1(X_i) - \hat{\mu}_0(X_i) + \sum_{j:D_j=1} w_{i\leftarrow j} \hat{R}_j - \Big(1 - \sum_{j:D_j=1} w_{i\leftarrow j} \Big) \hat{\mu}_1(X_i) \Big]\\
    =& \frac{1}{n} \sum_{i=1}^n \Big[\hat{\mu}_1(X_i) - \hat{\mu}_0(X_i)\Big] + \frac{1}{n} \Big[ \sum_{i=1,D_i = 1}^n \Big(1 + \sum_{j:D_j=0} w_{j\leftarrow i} \Big) \hat{R}_i - \sum_{i=1,D_i = 0}^n \Big(1 + \sum_{j:D_j=1} w_{j\leftarrow i} \Big) \hat{R}_i \Big]\\
    & + \frac{1}{n} \Big[ \sum_{i=1,D_i = 1}^n \Big(1 - \sum_{j:D_j=0} w_{i\leftarrow j} \Big) \hat{\mu}_0(X_i) - \sum_{i=1,D_i = 0}^n \Big(1 - \sum_{j:D_j=1} w_{i\leftarrow j} \Big) \hat{\mu}_1(X_i) \Big].
\end{align*}
This completes the proof.
\end{proof}

\subsection{Proof of Theorem~\ref{thm:dr}}

\begin{proof}[Proof of Theorem~\ref{thm:dr}]

{\bf Part I.} Suppose the propensity score model is correct, i.e., Assumption~\ref{asp:dr1} holds. For any $i \in \zahl{n}$, let $\bar{R}_i := Y_i - \bar{\mu}_{D_i}(X_i)$. By Lemma~\ref{lemma:mbc},
\begin{align*}
    & \hat\tau_w  = \hat{\tau}^{\rm reg} + \frac{1}{n} \sum_{i=1}^n (2D_i-1)\Big(1 + \sum_{j:D_j=1-D_i} w_{j\leftarrow i}\Big) \hat{R}_i + \frac{1}{n} \sum_{i=1}^n (2D_i-1) \Big(1 - \sum_{j:D_j=1-D_i} w_{i\leftarrow j} \Big) \hat{\mu}_{1-D_i}(X_i)\\
    = & \frac{1}{n} \sum_{i=1}^n \Big[\hat{\mu}_1(X_i) - \bar{\mu}_1(X_i)\Big] - \frac{1}{n} \sum_{i=1}^n \Big[\hat{\mu}_0(X_i) - \bar{\mu}_0(X_i)\Big]\\
    &+  \frac{1}{n} \Big[ \sum_{i=1}^n D_i \Big(1 + \sum_{D_j=1-D_i} w_{j\leftarrow i} \Big) \Big(\bar{\mu}_1(X_i) - \hat{\mu}_1(X_i)\Big) - \sum_{i=1}^n (1-D_i)\Big(1 + \sum_{D_j=1-D_i} w_{j\leftarrow i} \Big) \Big(\bar{\mu}_0(X_i) - \hat{\mu}_0(X_i)\Big) \Big]\\
    &+  \frac{1}{n} \Big[ \sum_{i=1}^n D_i \Big(1 + \sum_{D_j=1-D_i} w_{j\leftarrow i} - \frac{1}{e(X_i)}\Big) \bar{R}_i - \sum_{i=1}^n (1-D_i)\Big(1 + \sum_{D_j=1-D_i} w_{j\leftarrow i} - \frac{1}{1-e(X_i)}\Big) \bar{R}_i \Big]\\
    & + \frac{1}{n} \sum_{i=1}^n D_i \Big(1 - \sum_{D_j=1-D_i} w_{i\leftarrow j} \Big) \hat{\mu}_0(X_i) - \frac{1}{n} \sum_{i=1}^n (1-D_i) \Big(1 - \sum_{D_j=1-D_i} w_{i\leftarrow j} \Big) \hat{\mu}_1(X_i)\\
    & + \frac{1}{n} \Big[ \sum_{i=1}^n \Big(1 - \frac{D_i}{e(X_i)}\Big) \bar{\mu}_1(X_i) - \sum_{i=1}^n \Big(1 - \frac{1-D_i}{1-e(X_i)}\Big) \bar{\mu}_0(X_i) \Big]\\
    & + \frac{1}{n} \Big[ \sum_{i=1}^n \frac{D_i}{e(X_i)} Y_i - \sum_{i=1}^n \frac{1-D_i}{1-e(X_i)} Y_i \Big].
    \yestag\label{eq:dml1}
\end{align*}

For each pair of terms, we only establish the first half part under treatment, and the second half under control can be established in the same way.

For the first term in \eqref{eq:dml1},
\[
    \Big\lvert \frac{1}{n} \sum_{i=1}^n \Big[\hat{\mu}_1(X_i) - \bar{\mu}_1(X_i)\Big] \Big\rvert \le \lVert \hat{\mu}_1 - \bar{\mu}_1 \rVert_\infty = o_\P(1).
\]
Then
\begin{align}\label{eq:dml11}
    \frac{1}{n} \sum_{i=1}^n \Big[\hat{\mu}_1(X_i) - \bar{\mu}_1(X_i)\Big] - \frac{1}{n} \sum_{i=1}^n \Big[\hat{\mu}_0(X_i) - \bar{\mu}_0(X_i)\Big] = o_\P(1).
\end{align}

For the second term in \eqref{eq:dml1},
\begin{align*}
    & \Big\lvert \frac{1}{n} \sum_{i=1}^n D_i \Big(1 + \sum_{D_j=1-D_i} w_{j\leftarrow i} \Big) \Big(\bar{\mu}_1(X_i) - \hat{\mu}_1(X_i)\Big) \Big\rvert \\
    \le & \lVert \hat{\mu}_1 - \bar{\mu}_1 \rVert_\infty \frac{1}{n} \sum_{i=1}^n D_i \Big\lvert1 + \sum_{D_j=1-D_i} w_{j\leftarrow i} \Big\rvert = \lVert \hat{\mu}_1 - \bar{\mu}_1 \rVert_\infty \Big(\frac{n_1}{n} + \frac{1}{n} \sum_{i:D_i=1} \Big\lvert \sum_{j:D_j = 0} w_{j\leftarrow i}\Big\rvert \Big) = o_\P(1),  
\end{align*}
where the last step is due to Assumptions~\ref{asp:dr}, \ref{asp:weight}, \ref{asp:dr1}. We then have
\begin{align}\label{eq:dml12}
    \frac{1}{n} \Big[ \sum_{i=1}^n D_i \Big(1 + \sum_{D_j=1-D_i} &w_{j\leftarrow i} \Big) \Big(\bar{\mu}_1(X_i) - \hat{\mu}_1(X_i)\Big) -\notag\\
     &\sum_{i=1}^n (1-D_i)\Big(1 + \sum_{D_j=1-D_i} w_{j\leftarrow i} \Big) \Big(\bar{\mu}_0(X_i) - \hat{\mu}_0(X_i)\Big) \Big] = o_\P(1).
\end{align}

For the third term in \eqref{eq:dml1}, by the Cauchy-Schwarz inequality,
\begin{align*}
    & \E \Big[\Big\lvert\frac{1}{n} \sum_{i=1}^n D_i \Big(1 + \sum_{D_j=1-D_i} w_{j\leftarrow i} - \frac{1}{e(X_i)}\Big) \bar{R}_i \Big\rvert\Big] \le \E \Big[ \Big\lvert D_1 \Big(1 + \sum_{D_j=1-D_1} w_{j\leftarrow 1} - \frac{1}{e(X_1)}\Big) \bar{R}_1 \Big\rvert \Big]\\
    \le & \Big\{\E \Big[D_1 \Big( 1 + \sum_{D_j=1-D_1} w_{j\leftarrow 1} - \frac{1}{e(X_1)} \Big)^2\Big] \Big\}^{1/2} \Big\{\E \Big[D_1 \bar{R}_1\Big]^2\Big\}^{1/2} \\
    = & \Big\{\E \Big[D_1 \Big(\sum_{D_j=1-D_1} w_{j\leftarrow 1} - \frac{1-e(X_1)}{e(X_1)}\Big)^2\Big]\Big\}^{1/2} \Big\{\E\Big[D_1 (Y_1(1)-\bar{\mu}_1(X_1))^2\Big]\Big\}^{1/2}\\
    = & \Big\{\E \Big[D_1 \Big(\sum_{D_j=1-D_1} w_{j\leftarrow 1} - \frac{1-e(X_1)}{e(X_1)}\Big)^2\Big]\Big\}^{1/2} \Big\{\E\Big[D_1\Big(\sigma_1^2(X_1)+[\mu_1(X_1)-\bar{\mu}_1(X_1)]^2\Big)\Big]\Big\}^{1/2} = o(1),
\end{align*}
where $\sigma_1^2(x) = \E [U^2_1 \given X=x]$ for $x \in \cS$. We then obtain by the Markov inequality that
\begin{align}\label{eq:dml13}
    \frac{1}{n} \Big[ \sum_{i=1}^n D_i \Big(1 + &\sum_{D_j=1-D_i} w_{j\leftarrow i} - \frac{1}{e(X_i)}\Big) \bar{R}_i - \notag\\
    &\sum_{i=1}^n (1-D_i)\Big(1 + \sum_{D_j=1-D_i} w_{j\leftarrow i} - \frac{1}{1-e(X_i)}\Big) \bar{R}_i \Big] = o_\P(1).
\end{align}

For the fourth term in \eqref{eq:dml1},
\begin{align*}
    & \Big\lvert \frac{1}{n} \sum_{i=1}^n D_i \Big(1 - \sum_{D_j=1-D_i} w_{i\leftarrow j} \Big) \hat{\mu}_0(X_i) \Big\rvert \\
    \le &  \lVert \hat{\mu}_0 - \bar{\mu}_0 \rVert_\infty \Big(\frac{1}{n} \sum_{i=1}^n D_i \Big\lvert 1 - \sum_{D_j=1-D_i} w_{i\leftarrow j} \Big\rvert\Big) + \frac{1}{n} \sum_{i=1}^n \Big\lvert D_i\Big(1 - \sum_{D_j=1-D_i} w_{i\leftarrow j} \Big) \bar{\mu}_0(X_i) \Big\rvert.
\end{align*}
From Assumptions~\ref{asp:weight}-\ref{asp:dr1} and the Cauchy-Schwarz inequality, the above two terms are both $o_\P(1)$, and thus
\begin{align}\label{eq:dml14}
    \frac{1}{n} \sum_{i=1}^n D_i \Big(1 - \sum_{D_j=1-D_i} w_{i\leftarrow j} \Big) \hat{\mu}_0(X_i) - \frac{1}{n} \sum_{i=1}^n (1-D_i) \Big(1 - \sum_{D_j=1-D_i} w_{i\leftarrow j} \Big) \hat{\mu}_1(X_i) = o_\P(1).
\end{align}

For the fifth term in \eqref{eq:dml1}, notice that $\bar{\mu}_1$ does not depend on the samples, then
\[
    \E \Big[\frac{1}{n} \sum_{i=1}^n \Big(1 - \frac{D_i}{e(X_i)}\Big) \bar{\mu}_1(X_i) \Biggiven \mX \Big] = \frac{1}{n} \sum_{i=1}^n \E \Big[ 1 - \frac{D_i}{e(X_i)}\Biggiven X_i \Big] \bar{\mu}_1(X_i) = 0,
\]
and
\begin{align*}
    & \Var\Big[\frac{1}{n} \sum_{i=1}^n \Big(1 - \frac{D_i}{e(X_i)}\Big) \bar{\mu}_1(X_i)\Big] = \E\Big[\Var \Big[\frac{1}{n} \sum_{i=1}^n \Big(1 - \frac{D_i}{e(X_i)}\Big) \bar{\mu}_1(X_i) \Biggiven \mX \Big]\Big]\\
    = & \frac{1}{n} \E\Big[\Var \Big[\Big(1 - \frac{D_1}{e(X_1)}\Big) \bar{\mu}_1(X_1) \Biggiven X_1 \Big]\Big] = \frac{1}{n} \E \Big[\bar{\mu}^2_1(X_1) \Big(\frac{1}{e(X_1)} -1\Big)\Big] = O(n^{-1}).
\end{align*}
Then
\begin{align}\label{eq:dml15}
    \frac{1}{n} \Big[ \sum_{i=1}^n \Big(1 - \frac{D_i}{e(X_i)}\Big) \bar{\mu}_1(X_i) - \sum_{i=1}^n \Big(1 - \frac{1-D_i}{1-e(X_i)}\Big) \bar{\mu}_0(X_i) \Big] = o_\P(1).
\end{align}

For the sixth term in \eqref{eq:dml1}, notice that $\E [Y^2]$ are bounded from Assumption~\ref{asp:dr} and $[(X_i,D_i,Y_i)]_{i=1}^n$ are i.i.d.. Using the weak law of large numbers \cite[Theorem 2.2.3]{MR3930614} yields
\begin{align}\label{eq:dml16}
    \frac{1}{n} \Big[ \sum_{i=1}^n \frac{D_i}{e(X_i)} Y_i - \sum_{i=1}^n \frac{1-D_i}{1-e(X_i)} Y_i \Big] \stackrel{\sf p}{\longrightarrow} \E\Big[Y_i(1)-Y_i(0)\Big] = \tau.
\end{align}

Plugging \eqref{eq:dml11}, \eqref{eq:dml12}, \eqref{eq:dml13}, \eqref{eq:dml14}, \eqref{eq:dml15} into \eqref{eq:dml1} completes the proof.

{\bf Part II.} Suppose the outcome model is correct, i.e., Assumption~\ref{asp:dr2} holds. By Lemma~\ref{lemma:mbc},
\begin{align*}
    & \hat\tau_w  = \hat{\tau}^{\rm reg} + \frac{1}{n} \sum_{i=1}^n (2D_i-1)\Big(1 + \sum_{j:D_j=1-D_i} w_{j\leftarrow i}\Big) \hat{R}_i + \frac{1}{n} \sum_{i=1}^n (2D_i-1) \Big(1 - \sum_{j:D_j=1-D_i} w_{i\leftarrow j} \Big) \hat{\mu}_{1-D_i}(X_i)\\
    = & \frac{1}{n} \sum_{i=1}^n \Big[\hat{\mu}_1(X_i) - \mu_1(X_i)\Big] - \frac{1}{n} \sum_{i=1}^n \Big[\hat{\mu}_0(X_i) - \mu_0(X_i)\Big]\\
    &+  \frac{1}{n} \Big[ \sum_{i=1}^n D_i \Big(1 + \sum_{D_j=1-D_i} w_{j\leftarrow i}\Big) \Big(\mu_1(X_i) - \hat{\mu}_1(X_i)\Big) - \sum_{i=1}^n (1-D_i)\Big(1 + \sum_{D_j=1-D_i} w_{j\leftarrow i}\Big) \Big(\mu_0(X_i) - \hat{\mu}_0(X_i)\Big) \Big]\\
    &+  \frac{1}{n} \Big[ \sum_{i=1}^n D_i \Big(1 + \sum_{D_j=1-D_i} w_{j\leftarrow i}\Big) \Big(Y_i - \mu_1(X_i) \Big) - \sum_{i=1}^n (1-D_i) \Big(1 + \sum_{D_j=1-D_i} w_{j\leftarrow i}\Big) \Big(Y_i - \mu_0(X_i) \Big) \Big]\\
    &+  \frac{1}{n} \sum_{i=1}^n D_i \Big(1 - \sum_{D_j=1-D_i} w_{i\leftarrow j} \Big) \hat{\mu}_0(X_i) - \frac{1}{n} \sum_{i=1}^n (1-D_i) \Big(1 - \sum_{D_j=1-D_i} w_{i\leftarrow j} \Big) \hat{\mu}_1(X_i)\\
    &+  \frac{1}{n} \sum_{i=1}^n \Big[\mu_1(X_i) - \mu_0(X_i)\Big].
    \yestag\label{eq:dml2}
\end{align*}

For the first term in \eqref{eq:dml2}, in the same way as \eqref{eq:dml11},
\begin{align}\label{eq:dml21}
    \frac{1}{n} \sum_{i=1}^n \Big[\hat{\mu}_1(X_i) - \mu_1(X_i)\Big] - \frac{1}{n} \sum_{i=1}^n \Big[\hat{\mu}_0(X_i) - \mu_0(X_i)\Big] = o_\P(1).
\end{align}

For the second term in \eqref{eq:dml2}, in the same way as \eqref{eq:dml12}, by Assumptions~\ref{asp:dr}, \ref{asp:weight}, and \ref{asp:dr2}.
\begin{align}\label{eq:dml22}
    \frac{1}{n} \Big[ \sum_{i=1}^n D_i \Big(1 + \sum_{D_j=1-D_i} &w_{j\leftarrow i}\Big) \Big(\mu_1(X_i) - \hat{\mu}_1(X_i)\Big) -\notag\\ 
    &\sum_{i=1}^n (1-D_i)\Big(1 + \sum_{D_j=1-D_i} w_{j\leftarrow i}\Big) \Big(\mu_0(X_i) - \hat{\mu}_0(X_i)\Big) \Big] = o_\P(1).
\end{align}

For the third term in \eqref{eq:dml2}, noticing that $[w_{i\leftarrow j}]_{D_i+D_j=1}$ is a function of $\mX$ and $\mD$, we can obtain
\begin{align*}
  \E \Big[ \frac{1}{n} \sum_{i=1}^n D_i \Big(1 + \sum_{D_j=1-D_i} w_{j\leftarrow i}\Big) \Big(Y_i - \mu_1(X_i) \Big) \Biggiven \mX, \mD \Big] = 0,
\end{align*}
and
\begin{align*}
    &\E \Big[ \Big\lvert \frac{1}{n} \sum_{i=1}^n D_i \Big(1 + \sum_{D_j=1-D_i} w_{j\leftarrow i}\Big) \Big(Y_i - \mu_1(X_i) \Big) \Big\rvert\Big] \le \E \Big[ \Big\lvert \frac{1}{n} \sum_{i=1}^n D_i \Big(1 + \sum_{D_j=1-D_i} w_{j\leftarrow i}\Big) \Big\rvert\Big] \lVert \sigma_1 \rVert_\infty\\
    \lesssim& \lVert \sigma_1 \rVert_\infty = O(1),
\end{align*}
where $\sigma_1^2(x) = \E [U^2_1 \given X=x]$ for $x \in \cS$. Accordingly, by the weak law of large number, we obtain
\begin{align}\label{eq:dml23}
    \frac{1}{n} \Big[ \sum_{i=1}^n D_i \Big(1 + \sum_{D_j=1-D_i} &w_{j\leftarrow i}\Big) \Big(Y_i - \mu_1(X_i) \Big) - \notag\\
    &\sum_{i=1}^n (1-D_i) \Big(1 + \sum_{D_j=1-D_i} w_{j\leftarrow i}\Big) \Big(Y_i - \mu_0(X_i) \Big) \Big] = o_\P(1).
\end{align}

For the fourth term in \eqref{eq:dml2}, in the same way as \eqref{eq:dml14},
\begin{align*}
    \frac{1}{n} \sum_{i=1}^n D_i \Big(1 - \sum_{D_j=1-D_i} w_{i\leftarrow j} \Big) \hat{\mu}_0(X_i) - \frac{1}{n} \sum_{i=1}^n (1-D_i) \Big(1 - \sum_{D_j=1-D_i} w_{i\leftarrow j} \Big) \hat{\mu}_1(X_i) = o_\P(1).
\end{align*}

For the fifth term in \eqref{eq:dml2}, notice that $\E[\mu_\omega^2(X)]$ is bounded for $\omega \in \{0,1\}$. Using the weak law of large number, we obtain
\begin{align}\label{eq:dml24}
\frac{1}{n} \sum_{i=1}^n \Big[\mu_1(X_i) - \mu_0(X_i)\Big] \stackrel{\sf p}{\longrightarrow} \E \Big[\mu_1(X_1) - \mu_0(X_1)\Big] = \tau.
\end{align}

Plugging \eqref{eq:dml21}, \eqref{eq:dml22}, \eqref{eq:dml23}, \eqref{eq:dml24} into \eqref{eq:dml2} completes the proof.
\end{proof}

\subsection{Proof of Theorem~\ref{thm:mbc}}

\begin{proof}[Proof of Theorem~\ref{thm:mbc}]

Let $\epsilon_i = Y_i - \mu_{D_i}(X_i)$ for any $i \in \zahl{n}$. We decompose $\hat\tau_w $ as
\begin{align*}
    \hat\tau_w  =& \hat{\tau}^{\rm reg} + \frac{1}{n} \sum_{i=1}^n (2D_i-1)\Big(1 + \sum_{j:D_j=1-D_i} w_{j\leftarrow i}\Big) \hat{R}_i + \frac{1}{n} \sum_{i=1}^n (2D_i-1) \Big(1 - \sum_{j:D_j=1-D_i} w_{i\leftarrow j} \Big) \hat{\mu}_{1-D_i}(X_i)\\
    = & \frac{1}{n} \sum_{i=1}^n \Big[\mu_1(X_i) - \mu_0(X_i)\Big] + \frac{1}{n} \sum_{i=1}^n (2D_i-1)\Big(1 + \sum_{j:D_j=1-D_i} w_{j\leftarrow i}\Big) \epsilon_i \\
    & + \frac{1}{n}\sum_{i=1}^n (2D_i-1) \Big[\sum_{j:D_j=1-D_i} w_{i\leftarrow j} \mu_{1-D_i}(X_i)-\sum_{j:D_j=1-D_i} w_{i\leftarrow j} \mu_{1-D_i}(X_j)\Big]\\
    & - \frac{1}{n}\sum_{i=1}^n (2D_i-1) \Big[\sum_{j:D_j=1-D_i} w_{i\leftarrow j} \hat{\mu}_{1-D_i}(X_i)-\sum_{j:D_j=1-D_i} w_{i\leftarrow j} \hat{\mu}_{1-D_i}(X_j)\Big]\\
    & + \frac{1}{n} \sum_{i=1}^n (2D_i-1) \Big(1 - \sum_{j:D_j=1-D_i} w_{i\leftarrow j} \Big) \mu_{1-D_i}(X_i)\\
    =:& \bar{\tau}(\mX) + E_n + B_n - \hat{B}_n + \tilde{B}_n.
\end{align*}

For any $x \in \cS$, define $\sigma_\omega^2(x) := \E [U_\omega^2 \given X=x] = \E [[Y(\omega) - \mu_\omega(X)]^2 \given X=x ]$ for $\omega \in \{0,1\}$. Let
\[
  V^\tau := \E \Big[ \mu_1(X) - \mu_0(X) - \tau \Big]^2~~~{\rm and}~~~ V^E := \frac{1}{n} \sum_{i=1}^n \Big(1 + \sum_{j:D_j=1-D_i} w_{j\leftarrow i}\Big)^2 \sigma_{D_i}^2(X_i).
\]

We have the following central limit theorem on $\bar{\tau}(\mX) + E_n$.

\begin{lemma}\label{lemma:mbc,clt}
    Under Assumptions~\ref{asp:dr}-\ref{asp:se1},
    \begin{align}
        \sqrt{n} \Big(V^\tau + V^E\Big)^{-1/2} \Big(\bar{\tau}(\mX) + E_n - \tau \Big)\stackrel{\sf d}{\longrightarrow} N\Big(0,1\Big).
    \end{align}
\end{lemma}

While $V^E$ depends on the data, $V^E$ converges to a constant in probability.

\begin{lemma}\label{lemma:mbc,ve}
    Under Assumptions~\ref{asp:dr}-\ref{asp:dr1},
    \begin{align*}
        V^E \stackrel{\sf p}{\longrightarrow} \E \Big[\frac{\sigma_1^2(X)}{e(X)} + \frac{\sigma_0^2(X)}{1-e(X)}\Big].
    \end{align*}
\end{lemma}

For the bias term $B_M-\hat B_M$, in light of the smoothness conditions on $\mu_\omega$ and approximation conditions on $\hat{\mu}_\omega$ for $\omega \in \{0,1\}$, one can establish the following lemma.

\begin{lemma}\label{lemma:mbc,bias}
    Under Assumptions~\ref{asp:dr}, \ref{asp:weight}, \ref{asp:se2},
    \begin{align*}
        \sqrt{n} \Big(B_n - \hat{B}_n \Big)\stackrel{\sf p}{\longrightarrow} 0.
    \end{align*}
\end{lemma}

Under Assumption~\ref{asp:se2},
\begin{align*}
    \E[\lvert \tilde{B}_n \rvert] =& \E\Big[ \Big\lvert \frac{1}{n} \sum_{i=1}^n (2D_i-1) \Big(1 - \sum_{j:D_j=1-D_i} w_{i\leftarrow j} \Big) \mu_{1-D_i}(X_i) \Big\rvert \Big]\\
    \le& \E\Big[ \Big\lvert \Big(1 - \sum_{j:D_j=1-D_1} w_{1\leftarrow j} \Big) \mu_{1-D_1}(X_1) \Big\rvert \Big] \\
    \le& \Big\{\E\Big[ 1 - \sum_{j:D_j=1-D_1} w_{1\leftarrow j} \Big]^2 \Big\}^{1/2} \{\E[\mu_{1-D_1}^2(X_1)]\}^{1/2} = o(n^{-1/2}),
\end{align*}
and then
\begin{align}\label{eq:mbc,bias}
    \sqrt{n} \tilde{B}_n \stackrel{\sf p}{\longrightarrow} 0.
\end{align}

Combining Lemma~\ref{lemma:mbc,clt}, Lemma~\ref{lemma:mbc,ve}, Lemma~\ref{lemma:mbc,bias}, and Equation~\ref{eq:mbc,bias} completes the proof.

The consistency of the variance estimator can be established in a similar way as the proof of Theorem 5.1 in \cite{lin2021estimation}.
\end{proof}

\subsection{Proof of Theorem~\ref{thm:dml}}

\begin{proof}[Proof of Theorem~\ref{thm:dml}]
The proof for double robustness is the same as Theorem~\ref{thm:dr}.

We follow the proof of Theorem 5.1(ii) in \cite{lin2021estimation} for the semiparametric efficiency, by checking Assumptions 3.1 and 3.2 in \cite{chernozhukov2018double}. In the following the notation in \cite{chernozhukov2018double} is adopted.

The score (or the efficient influence function as used in \citet[Section 3.4]{tsiatis2006semiparametric}) is
\[
  \psi(X,D,Y;\tilde{\tau},\tilde{\zeta}) := \tilde{\mu}_1(X) - \tilde{\mu}_0(X) + \frac{D(Y-\tilde{\mu}_1(X))}{\tilde{e}(X)} - \frac{(1-D)(Y-\tilde{\mu}_0(X))}{1-\tilde{e}(X)} - \tilde{\tau},
\]
where $\tilde{\zeta}(x) = (\tilde{\mu}_0(x),\tilde{\mu}_1(x),\tilde{\rho}_0(x),\tilde{\rho}_1(x))$ are the nuisance parameters by letting $\tilde{\rho}_0(x) = 1/(1-\tilde{e}(x))$ and $\tilde{\rho}_1(x) = 1/\tilde{e}(x)$. Let $\rho_0(x) = 1/(1-e(x))$ and $\rho_1(x) = 1/e(x)$. Then the true value is $\zeta(x) = (\mu_0(x),\mu_1(x),\rho_0(x),\rho_1(x))$.

We can then write the score as
\[
  \psi(X,D,Y;\tilde{\tau},\tilde{\zeta}) = \tilde{\mu}_1(X) - \tilde{\mu}_0(X) + D(Y-\tilde{\mu}_1(X))\tilde{\rho}_1(X) - (1-D)(Y-\tilde{\mu}_0(X))\tilde{\rho}_0(X) - \tilde{\tau}.
\]

For the $\kappa$ in Assumption~\ref{asp:dml1}, let $q = 2+ \kappa/2$, $q_1 = 2 + \kappa$ and $q_2$ such that $q^{-1} = q_1^{-1} + q_2^{-1}$. Let $\cT_n$ be the set consisting of all $\tilde{\zeta}$ such that for $\omega \in \{0,1\}$,
\begin{align*}
    & \lVert \tilde{\mu}_\omega - \mu_\omega \rVert_{p_1} = O(r_1), ~~ \lVert \tilde{\rho}_\omega - \rho_\omega \rVert_{p_2} = O(r_2),\\
    & \lVert \tilde{\mu}_\omega - \mu_\omega \rVert_\infty = o(1), ~~ \lVert \tilde{\rho}_\omega - \rho_\omega \rVert
    _2 = o(1), ~~ \lVert \tilde{\rho}_\omega\rVert_{q_2} = O(1),
\end{align*}
where $p_1,p_2$ are the ones in Assumption~\ref{asp:dml2}. Then the selection of $\cT_n$ satisfies Assumption 3.2(a) in \cite{chernozhukov2018double} from Assumptions~\ref{asp:dr1}, \ref{asp:dr2}, \ref{asp:dml2}.

Steps 1-3 in the proof of Theorem 5.1(ii) in \cite{lin2021estimation} can be directly applied.

For step 4 therein, we can establish in the same way that for $\omega \in \{0,1\}$, $\lVert \mu_\omega \rVert_{2+\kappa} = O(1)$ from $\lVert Y \rVert_{2+\kappa} = O(1)$, and $\tau = O(1)$. Then from H\"older's inequality and $\lVert \rho_\omega \rVert_\infty$ is bounded for $\omega \in \{0,1\}$, for any $\tilde{\zeta} \in \cT_n$,
\begin{align*}
  & \lVert \psi(X,D,Y;\tau,\tilde{\zeta}) \lVert_q = \lVert \tilde{\mu}_1(X) - \tilde{\mu}_0(X) + D(Y-\tilde{\mu}_1(X))\tilde{\rho}_1(X) - (1-D)(Y-\tilde{\mu}_0(X))\tilde{\rho}_0(X) - \tau \rVert_q\\
  \le & \lVert \tilde{\mu}_1(X) \rVert_q + \lVert \tilde{\mu}_0(X) \rVert_q + \lVert (Y-\tilde{\mu}_1(X))\tilde{\rho}_1(X) \rVert_q + \lVert (Y-\tilde{\mu}_0(X))\tilde{\rho}_0(X) \rVert_q + \tau\\
  \le & \lVert \mu_1 \rVert_q + \lVert \tilde{\mu}_1 - \mu_1 \rVert_q + \lVert \mu_0 \rVert_q + \lVert \tilde{\mu}_0 - \mu_0 \rVert_q + (\lVert Y \rVert_{q_1} + \lVert \mu_1 \rVert_{q_1} + \lVert \tilde{\mu}_1 - \mu_1 \rVert_{q_1}) \lVert \tilde{\rho}_1 \rVert_{q_2}\\
  & + (\lVert Y \rVert_{q_1} + \lVert \mu_0 \rVert_{q_1} + \lVert \tilde{\mu}_0 - \mu_0 \rVert_{q_1}) \lVert \tilde{\rho}_0 \rVert_{q_2} + \tau \\
  \le & \lVert \mu_1 \rVert_{2+\kappa} + \lVert \tilde{\mu}_1 - \mu_1 \rVert_\infty + \lVert \mu_0 \rVert_{2+\kappa} + \lVert \tilde{\mu}_0 - \mu_0 \rVert_\infty + (\lVert Y \rVert_{2+\kappa} + \lVert \mu_1 \rVert_{2+\kappa} + \lVert \tilde{\mu}_1 - \mu_1 \rVert_\infty) \lVert \tilde{\rho}_1 \rVert_{q_2}\\
  & + (\lVert Y \rVert_{2+\kappa} + \lVert \mu_0 \rVert_{2+\kappa} + \lVert \tilde{\mu}_0 - \mu_0 \rVert_\infty) \lVert \tilde{\rho}_0 \rVert_{q_2} + \tau = O(1).
\end{align*}
The last step is from the definition of $\cT_n$ and the selection of $q,q_1$. Then we complete this step.

For step 5 therein, by H\"older's inequality, for any $\tilde{\zeta} \in \cT_n$,
\begin{align*}
  & \lVert \psi(X,D,Y;\tau,\tilde{\zeta}) - \psi(X,D,Y;\tau,\zeta) \rVert_2\\
  \le & \lVert \tilde{\mu}_1 - \mu_1 \rVert_2 + \lVert \tilde{\mu}_0- \mu_0 \rVert_2 + \lVert D(Y-\tilde{\mu}_1(X))\tilde{\rho}_1(X) - D(Y-\mu_1(X))\rho_1(X) \rVert_2 \\
  & + \lVert (1-D)(Y-\tilde{\mu}_0(X))\tilde{\rho}_0(X) - (1-D)(Y-\mu_0(X))\rho_0(X) \rVert_2\\
  \le & \lVert \tilde{\mu}_1 - \mu_1 \rVert_2 + \lVert \tilde{\mu}_0- \mu_0 \rVert_2 + \lVert (Y-\mu_1(X)) (\tilde{\rho}_1 - \rho_1) \rVert
  _2 + \lVert (\tilde{\mu}_1 - \mu_1) \tilde{\rho}_1 \rVert_2 \\
  & + \lVert (Y-\mu_0(X)) (\tilde{\rho}_0 - \rho_0) \rVert
  _2 + \lVert (\tilde{\mu}_0 - \mu_0) \tilde{\rho}_0 \rVert_2\\
  \le & \lVert \tilde{\mu}_1 - \mu_1 \rVert_2 + \lVert \tilde{\mu}_0- \mu_0 \rVert_2 + O(\lVert \tilde{\rho}_1 - \rho_1 \rVert
  _2) + \lVert \tilde{\mu}_1 - \mu_1 \rVert_\infty \lVert \tilde{\rho}_1 \rVert_2 \\
  & + O(\lVert \tilde{\rho}_0 - \rho_0 \rVert
  _2) + \lVert \tilde{\mu}_0 - \mu_0 \rVert_\infty \lVert \tilde{\rho}_0 \rVert_2 \\
  =& o(1).
\end{align*}
The last two steps are due to the definition of $\cT_n$ and that the construction of weights does not depend on the responses.

Notice that for any $t \in (0,1)$,
\begin{align*}
  &\partial_t^2 \E \psi(X,D,Y;\tau,\zeta + t (\tilde{\zeta} - \zeta)) \\
  = & -2\Big(\E [D(\tilde{\mu}_1(X) - \mu_1(X))(\tilde{\rho}_1(X)-\rho_1(X))] - \E [(1-D)(\tilde{\mu}_0(X) - \mu_0(X))(\tilde{\rho}_0(X)-\rho_0(X))] \Big).
\end{align*}
Then by the definition of $\cT_n$, for any $\tilde{\zeta} \in \cT_n$,
\[
  \lvert \partial_t^2 \E \psi(X,D,Y;\tau,\zeta + t (\tilde{\zeta} - \zeta)) \rvert \le 2 [\lVert \tilde{\mu}_1 - \mu_1 \rVert_{p_1} \lVert \tilde{\rho}_1 - \rho_1 \rVert_{p_2} + \lVert \tilde{\mu}_0 - \mu_0 \rVert_{p_1} \lVert \tilde{\rho}_0 - \rho_0 \rVert_{p_2}] = O(r_1r_2) = o(n^{-1/2}).
\]
This completes the proof this step and thus finishes the whole proof.
\end{proof}

\subsection{Proof of Theorem~\ref{thm:kernel}}

\begin{proof}[Proof of Theorem~\ref{thm:kernel}]

{\bf Proof of Theorem~\ref{thm:kernel}\ref{thm:kernel1}.} Assumptions~\ref{asp:weight}, \ref{asp:se2}\ref{asp:se2,w1} hold since we always have $\sum_{j:D_j=1-D_1} w_{1\leftarrow j} = 1$. Assumption~\ref{asp:dr2}\ref{asp:dr2,w} holds since the construction of weights only based on $\mX$ and $\mD$ from the definition. Assumption~\ref{asp:dr2}\ref{asp:dr2,w2} holds since all the weights are nonnegative and then $\E[ \lvert \sum_{j:D_j=1-D_1} w_{j\leftarrow 1} \rvert] = \E[ \sum_{j:D_j=1-D_1} w_{j\leftarrow 1} ] \lesssim \E[ \sum_{j:D_j=1-D_1} w_{1\leftarrow j}] = 1$.

{\bf Proof of Theorem~\ref{thm:kernel}\ref{thm:kernel2}.} To verify Assumption~\ref{asp:dr1}\ref{asp:dr1,w}, notice that
\begin{align*}
    & \E \Big[ \sum_{j:D_j=1-D_1} w_{j\leftarrow 1} - \Big(D_1 \frac{1-e(X_1)}{e(X_1)} + (1-D_1) \frac{e(X_1)}{1-e(X_1)} \Big) \Big]^2\\
    =& \E \Big[\E \Big[ \Big(\sum_{j:D_j=1-D_1} w_{j\leftarrow 1} - D_1 \frac{1-e(X_1)}{e(X_1)} - (1-D_1) \frac{e(X_1)}{1-e(X_1)} \Big)^2 \Biggiven \mD \Big] \Big]\\
    =& \E \Big[\E \Big[ \Big(\sum_{j:D_j=0} w_{j\leftarrow 1} - \frac{1-e(X_1)}{e(X_1)}\Big)^2 \Biggiven \mD,D_1=1 \Big] \ind\Big(D_1 = 1\Big)\Big] \\
    &+ \E \Big[\E \Big[ \Big(\sum_{j:D_j=1} w_{j\leftarrow 1} - \frac{e(X_1)}{1-e(X_1)} \Big)^2 \Biggiven \mD,D_1=0 \Big] \ind\Big(D_1 = 0\Big)\Big].
\end{align*}

Then it suffices to consider the first term above under $D_1 = 1$, and the second term under $D_1=0$ can be established in the same way.

For any $\omega \in \{0,1\}$, let $f_\omega(\cdot)$ be the density function of $X|D=\omega$. Under $D_1 = 1$, we have
\begin{align*}
    &\sum_{j:D_j=0} w_{j\leftarrow 1} - \frac{1-e(X_1)}{e(X_1)} \\
    =& \sum_{j:D_j=0} \frac{K_H(X_j-X_1)}{\sum_{k:D_k=1} K_H(X_j-X_k)} - \frac{1-e(X_1)}{e(X_1)}\\
    =& \Big[\frac{n_0}{n_1} \frac{f_0(X_1)}{f_1(X_1)} - \frac{1-e(X_1)}{e(X_1)} \Big] + \frac{n_0}{n_1} f_1^{-1}(X_1) \Big(\frac{1}{n_0} \sum_{j:D_j=0} K_H(X_j-X_1) - f_0(X_1)\Big)\\
    &+ \frac{n_0}{n_1} \frac{1}{n_0} \sum_{j:D_j=0} \Big(f_1^{-1}(X_j) - f_1^{-1}(X_1)\Big) K_H(X_j-X_1)\\
    &+ \frac{n_0}{n_1} \frac{1}{n_0} \sum_{j:D_j=0} \Big[\Big(\frac{1}{n_1}\sum_{k:D_k=1} K_H(X_j-X_k)\Big)^{-1} - f_1^{-1}(X_j)\Big] K_H(X_j-X_1)\\
    =:& R_1 + \frac{n_0}{n_1}R_2 + \frac{n_0}{n_1}R_3 + \frac{n_0}{n_1}R_4.
    \yestag\label{eq:kernel1}
\end{align*}

For $R_1$, from the law of large number, we have
\begin{align*}
    \lim_{n \to \infty } \E \Big[\E \Big[ R_1^2 \Biggiven \mD,D_1=1 \Big] \ind\Big(D_1 = 1\Big)\Big] = 0.
\end{align*}

For $R_2$ to $R_4$, we have the following lemma.

\begin{lemma}\label{lemma:kernel1}
    Under Assumption~\ref{asp:kernel,dr}, we have
    \begin{align*}
        \lim_{n \to \infty } \E \Big[\frac{n_0^2}{n_1^2} \E \Big[ R_i^2 \Biggiven \mD,D_1=1 \Big] \ind\Big(D_1 = 1\Big)\Big] = 0, ~~ i=2,3,4.
    \end{align*}
\end{lemma}

We then complete the proof by \eqref{eq:kernel1}.

{\bf Proof of Theorem~\ref{thm:kernel}\ref{thm:kernel3}.} To verify Assumption~\ref{asp:se2}, notice that
\begin{align*}
    \E \Big[ \sum_{j:D_j=1-D_1} \lvert w_{1\leftarrow j} \rvert \cdot \lVert X_j - X_1 \rVert^k \Big] =& \E \Big[\E \Big[ \sum_{j:D_j=0} w_{1\leftarrow j} \lVert X_j - X_1 \rVert^k \Biggiven \mD,D_1=1 \Big] \ind\Big(D_1 = 1\Big)\Big] \\
    &+ \E \Big[\E \Big[ \sum_{j:D_j=1} w_{1\leftarrow j} \lVert X_j - X_1 \rVert^k \Biggiven \mD,D_1=0 \Big] \ind\Big(D_1 = 0\Big)\Big].
\end{align*}

It suffices to consider the first term under $D_1=1$.

Under $D_1=1$, notice that
\begin{align*}
    & \sum_{j:D_j=0} w_{1\leftarrow j} \lVert X_j - X_1 \rVert^k = \Big(\sum_{k:D_k=0} K_H(X_1-X_k)\Big)^{-1} \sum_{j:D_j=0} K_H(X_1-X_j) \lVert X_j - X_1 \rVert^k\\
    = & \Big(\frac{1}{n_0} \sum_{k:D_k=0} K_H(X_1-X_k)\Big)^{-1} \frac{1}{n_0} \sum_{j:D_j=0} K_H(X_1-X_j) \lVert X_j - X_1 \rVert^k.
\end{align*}
From the properties of kernel density estimation, it suffices to consider
\begin{align*}
    & \E \Big[\E \Big[ f_0^{-1}(X_1) \frac{1}{n_0} \sum_{j:D_j=0} K_H(X_1-X_j) \lVert X_j - X_1 \rVert^k \Biggiven \mD,D_1=1 \Big] \ind\Big(D_1 = 1\Big)\Big]\\
    =& \E \Big[\E \Big[ f_0^{-1}(X_1) K_H(X_1-X_2) \lVert X_2 - X_1 \rVert^k \Biggiven D_1=1,D_2=0 \Big] \ind\Big(D_1 = 1\Big)\Big]\\
    \lesssim& \int K_H(y-x) \lVert y - x \rVert^k \d x \d y = \lvert H \rvert^{-1/2} \int K(H^{-1/2}(y-x)) \lVert y - x \rVert^k \d x \d y\\
    =& \int K(z)  \lVert H^{1/2} z \rVert^k \d x \d z \le \lVert H^{1/2} \rVert_2^k \int K(z)  \lVert z \rVert^k \d z.
\end{align*}
Then it suffices to assume
\[
    \int K(z)  \lVert z \rVert^k \d z = O(1),~~\lVert H^{1/2} \rVert_2^k = o(n^{-1/2}),~~ \lVert H^{1/2} \rVert_2^\ell = o(n^{-1/2+ \gamma_\ell}) ~~\mbox{\rm for all}~~ \ell \in \zahl{k-1},
\]
and under Assumptions~\ref{asp:kernel,dr}, \ref{asp:kernel,se}, the above conditions hold.

{\bf Proof of Theorem~\ref{thm:kernel}\ref{thm:kernel4}.} To verify Assumption~\ref{asp:dml2}, from \eqref{eq:kernel1}, we establish the convergence rate of each term in \eqref{eq:kernel1} seperately.

For $R_1$, it is easy to check
\begin{align*}
    \E \Big[\E \Big[ \Big\lvert R_1 \Big\rvert \Biggiven \mD,D_1=1 \Big] \ind\Big(D_1 = 1\Big)\Big] = O(n^{-1/2}).
\end{align*}

For $R_2$ to $R_4$, we have the following lemma.

\begin{lemma}\label{lemma:kernel2}
    Under Assumptions~\ref{asp:kernel,dr}, \ref{asp:kernel,dml}, we have
    \begin{align*}
        \E \Big[\frac{n_0}{n_1} \E \Big[ \Big\lvert R_i \Big\rvert \Biggiven \mD,D_1=1 \Big] \ind\Big(D_1 = 1\Big)\Big] \lesssim \lVert H^{1/2} \rVert_2 + (n \lvert H^{1/2} \rvert)^{-1/2}, ~~ i=2,3,4.
    \end{align*}
\end{lemma}

From the properties of kernel density estimation and $K$ is bounded, we obtain for any $\kappa>0$,
\[
    \E \Big[ \sum_{j:D_j=1-D_1} w_{j\leftarrow 1} \Big]^\kappa = O(1).
\]
Then the proof of verifying Assumption~\ref{asp:dml2} is complete.
\end{proof}

\subsection{Proof of Theorem~\ref{thm:wnn}}

\begin{proof}[Proof of Theorem~\ref{thm:wnn}]

{\bf Proof of Theorem~\ref{thm:wnn}\ref{thm:wnn1}.} Assumptions~\ref{asp:weight}, \ref{asp:dr2}\ref{asp:dr2,w}, \ref{asp:dr2}\ref{asp:dr2,w2} \ref{asp:se2}\ref{asp:se2,w1} hold since we take $\sum_{m=1}^M \gamma_{M,m} = 1$, the construction of weights is only based on $\mX$ and $\mD$, and all the weights are nonnegative.

{\bf Proof of Theorem~\ref{thm:wnn}\ref{thm:wnn2}.} To verify Assumption~\ref{asp:dr1}\ref{asp:dr1,w}, it suffices to consider
\[
    \E \Big[\E \Big[ \Big(\sum_{j:D_j=1} w_{j\leftarrow 1} - \frac{e(X_1)}{1-e(X_1)}\Big)^2 \Biggiven \mD,D_1=0 \Big] \ind\Big(D_1 = 0\Big)\Big].
\]

We first define the modified catchment area similar to Definition 2.1 in \cite{lin2021estimation}. For any $m \in \zahl{M}$, let $a_m(\cdot):\bR^d \to \cB(\bR^d)$ be the mapping from $\bR^d$ to the class of all Borel sets in $\bR^d$ so that
\begin{align*}
    a_m(x) = a_m\Big(x,\{X_i\}_{i:D_i=0}\Big):= \Big\{z \in \bR^d: \lVert \cX^0_{(m-1)}(z) - z \rVert < \lVert x-z \rVert \le \lVert \cX^0_{(m)}(z) - z \rVert\Big\},
\end{align*}
where $\cX^0_{(m)}(\cdot)$ is the mapping that returns the value of input's $m$-th NN in $\{X_i\}_{i:D_i=0}$, with $\cX^0_{0}(z) = z$ for $z \in \bR^d$. Let $a_m(i)$ be the shorthand of $a_m(X_i)$ for $i \in \zahl{n}$. From the definition of the modified catchment area, $j_m(j)=i$ if and only if $X_j \in a_m(i)$ for $i,j\in\zahl{n}$ with $D_i=0,D_j=1$ and $m \in \zahl{M}$.

We rewrite
\begin{align*}
    &\sum_{j:D_j=1} w_{j\leftarrow 1} - \frac{e(X_1)}{1-e(X_1)} \\
    =& \sum_{j:D_j=1} \sum_{m=1}^M \gamma_{M,m} \ind(j_m(j)=1) - \frac{e(X_1)}{1-e(X_1)}\\
    =& \Big[\frac{n_1}{n_0} \frac{f_1(X_1)}{f_0(X_1)} - \frac{e(X_1)}{1-e(X_1)} \Big] +\Big[\sum_{m=1}^M \gamma_{M,m} n_1 \nu_1\big(a_m(1)\big) - \frac{n_1}{n_0} \frac{f_1(X_1)}{f_0(X_1)}\Big]\\
    &+ \sum_{j:D_j=1} \Big[ \sum_{m=1}^M \gamma_{M,m} \Big(  \ind(j_m(j)=1) - \nu_1\big(a_m(1)\big) \Big) \Big].
    \yestag\label{eq:wnn1}
\end{align*}

We first establish a lemma to generalize Lemma 4.1 in \cite{lin2021estimation} for the weighted nearest neighbors case, with notation adopted from there.

\begin{lemma} \label{lemma:moment,catch}
Let $a_m(x),m\in\zahl{M}$ be the modified catchment area of $x$ based on $n_0$ samples from probability measure $\nu_0$ with density $f_0$. Let $\nu_1$ be another probability measure with density $f_1$. Assuming $M\log n_0/n_0 \to 0$ as $n_0 \to \infty$, we have
\[
    \lim_{n_0\to\infty} n_0 \E\Big[ \sum_{m=1}^M \gamma_{M,m} \nu_1\big(a_m(x)\big)\Big] = \frac{f_1(x)}{f_0(x)}
\]
holds for $\nu_0$-almost all $x$. For any positive integer $p$, if we further assume
\[
    \limsup_{n_0 \to \infty} n_0 \int_0^\infty \Big[\sum_{m=1}^M \gamma_{M,m}^p \P\Big( U_{(m-1)} \le t \le U_{(m)} \Big) \Big]^{1/p} \d t \le 1,
\]
then
    \[
    \lim_{n_0\to\infty} n_0^p \E\Big[ \Big(\sum_{m=1}^M \gamma_{M,m} \nu_1\big(a_m(x)\big) \Big)^p\Big] = \Big[\frac{f_1(x)}{f_0(x)}\Big]^p
    \]
    holds for $\nu_0$-almost all $x$.
\end{lemma}

By leveraging the same technique to establish the global $L_p$ risk consistency as Theorem 4.2 in \cite{lin2021estimation}, as long as Lemma~\ref{lemma:moment,catch} holds for $p=2$, we obtain
\begin{align}\label{eq:wnn2}
    \lim_{n \to \infty} \E \Big[\E \Big[ \Big(\sum_{m=1}^M \gamma_{M,m} n_1 \nu_1\big(a_m(1)\big) - \frac{n_1}{n_0} \frac{f_1(X_1)}{f_0(X_1)} \Big)^2 \Biggiven \mD,D_1=0 \Big] \ind\Big(D_1 = 0\Big)\Big] = 0.
\end{align}

For the third term in \eqref{eq:wnn1}, from the i.i.d.-ness of $[X_j]_{j:D_j=1}$ conditional on $\mD$,
\begin{align*}
    &\E\Big\{ \Big[\sum_{j:D_j=1} \Big[ \sum_{m=1}^M \gamma_{M,m} \Big(  \ind(j_m(j)=1) - \nu_1\big(a_m(1)\big) \Big) \Big] \Big]^2\Biggiven \mX_0,\mD,D_1=0 \Big\}\\
    =& \Var\Big\{ \sum_{j:D_j=1} \Big( \sum_{m=1}^M \gamma_{M,m} \ind(j_m(j)=1) \Big) \Biggiven \mX_0,\mD,D_1=0 \Big\}\\
    =& n_1 \Var\Big\{ \sum_{m=1}^M \gamma_{M,m} \ind(j_m(2)=1)\Biggiven \mX_0,\mD,D_1=0,D_2=1 \Big\}.
\end{align*}
Notice that for any $m,m' \in \zahl{M}$ and $m \neq m'$,
\begin{align*}
    &\Var\Big\{ \ind(j_m(2)=1) \Biggiven \mX_0,\mD,D_1=0,D_2=1 \Big\} 
    \le \P\Big(j_m(2)=1 \Biggiven \mX_0,\mD,D_1=0,D_2=1\Big) 
    = \nu_1\big(a_m(1)\big),
\end{align*}
and
\begin{align*}
    &\Cov\Big\{ \ind(j_m(2)=1),\ind(j_{m'}(2)=1) \Biggiven \mX_0,\mD,D_1=0,D_2=1 \Big\} \\
    =& -\P\Big(j_m(2)=1 \Biggiven \mX_0,\mD,D_1=0,D_2=1\Big)\P\Big(j_{m'}(2)=1 \Biggiven \mX_0,\mD,D_1=0,D_2=1\Big) \\
    =& -\nu_1\big(a_m(1)\big)\nu_1\big(a_{m'}(1)\big).
\end{align*}
Then
\begin{align*}
    \Var\Big\{ \sum_{m=1}^M \gamma_{M,m} \ind(j_m(2)=1)\Biggiven \mX_0,\mD,D_1=0,D_2=1 \Big\} \le \sum_{m=1}^M \gamma_{M,m}^2 \nu_1\big(a_m(1)\big).
\end{align*}
From Lemma~\ref{lemma:moment,catch} and $\sum_{m=1}^M \gamma_{M,m}^2 \to 0$, we obtain
\begin{align*}
    \E\Big[n_1\sum_{m=1}^M \gamma_{M,m}^2 \nu_1\big(a_m(1)\big) \Biggiven \mD, D_1=0\Big] = \frac{n_1}{n_0} \sum_{m=1}^M \gamma_{M,m}^2 = o(n_1/n_0),
\end{align*}
and then
\begin{align}\label{eq:wnn3}
    \lim_{n \to \infty} \E \Big[\E \Big[ \Big(\sum_{j:D_j=1} \Big[ \sum_{m=1}^M \gamma_{M,m} \Big(  \ind(j_m(j)=1) - \nu_1\big(a_m(1)\big) \Big) \Big] \Big)^2 \Biggiven \mD,D_1=0 \Big] \ind\Big(D_1 = 0\Big)\Big] = 0.
\end{align}

Combining \eqref{eq:wnn2} with \eqref{eq:wnn3} by \eqref{eq:wnn1} completes the proof of verifying Assumption~\ref{asp:dr1}\ref{asp:dr1,w}.

{\bf Proof of Theorem~\ref{thm:wnn}\ref{thm:wnn3}.} To verify Assumption~\ref{asp:se2}, notice that for any positive integer $p$,
\begin{align*}
    \sum_{j:D_j=1-D_1} \lvert w_{1\leftarrow j} \rvert \cdot \lVert X_j - X_1 \rVert^p =& \sum_{j:D_j=1-D_1} \sum_{m=1}^M \gamma_{M,m} \ind(j_m(1)=j) \lVert X_j - X_1 \rVert^p\\
    =& \sum_{m=1}^M \gamma_{M,m} \lVert X_{j_m(1)} - X_1 \rVert^p.
\end{align*}
From Lemma A.2 in \cite{lin2021estimation}, we obtain
\begin{align*}
    & \E \Big[ \sum_{j:D_j=1-D_1} w_{1\leftarrow j} \lVert X_j - X_1 \rVert^p \Big] = \sum_{m=1}^M \gamma_{M,m} \E \Big[\lVert X_{j_m(1)} - X_1 \rVert^p \Big] = O\Big(\sum_{m=1}^M \gamma_{M,m} \Big(\frac{m}{n}\Big)^{p/d} \Big).
\end{align*}

The proof is then complete by Assumption~\ref{asp:wnn,se}.

{\bf Proof of Theorem~\ref{thm:wnn}\ref{thm:wnn4}.} To verify Assumption~\ref{asp:dml2}, notice that for any $\kappa>0$,
\begin{align*}
    &\E \Big[ \sum_{j:D_j=1-D_1} w_{j\leftarrow 1} \Big]^\kappa = \E \Big[ \sum_{j:D_j=1} \sum_{m=1}^M \gamma_{M,m} \ind(j_m(j)=1) \Big]^\kappa \\
    \le& \Big(M\max_{m \in \zahl{M}}\gamma_{M,m}\Big)^\kappa \E \Big[ \sum_{j:D_j=1} \frac{1}{M} \sum_{m=1}^M  \ind(j_m(j)=1) \Big]^\kappa.
\end{align*}
From Theorem 4.2 in \cite{lin2021estimation} and $M\max_{m \in \zahl{M}}\gamma_{M,m} = O(1)$, we obtain for any $\kappa>0$,
\[
    \E \Big[ \sum_{j:D_j=1-D_1} w_{j\leftarrow 1} \Big]^\kappa = O(1).
\]

For the rate of convergence, we first consider the pointwise bias and variance, and then the global rates of convergence under the $L_1$ risk can be established by leveraging the proof technique of Theorem 4.4 in \cite{lin2021estimation}.

Notice that
\begin{align*}
    &\frac{1}{2} \E\Big[\Big(\sum_{m=1}^M \gamma_{M,m} n_1 \nu_1\big(a_m(x)\big) - \frac{n_1}{n_0} \frac{f_1(x)}{f_0(x)}\Big)^2\Big] \\
    \le& \E\Big[\Big(\sum_{m=1}^M \frac{1}{M} n_1 \nu_1\big(a_m(x)\big) - \frac{n_1}{n_0} \frac{f_1(x)}{f_0(x)}\Big)^2\Big] + \E\Big[\Big(\sum_{m=1}^M \Big(\gamma_{M,m} - \frac{1}{M}\Big) n_1 \nu_1\big(a_m(x)\big)\Big)^2\Big].
\end{align*}

From Theorem 4.3 in \cite{lin2021estimation}, the first term is $O((M/n)^{2/d} + M^{-1})$, and it remains to consider the second term. Notice that in the same way as Lemma~\ref{lemma:moment,catch} and by Assumption~\ref{asp:wnn,dml},
\begin{align*}
    \E\Big[ \Big(\sum_{m=1}^M \Big(\gamma_{M,m} - \frac{1}{M}\Big) n_1\nu_1\big(a_m(x)\big) \Big)^2\Big] 
    \lesssim& n^2 \Big[\int_0^\infty \Big[\sum_{m=1}^M  \Big(\gamma_{M,m} - \frac{1}{M}\Big)^2 \P\Big( U_{(m-1)} \le t \le U_{(m)} \Big) \Big]^{1/2} \d t\Big]^2 \\
    \lesssim& r_3^2.
\end{align*}
We then obtain
\begin{align*}
    \E\Big[\Big(\sum_{m=1}^M \gamma_{M,m} n_1 \nu_1\big(a_m(x)\big) - \frac{n_1}{n_0} \frac{f_1(x)}{f_0(x)}\Big)^2\Big] \lesssim \Big(\frac{M}{n}\Big)^{2/d} + \frac{1}{M} + r_3^2.
\end{align*}
Notice that
\begin{align*}
    &\E\Big\{\sum_{j:D_j=1} \Big[ \sum_{m=1}^M \gamma_{M,m} \Big(  \ind(X_j \in a_m(x)) - \nu_1\big(a_m(x)\big) \Big) \Big]\Big\}^2\\
    =& \Var\Big\{\sum_{j:D_j=1} \Big[ \sum_{m=1}^M \gamma_{M,m} \Big(  \ind(X_j \in a_m(x)) - \nu_1\big(a_m(x)\big) \Big) \Big]\Big\}\\
    =& \E\Big\{\Var\Big\{\sum_{j:D_j=1} \Big[ \sum_{m=1}^M \gamma_{M,m} \Big(  \ind(X_j \in a_m(x)) - \nu_1\big(a_m(x)\big) \Big) \Big] \Biggiven \mX_0,\mD\Big\}\Big\}\\
    =& \E\Big\{n_1 \Var\Big[\sum_{m=1}^M \gamma_{M,m} \ind(X_1 \in a_m(x)) \Biggiven \mX_0,\mD,D_1=1\Big]\Big\}\\
    \le& \E\Big[n_1 \sum_{m=1}^M \gamma_{M,m}^2 \nu_1\big(a_m(x)\big) \Big] \\
    =& \E\Big[\frac{n_1}{n_0} \sum_{m=1}^M \gamma_{M,m}^2\Big] \\
    \lesssim& \sum_{m=1}^M \gamma_{M,m}^2.
\end{align*}
From \eqref{eq:wnn1} and the proof of Theorem 4.4 in \cite{lin2021estimation}, the proof is complete.
\end{proof}

\subsection{Proof of Theorem~\ref{thm:lp}}

\begin{proof}[Proof of Theorem~\ref{thm:lp}]

{\bf Proof of Theorem~\ref{thm:lp}\ref{thm:lp1}.} Assumptions~\ref{asp:weight}, \ref{asp:dr2}\ref{asp:dr2,w}, \ref{asp:se2}\ref{asp:se2,w1} hold by noticing that the construction of weights only based on $\mX$ and $\mD$, and the first column of $\mB_1$ are all one and then
\begin{align*}
    \sum_{j:D_j=1-D_1} w_{1\leftarrow j} = \sum_{j:D_j=1-D_1} e_1^\top (\mB_1^\top \mW_1 \mB_1)^{-1} b_{1j} K_H(X_j-X_1) = e_1^\top (\mB_1^\top \mW_1 \mB_1)^{-1} \mB_1^\top \mW_1 \mB_1 e_1 = 1.
\end{align*}

To verify \ref{asp:dr2}\ref{asp:dr2,w2}, notice that
\begin{align*}
    \E \Big[ \Big\lvert \sum_{j:D_j=1-D_1} w_{j\leftarrow 1} \Big\rvert \Big] \le \E \Big[ \sum_{j:D_j=1-D_1}  \Big\lvert w_{j\leftarrow 1} \Big\rvert \Big] = \E \Big[ \sum_{j:D_j=1-D_1}  \Big\lvert w_{j\leftarrow 1} \Big\rvert \Big] \lesssim \E \Big[ \sum_{j:D_j=1-D_1}  \Big\lvert w_{1\leftarrow j} \Big\rvert \Big].
\end{align*}

Then it suffices to establish the following lemma.
\begin{lemma}\label{lemma:lp,w}
    As long as $K(\cdot)$ is bounded with a compact support and is bounded away from zero, we have
    \begin{align*}
        \E \Big[ \sum_{j:D_j=1-D_1}  \Big\lvert w_{1\leftarrow j} \Big\rvert \Big] = O(1).
    \end{align*}
\end{lemma}

{\bf Proof of Theorem~\ref{thm:lp}\ref{thm:lp2}.} To verify Assumption~\ref{asp:dr1}\ref{asp:dr1,w}, in the same way as Theorem~\ref{thm:kernel}\ref{thm:kernel2}, it suffices to consider
\[
    \E \Big[\E \Big[ \Big(\sum_{j:D_j=0} w_{j\leftarrow 1} - \frac{1-e(X_1)}{e(X_1)}\Big)^2 \Biggiven \mD,D_1=1 \Big] \ind\Big(D_1 = 1\Big)\Big].
\]
We have
\begin{align*}
    &\sum_{j:D_j=0} w_{j\leftarrow 1} - \frac{1-e(X_1)}{e(X_1)} = \sum_{j:D_j=0} e_1^\top (\mB_j^\top \mW_j \mB_j)^{-1} b_{j1} K_H(X_1-X_j) - \frac{1-e(X_1)}{e(X_1)}\\
    =& \Big[\frac{n_0}{n_1} \frac{f_0(X_1)}{f_1(X_1)} - \frac{1-e(X_1)}{e(X_1)} \Big] + \frac{n_0}{n_1} f_1^{-1}(X_1) \Big(\frac{1}{n_0} \sum_{j:D_j=0} K_H(X_j-X_1) - f_0(X_1)\Big)\\
    &+ \frac{n_0}{n_1} \frac{1}{n_0} \sum_{j:D_j=0} \Big(f_1^{-1}(X_j) - f_1^{-1}(X_1)\Big) K_H(X_j-X_1)\\
    &+ \frac{n_0}{n_1} \frac{1}{n_0} \sum_{j:D_j=0} \Big[\Big(n_1 e_1^\top (\mB_j^\top \mW_j \mB_j)^{-1} b_{j1} \Big) - f_1^{-1}(X_j)\Big] K_H(X_j-X_1).
\end{align*}

Notice that the first three terms are the same as \eqref{eq:kernel1}, and then can be handled in the same way as Theorem~\ref{thm:kernel}. Now it suffices to consider the last term, and we have the following lemma.

\begin{lemma}\label{lemma:lp1}
    Under Assumptions~\ref{asp:kernel,dr} and \ref{asp:lp}, we have
    \begin{align*}
        \lim_{n \to \infty } \E \Big[\frac{n_0^2}{n_1^2} \E \Big[ \Big[\frac{1}{n_0} \sum_{j:D_j=0} \Big[\Big(n_1 e_1^\top (\mB_j^\top \mW_j \mB_j)^{-1} b_{j1} \Big) - f_1^{-1}(X_j)\Big] K_H(X_j-X_1) \Big]^2 \Biggiven \mD,D_1=1 \Big] \ind\Big(D_1 = 1\Big)\Big] = 0.
    \end{align*}
\end{lemma}

The proof of Theorem~\ref{thm:lp}\ref{thm:lp3} and Theorem~\ref{thm:lp}\ref{thm:lp4} can be established in the same way as those of Theorems~\ref{thm:kernel}\ref{thm:kernel3} and \ref{thm:kernel}\ref{thm:kernel4} by performing a similar analysis as Lemma~\ref{lemma:lp1}.
\end{proof}

\subsection{Proof of Theorem~\ref{thm:rf}}

\begin{proof}[Proof of Theorem~\ref{thm:rf}]

{\bf Proof of Theorem~\ref{thm:rf}\ref{thm:rf1}.} Notice that
\begin{align*}
    &\sum_{j:D_j=1-D_1} w_{1\leftarrow j}\\
    =& \sum_{j:D_j=1-D_1} B^{-1} \sum_{b=1}^B (\lvert \{k \in \cI^{1-D_1}_b:X_k \in L^{1-D_1}_b(X_1)\} \rvert)^{-1}\ind(j \in \cI^{1-D_1}_b:X_j \in L^{1-D_1}_b(X_1))\\
    =&  B^{-1} \sum_{b=1}^B (\lvert \{k \in \cI^{1-D_1}_b:X_k \in L^{1-D_1}_b(X_1)\} \rvert)^{-1} \sum_{j:D_j=1-D_1} \ind(j \in \cI^{1-D_1}_b:X_j \in L^{1-D_1}_b(X_1))\\
    =&  B^{-1} \sum_{b=1}^B (\lvert \{k \in \cI^{1-D_1}_b:X_k \in L^{1-D_1}_b(X_1)\} \rvert)^{-1} \lvert \{j \in \cI^{1-D_1}_b:X_j \in L^{1-D_1}_b(X_1)\} \rvert \\
    =& 1.
\end{align*}

Then Assumptions~\ref{asp:weight}, \ref{asp:se2}\ref{asp:se2,w1} hold. Assumption \ref{asp:dr2}\ref{asp:dr2,w2} holds by noticing that all weights are nonnegative.

{\bf Proof of Theorem~\ref{thm:rf}\ref{thm:rf-dr}.} To verify Assumption~\ref{asp:dr1}\ref{asp:dr1,w}, it suffices to consider
\[
    \E \Big[\E \Big[ \Big(\sum_{j:D_j=0} w_{j\leftarrow 1} - \frac{1-e(X_1)}{e(X_1)}\Big)^2 \Biggiven \mD,D_1=1 \Big] \ind\Big(D_1 = 1\Big)\Big].
\]

Let $\{L^1_{bt}\}_{t \ge 1}$ be the set of terminal leaves in $L^1_b$ for $b \in \zahl{B}$. We rewrite
\begin{align*}
    &\sum_{j:D_j=0} w_{j\leftarrow 1} - \frac{1-e(X_1)}{e(X_1)} \\
    =& \sum_{j:D_j=0} B^{-1} \sum_{b=1}^B (\lvert \{k \in \cI^1_b:X_k \in L^1_b(X_j)\} \rvert)^{-1}\ind(1 \in \cI^1_b:X_1 \in L^1_b(X_j)) - \frac{1-e(X_1)}{e(X_1)}\\
    =& B^{-1} \sum_{b=1}^B \sum_{t\ge1} \sum_{j:D_j=0}  (\lvert \{k \in \cI^1_b:X_k \in L^1_{bt}\} \rvert)^{-1}\ind(1 \in \cI^1_b:X_1 \in L^1_{bt}) \ind(X_j \in L^1_{bt}) - \frac{1-e(X_1)}{e(X_1)}\\
    =& \Big[B^{-1} \sum_{b=1}^B \sum_{t\ge1} \sum_{j:D_j=0}  (\lvert \{k \in \cI^1_b:X_k \in L^1_{bt}\} \rvert)^{-1}\ind(1 \in \cI^1_b:X_1 \in L^1_{bt}) [\ind(X_j \in L^1_{bt}) - \nu_0(L^1_{bt})]\Big]\\
    &+ \Big[n_0 B^{-1} \sum_{b=1}^B \sum_{t\ge1} (\lvert \{k \in \cI^1_b:X_k \in L^1_{bt}\} \rvert)^{-1}\ind(1 \in \cI^1_b:X_1 \in L^1_{bt}) \Big(\nu_0(L^1_{bt}) - \frac{f_0(X_1)}{f_1(X_1)} \nu_1(L^1_{bt})\Big)\Big]\\
    &+ \Big[\frac{n_0}{n_1} \frac{f_0(X_1)}{f_1(X_1)} \Big(n_1 B^{-1} \sum_{b=1}^B \sum_{t\ge1} (\lvert \{k \in \cI^1_b:X_k \in L^1_{bt}\} \rvert)^{-1}\ind(1 \in \cI^1_b:X_1 \in L^1_{bt}) \nu_1(L^1_{bt}) - 1\Big)\Big]\\
    &+ \Big[\frac{n_0}{n_1} \frac{f_0(X_1)}{f_1(X_1)} - \frac{1-e(X_1)}{e(X_1)} \Big]\\
    =:& S_1 + S_2 + S_3 + S_4.
\end{align*}

The term $S_4$ can be handled by the law of large number. Then it suffices to have the following lemma.

\begin{lemma}\label{lemma:rf1}
    Under Assumptions~\ref{asp:rf,dr}, \ref{asp:rf,honest}, we have
    \begin{align*}
        \lim_{n \to \infty } \E \Big[\E \Big[ S_i^2 \Biggiven \mD,D_1=1 \Big] \ind\Big(D_1 = 1\Big)\Big] = 0, ~~ i=1,2,3.
    \end{align*}
\end{lemma}

{\bf Proof of Theorem~\ref{thm:rf}\ref{thm:rf2}.} It is true by directly checking the definition of honesty.

{\bf Proof of Theorem~\ref{thm:rf}\ref{thm:rf3}.} To verify Assumption~\ref{asp:se2}, notice that if $D_1=1$,
\begin{align*}
    &\sum_{j:D_j=1-D_1} \lvert w_{1\leftarrow j} \rvert \cdot \lVert X_j - X_1 \rVert^k \\
    =& \sum_{j:D_j=0} B^{-1} \sum_{b=1}^B (\lvert \{k \in \cI^{0}_b:X_k \in L^{0}_b(X_1)\} \rvert)^{-1}\ind(j \in \cI^{0}_b:X_j \in L^{0}_b(X_1)) \lVert X_j - X_1 \rVert^k\\
    \le& \sum_{j:D_j=0} B^{-1} \sum_{b=1}^B (\lvert \{k \in \cI^{0}_b:X_k \in L^{0}_b(X_1)\} \rvert)^{-1}\ind(j \in \cI^{0}_b:X_j \in L^{0}_b(X_1)) {\rm diam}^k(L^{0}_b(X_1) \cap \cS)\\
    =& B^{-1} \sum_{b=1}^B {\rm diam}^k(L^{0}_b(X_1) \cap \cS)
\end{align*}
Then
\begin{align*}
    & \E \Big[ \sum_{j:D_j=1-D_1} \lvert w_{1\leftarrow j} \rvert \cdot \lVert X_j - X_1 \rVert^k \Big]\\
    \le& \int \E[ {\rm diam}^k(L^{0}_b(x) \cap \cS)] f_1(x) \d x + \int \E[{\rm diam}^k(L^{1}_b(x) \cap \cS)] f_0(x) \d x.
\end{align*}

The proof is thus complete by Assumption~\ref{asp:rf,se}.
\end{proof}

\section{Proofs of the rest results}\label{sec:proof}

\subsection{Proof of Lemma~\ref{lemma:rf,dist}}

\begin{proof}[Proof of Lemma~\ref{lemma:rf,dist}]
For any $x \in \cS$ and $p \in \zahl{d}$, let $c(x)$ and $c_p(x)$ be the number of splits yielding $L_t(x)$ and that along the $p$-th axis, respectively. From $T$ is $(\alpha,\theta)$-regular, we have $\theta \le s\alpha^{c(x)} \le \lfloor \theta/\alpha \rfloor$, and then $\log(s/\theta)/\log(\alpha^{-1}) \ge c(x) \ge \log(s/(\alpha^{-1}\theta))/\log(\alpha^{-1})$. From $T$ is $\phi$-balanced, we have
\begin{align*}
    c_p(x) \ge \frac{\log(s/(\alpha^{-1}\theta))}{\log(\alpha^{-1})}\frac{\phi}{d}.
\end{align*}

Conditional on $c_p(x)$, let $\cE_t$ be the event that the $t$-th split along $p$-th axis leaves less than $(1-\epsilon)\alpha$ of the parent leaf's Lebesgue measure in one direction for $t \in \zahl{c_p(x)}$. Notice that under $\cE_1^c\cap\cdots\cap\cE_{c_p(x)}^c$, each split along $p$-th axis leaves at most $1-(1-\epsilon)\alpha$ of the Lebesgue measure on each side, and then ${\rm diam}_p^k(L_t(x) \cap \cS) \le [1-(1-\epsilon)\alpha]^{kc_p(x)}$ since $\cS=[0,1]^d$. Then
\begin{align*}
    &\E[{\rm diam}_p^k(L_t(x) \cap \cS) \given c_p(x)] \\
    =& \E[{\rm diam}_p^k(L_t(x) \cap \cS) \ind(\cE_1^c\cap\cdots\cap\cE_{c_p(x)}^c) \given c_p(x)] + 
    \E[{\rm diam}_p^k(L_t(x) \cap \cS) \ind(\cE_1\cup\cdots\cup\cE_{c_p(x)}) \given c_p(x)]\\
    \le& [1-(1-\epsilon)\alpha]^{kc_p(x)} + \P(\cE_1\cup\cdots\cup\cE_{c_p(x)} \given c_p(x)) \le [1-(1-\epsilon)\alpha]^{kc_p(x)} + \sum_{t=1}^{c_p(x)} \P(\cE_t \given c_p(x)).
\end{align*}

Let $S_{p,t}$ be the number of samples in the parent leaf of $t$-th split along $p$-th axis. Conditional on $S_{p,t}$ and $c_p(x)$, the conditional distribution of samples in the parent leaf are i.i.d. from uniform distribution in the parent leaf. Then from the regularity of the tree and the Chernoff inquality,
\begin{align*}
    \P(\cE_t \given S_{p,t},c_p(x)) &= \P({\rm Bin}(S_{p,t},(1-\epsilon)\alpha)>S_{p,t}\alpha) \le \exp(-S_{p,t}\alpha(\log((1-\epsilon)^{-1})-\epsilon)) \\
    &\le \exp(-\theta\alpha(\log((1-\epsilon)^{-1})-\epsilon)).
\end{align*}

Then from the lower bound of $c_p(x)$, $c_p(x) \le c(x)$ and the upper bound of $c(x)$,
\begin{align*}
    &\E[{\rm diam}_p^k(L_t(x) \cap \cS)] \le \E\Big([1-(1-\epsilon)\alpha]^{kc_p(x)} + c_p(x)\exp(-\theta\alpha(\log((1-\epsilon)^{-1})-\epsilon))\Big)\\
    \le& \Big(1-(1-\epsilon)\alpha\Big)^{k \frac{\log(s/(\alpha^{-1}\theta))}{\log(\alpha^{-1})} \frac{\phi}{d} } + \frac{\log(s/\theta)}{\log(\alpha^{-1})} \exp\Big[-\theta\alpha\Big(\log\Big(\frac{1}{1-\epsilon}\Big)-\epsilon\Big)\Big].
\end{align*}
This then completes the proof.
\end{proof}

\subsection{Proof of Lemma~\ref{lemma:mbc,clt}}

\begin{proof}[Proof of Lemma~\ref{lemma:mbc,clt}]
From the standard central limit theorem \cite[Theorem 27.1]{MR1324786}, we have
\begin{align}\label{eq:mbc,clt1}
  \sqrt{n} \Big(\bar{\tau}(\mX) - \tau\Big) \stackrel{\sf d}{\longrightarrow} N\Big(0,V^\tau\Big).
\end{align}

Let $E_{n,i} = (2D_i-1)\Big(1 + \sum_{j:D_j=1-D_i} w_{j\leftarrow i}\Big) \epsilon_i$ for any $i \in \zahl{n}$. Conditional on $\mX,\mD$, $[E_{n,i}]_{i=1}^n$ are independent from Assumption~\ref{asp:dr2}. Notice that $\E [E_{n,i} \given \mX, \mD] = 0$ and $\sum_{i=1}^n \Var[E_{n,i} \given \mX, \mD] = n V^E$. To apply the Lindeberg-Feller central limit theorem \cite[Theorem 27.2]{MR1324786}, it suffices to verify that: for a given $(\mX, \mD)$,
\[
    \frac{1}{nV^E} \sum_{i=1}^n \E\Big[\Big(E_{n,i}\Big)^2 \ind\Big(\lvert E_{n,i} \rvert > \delta \sqrt{nV^E}\Big) \Biggiven \mX, \mD \Big] \to 0,
\]
for all $\delta>0$.

Let $p_1=1+\kappa/2>1$ and take $p_2>1$ such that $p_1^{-1} + p_2^{-1} = 1$ for $\kappa$ in Assumption~\ref{asp:se1}. Let $C_\sigma := \sup_{x \in \cS, \omega \in \{0,1\}} \{\E [\lvert U_\omega \rvert^{2p_1} \given X=x] \vee \E [U^2_\omega \given X=x]\} < \infty$ from Assumption~\ref{asp:se1} and recall $\eta$ from Assumption~\ref{asp:dr}. Then
\begin{align*}
  & \frac{1}{nV^E} \sum_{i=1}^n \E\Big[\Big(E_{n,i}\Big)^2 \ind\Big(\lvert E_{n,i} \rvert > \delta \sqrt{nV^E}\Big) \Biggiven \mX, \mD \Big]\\
  =& \frac{1}{nV^E} \sum_{i=1}^n \E\Big[\Big(1 + \sum_{j:D_j=1-D_i} w_{j\leftarrow i}\Big)^2 \epsilon_i^2 \ind\Big(\lvert E_{n,i} \rvert > \delta \sqrt{nV^E}\Big) \Biggiven \mX, \mD \Big]\\
  \le& \frac{2}{nV^E} \sum_{i=1}^n \Big\{\E\Big[\Big(\frac{D_i}{e(X_i)} + \frac{1-D_i}{1-e(X_i)}\Big)^2 \epsilon_i^2 \ind\Big(\lvert E_{n,i} \rvert > \delta \sqrt{nV^E}\Big) \Biggiven \mX, \mD \Big] \\
  &+ \E\Big[\Big(\sum_{j:D_j=1-D_i} w_{j\leftarrow i} - \Big(D_i \frac{1-e(X_i)}{e(X_i)} + (1-D_i) \frac{e(X_i)}{1-e(X_i)} \Big)\Big)^2 \epsilon_i^2 \ind\Big(\lvert E_{n,i} \rvert > \delta \sqrt{nV^E}\Big) \Biggiven \mX, \mD \Big]\Big\}\\
  \le& \frac{2}{nV^E} \sum_{i=1}^n \Big\{\eta^{-2} \E\Big[\epsilon_i^2 \ind\Big(\lvert E_{n,i} \rvert > \delta \sqrt{nV^E}\Big) \Biggiven \mX, \mD \Big] \\
  &+ C_\sigma \Big[\sum_{j:D_j=1-D_i} w_{j\leftarrow i} - \Big(D_i \frac{1-e(X_i)}{e(X_i)} + (1-D_i) \frac{e(X_i)}{1-e(X_i)} \Big)\Big]^2\Big\}.
\end{align*}

From H\"older's inequality and Markov's inequality, for any $i \in \zahl{n}$,
\begin{align*}
    &\E\Big[\epsilon_i^2 \ind\Big(\lvert E_{n,i} \rvert > \delta \sqrt{nV^E}\Big) \Biggiven \mX, \mD \Big] \le \Big\{\E\Big[\epsilon_i^{2p_1}\Biggiven \mX, \mD\Big]\Big\}^{1/p_1} \Big\{\P\Big(\lvert E_{n,i} \rvert > \delta \sqrt{nV^E} \Biggiven \mX, \mD\Big)\Big\}^{1/p_2}\\
    \le& C_\sigma^{1/p_1} \Big\{\frac{1}{\delta^2 n V^E} \E\Big[E_{n,i}^2\Biggiven \mX, \mD\Big]\Big\}^{1/p_2} \le C_\sigma \Big\{\frac{1}{\delta^2 n V^E} \Big(1 + \sum_{j:D_j=1-D_i} w_{j\leftarrow i}\Big)^2\Big\}^{1/p_2}.
\end{align*}

Let $c_\sigma = \inf_{x \in \cS, \omega \in \{0,1\}} \E [U^2_\omega \given X=x] >0$. From the definition of $V^E$, we have $V^E \ge c_\sigma$ for almost all $\mX, \mD$. Then from Assumption~\ref{asp:dr1}\ref{asp:dr1,w},
\[
  \E \Big[\frac{1}{nV^E} \sum_{i=1}^n \E\Big[\Big(E_{n,i}\Big)^2 \ind\Big(\lvert E_{n,i} \rvert > \delta \sqrt{nV^E}\Big) \Biggiven \mX, \mD \Big] \Big] = O(n^{-1/p_2}) + o(1) = o(1).
\]
We thus obtain
\[
  \frac{1}{nV^E} \sum_{i=1}^n \E\Big[\Big(E_{n,i}\Big)^2 \ind\Big(\lvert E_{n,i} \rvert > \delta \sqrt{nV^E}\Big) \Biggiven \mX, \mD \Big] = o_\P(1).
\]
Applying the Lindeberg-Feller central limit theorem then yields
\begin{align}\label{eq:mbc,clt2}
  \sqrt{n} (V^E)^{-1/2} E_n = \Big(nV^E\Big)^{-1/2} \sum_{i=1}^n E_{n,i} \stackrel{\sf d}{\longrightarrow} N\Big(0,1\Big).
\end{align}

Noticing that $\sqrt{n} \Big(\bar{\tau}(\mX) - \tau\Big)$ and $\sqrt{n} (V^E)^{-1/2} E_n$ are asymptotically independent, leveraging the same argument as made in \citet[Proof of Theorem 4, Page 267]{abadie2006large} and then combining \eqref{eq:mbc,clt1} and \eqref{eq:mbc,clt2} comletes the proof.
\end{proof}

\subsection{Proof of Lemma~\ref{lemma:mbc,ve}}

\begin{proof}[Proof of Lemma~\ref{lemma:mbc,ve}]
We decompose $V^E$ as follows:
\begin{align*}
   V^E =& \frac{1}{n} \sum_{i=1}^n \Big(1 + \sum_{j:D_j=1-D_i} w_{j\leftarrow i}\Big)^2 \sigma_{D_i}^2(X_i)\\
  =& \Big[\frac{1}{n} \sum_{i:D_i = 1}^n \Big(\frac{1}{e(X_i)}\Big)^2 \sigma_1^2(X_i) + \frac{1}{n} \sum_{i:D_i = 0}^n \Big(\frac{1}{1-e(X_i)}\Big)^2 \sigma_0^2(X_i) \Big]\\
  &+ \frac{1}{n} \sum_{i:D_i = 1}^n \Big[\Big(1 + \sum_{j:D_j=1-D_i} w_{j\leftarrow i}\Big)^2 - \Big(\frac{1}{e(X_i)}\Big)^2\Big] \sigma_1^2(X_i) \\
  &+ \frac{1}{n} \sum_{i:D_i = 0}^n \Big[\Big(1 + \sum_{j:D_j=1-D_i} w_{j\leftarrow i}\Big)^2 - \Big(\frac{1}{1-e(X_i)}\Big)^2\Big] \sigma_0^2(X_i).
  \yestag\label{eq:mbc,clt4}
\end{align*}

For the first term in \eqref{eq:mbc,clt4}, notice that $[(X_i,D_i,Y_i)]_{i=1}^n$ are i.i.d. and $\E[D_i (e(X_i))^{-2} \sigma_1^2(X_i)], \E[(1-D_i) (1-e(X_i))^{-2} \sigma_0^2(X_i)] < \infty$. Using the weak law of large numbers \cite[Theorem 2.2.14]{MR3930614}, we have
\begin{align*}
  \frac{1}{n} \sum_{i=1, D_i = 1}^n \Big(\frac{1}{e(X_i)}\Big)^2 \sigma_1^2(X_i) + \frac{1}{n} \sum_{i=1, D_i = 0}^n \Big(\frac{1}{1-e(X_i)}\Big)^2 \sigma_0^2(X_i) \stackrel{\sf p}{\longrightarrow} \E \Big[\frac{\sigma_1^2(X)}{e(X)} + \frac{\sigma_0^2(X)}{1-e(X)}\Big].
\end{align*}

For the second term in \eqref{eq:mbc,clt4}, using the Cauchy--Schwarz inequality,
\begin{align*}
  & \E \Big[\Big\lvert\frac{1}{n} \sum_{D_i = 1}^n \Big[\Big(1 + \sum_{j:D_j=1-D_i} w_{j\leftarrow i}\Big)^2 - \Big(\frac{1}{e(X_i)}\Big)^2\Big] \sigma_1^2(X_i) \Big\rvert \Big]\\
  \le & \E \Big[D_i\Big\lvert\Big(1 + \sum_{j:D_j=1-D_i} w_{j\leftarrow i}\Big)^2 - \Big(\frac{1}{e(X_i)}\Big)^2\Big\rvert \Big] \lVert \sigma_1^2 \rVert_\infty\\
  \le & \Big\{\E \Big[D_i\Big(1 + \sum_{j:D_j=1-D_i} w_{j\leftarrow i} - \frac{1}{e(X_i)}\Big)^2 \Big]\Big\}^{1/2} \Big\{\E \Big[D_i\Big(1 + \sum_{j:D_j=1-D_i} w_{j\leftarrow i} + \frac{1}{e(X_i)}\Big)^2 \Big]\Big\}^{1/2} \lVert \sigma_1^2 \rVert_\infty\\
  = & \Big\{\E \Big[ D_i \Big[\sum_{j:D_j=1-D_i} w_{j\leftarrow i} - \Big(D_i \frac{1-e(X_i)}{e(X_i)} + (1-D_i) \frac{e(X_i)}{1-e(X_i)} \Big) \Big]^2 \Big]\Big\}^{1/2} \\
  & \Big\{\E \Big[D_i\Big(1 + \sum_{j:D_j=1-D_i} w_{j\leftarrow i} + \frac{1}{e(X_i)}\Big)^2 \Big]\Big\}^{1/2} \lVert \sigma_1^2 \rVert_\infty\\
  = & o(1),
\end{align*}
where the last step is due to Assumption~\ref{asp:dr1}. Then we obtain
\begin{align*}
    \frac{1}{n} \sum_{D_i = 1}^n \Big[\Big(1 + \sum_{j:D_j=1-D_i} w_{j\leftarrow i}\Big)^2 - \Big(\frac{1}{e(X_i)}\Big)^2\Big] \sigma_1^2(X_i) \stackrel{\sf p}{\longrightarrow} 0.
\end{align*}

For the third term in \eqref{eq:mbc,clt4}, we can establish in the same way that
\begin{align*}
    \frac{1}{n} \sum_{D_i = 0}^n \Big[\Big(1 + \sum_{j:D_j=1-D_i} w_{j\leftarrow i}\Big)^2 - \Big(\frac{1}{1-e(X_i)}\Big)^2\Big] \sigma_0^2(X_i) \stackrel{\sf p}{\longrightarrow} 0.
\end{align*}

Then from \eqref{eq:mbc,clt4},
\begin{align*}
  V^E \stackrel{\sf p}{\longrightarrow} \E \Big[\frac{\sigma_1^2(X)}{e(X)} + \frac{\sigma_0^2(X)}{1-e(X)}\Big].
\end{align*}
This completes the whole proof.
\end{proof}

\subsection{Proof of Lemma~\ref{lemma:mbc,bias}}

\begin{proof}[Proof of Lemma~\ref{lemma:mbc,bias}]
We decompose $B_n - \hat{B}_n$ as
\begin{align*}
  & \lvert B_n - \hat{B}_n \rvert \\
  =& \Big\lvert  \frac{1}{n}\sum_{i=1}^n (2D_i-1) \Big[\sum_{j:D_j=1-D_i} w_{i\leftarrow j} \Big(\mu_{1-D_i}(X_i) - \mu_{1-D_i}(X_j) - \hat{\mu}_{1-D_i}(X_i) + \hat{\mu}_{1-D_i}(X_j) \Big)\Big] \Big\rvert\\
  \le& \frac{1}{n}\sum_{i=1}^n \sum_{j:D_j=1-D_i} \lvert w_{i\leftarrow j} \rvert \max_{\omega \in \{0,1\}} \Big\lvert \mu_\omega(X_i) - \mu_\omega(X_j) - \hat{\mu}_\omega(X_i) + \hat{\mu}_\omega(X_j) \Big\rvert.
  \yestag\label{eq:mbc,bias1}
\end{align*}

For any $\omega \in \{0,1\}$, by Taylor expansion to $k$-th order,
\begin{align}\label{eq:mbc,bias2}
    \Big\lvert \mu_\omega (X_j) - \mu_{\omega} (X_i) - \sum_{\ell = 1}^{k-1} \frac{1}{\ell!} \sum_{t \in \Lambda_\ell} \partial^t \mu_{\omega} (X_i) (X_j-X_i)^t \Big\rvert \le \max_{t \in \Lambda_k} \lVert \partial^t \mu_{\omega} \rVert_\infty \frac{1}{k!} \sum_{t \in \Lambda_k} \lVert X_j-X_i \rVert^k.
\end{align}
In the same way,
\begin{align}\label{eq:mbc,bias3}
  \Big\lvert \hat{\mu}_\omega (X_j) - \hat{\mu}_{\omega} (X_i) - \sum_{\ell = 1}^{k-1} \frac{1}{\ell!} \sum_{t \in \Lambda_\ell} \partial^t \hat{\mu}_{\omega} (X_i) (X_j-X_i)^t \Big\rvert \le \max_{t \in \Lambda_k} \lVert \partial^t \hat{\mu}_{\omega} \rVert_\infty \frac{1}{k!} \sum_{t \in \Lambda_k} \lVert X_j-X_i \rVert^k.
\end{align}
We also have
\begin{align}\label{eq:mbc,bias4}
  \Big\lvert \sum_{\ell = 1}^{k-1} \frac{1}{\ell!} \sum_{t \in \Lambda_\ell} (\partial^t \hat{\mu}_{\omega} (X_i) - \partial^t \mu_{\omega} (X_i)) (X_j-X_i)^t \Big\rvert \le \sum_{\ell = 1}^{k-1} \max_{t \in \Lambda_\ell} \lVert \partial^t \hat{\mu}_{\omega} - \partial^t \mu_{\omega} \rVert_\infty \frac{1}{\ell!} \sum_{t \in \Lambda_\ell} \lVert X_j-X_i \rVert^\ell.
\end{align}

Plugging \eqref{eq:mbc,bias2}, \eqref{eq:mbc,bias3}, \eqref{eq:mbc,bias4} into \eqref{eq:mbc,bias1} yields
\begin{align*}
  & \lvert B_n - \hat{B}_n \rvert \\
  \lesssim & \Big( \max_{\omega \in \{0,1\}} \max_{t \in \Lambda_k} \lVert \partial^t \mu_{\omega} \rVert_\infty + \max_{\omega \in \{0,1\}} \max_{t \in \Lambda_k} \lVert \partial^t \hat{\mu}_{\omega} \rVert_\infty\Big) \Big(\frac{1}{n} \sum_{i=1}^n \sum_{j:D_j=1-D_i} \lvert w_{i\leftarrow j} \rvert \lVert X_j-X_i \rVert^k \Big) \\
  & + \sum_{\ell=1}^{k-1} \Big( \max_{\omega \in \{0,1\}} \max_{t \in \Lambda_\ell} \lVert \partial^t \hat{\mu}_{\omega} - \partial^t \mu_{\omega} \rVert_\infty \Big) \Big(\frac{1}{n} \sum_{i=1}^n \sum_{j:D_j=1-D_i} \lvert w_{i\leftarrow j} \rvert \lVert X_j-X_i \rVert^\ell \Big).
\end{align*}

This completes the proof by leveraging Assumption~\ref{asp:se2}.
\end{proof}

\subsection{Proof of Lemma~\ref{lemma:kernel1}}

\begin{proof}[Proof of Lemma~\ref{lemma:kernel1}]

{\bf Part I.} We decompose $R_2$ as
\begin{align*}
    & f_1^2(X_1) \E \Big[ R_2^2 \Biggiven X_1,\mD,D_1=1 \Big] \\
    = & f_1^2(X_1) \Big\{\E \Big[ R_2 \Biggiven X_1,\mD,D_1=1 \Big]\Big\}^2 + f_1^2(X_1) \Var \Big[ R_2 \Biggiven X_1,\mD,D_1=1 \Big]\\
    = & \Big\{\E \Big[ K_H(X_2-X_1) \Biggiven X_1,D_1=1,D_2=0 \Big] - f_0(X_1)\Big\}^2 + \\
    &\frac{1}{n_0} \Var \Big[ K_H(X_2-X_1) \Biggiven X_1,D_1=1,D_2=0 \Big].
    \yestag\label{eq:kernel2}
\end{align*}

From $K_H(x) = \lvert H \rvert^{-1/2} K(H^{-1/2}x)$ for any $x \in \bR^d$ and $\int K(z) \d z = 1$,
\begin{align*}
    & \Big\{\E \Big[ K_H(X_2-X_1) \Biggiven X_1,D_1=1,D_2=0 \Big] - f_0(X_1)\Big\}^2\\
    =& \Big(\int K_H(x-X_1) f_0(x) \d x - f_0(X_1) \Big)^2 = \Big(\lvert H \rvert^{-1/2} \int K(H^{-1/2}(x-X_1)) f_0(x) \d x - f_0(X_1) \Big)^2\\
    =& \Big(\int K(z) f_0(X_1+H^{1/2}z) \d z - f_0(X_1) \Big)^2 = \Big(\int K(z) [f_0(X_1+H^{1/2}z) - f_0(X_1)] \d z \Big)^2\\
    =& \int K(z_1) K(z_2) [f_0(X_1+H^{1/2}z_1) - f_0(X_1)] [f_0(X_1+H^{1/2}z_2) - f_0(X_1)] \d z_1 \d z_2.
\end{align*}

From Assumption~\ref{asp:kernel,dr}, we have $f_0$ and $f_1$ are both bounded and bounded away from zero, and then
\begin{align*}
    & \E \Big[\E \Big[ \frac{n_0^2}{n_1^2} f_1^{-2}(X_1) \Big\{\E \Big[ K_H(X_2-X_1) \Biggiven X_1,D_1=1,D_2=0 \Big] - f_0(X_1)\Big\}^2 \Biggiven \mD,D_1=1 \Big] \ind\Big(D_1 = 1\Big)\Big]\\
    = & \E \Big[ \frac{n_0^2}{n_1^2} \int K(z_1) K(z_2) [f_0(x+H^{1/2}z_1) - f_0(x)] [f_0(x+H^{1/2}z_2) - f_0(x)] f_1^{-1}(x) \d x \d z_1 \d z_2 \ind\Big(D_1 = 1\Big)\Big]\\
    \lesssim & \int K(z_1) K(z_2) \Big\lvert [f_0(x+H^{1/2}z_1) - f_0(x)] [f_0(x+H^{1/2}z_2) - f_0(x)] \Big\rvert \d x \d z_1 \d z_2.
\end{align*}

Notice that $f_0$ is bounded and $\int K(z) \d z = 1$. From $\lVert H^{1/2} \rVert_2 \to 0$, $f_0$ continuous almost everywhere, and Fatou's lemma, one reaches
\begin{align*}
    & \limsup_{n \to \infty}  \int K(z_1) K(z_2) \Big\lvert [f_0(x+H^{1/2}z_1) - f_0(x)] [f_0(x+H^{1/2}z_2) - f_0(x)] \Big\rvert \d x \d z_1 \d z_2\\
    \le & \int \limsup_{n \to \infty} K(z_1) K(z_2) \Big\lvert [f_0(x+H^{1/2}z_1) - f_0(x)] [f_0(x+H^{1/2}z_2) - f_0(x)] \Big\rvert \d x \d z_1 \d z_2 \\
    =& 0.
\end{align*}

Then
\begin{align}\label{eq:kernel3}
    \lim_{n \to \infty} \E \Big[\E \Big[ \frac{n_0^2}{n_1^2} f_1^{-2}(X_1) \Big\{\E \Big[ K_H(X_2-X_1) \Biggiven X_1,D_1=1,D_2=0 \Big] - f_0(X_1)\Big\}^2 \Biggiven \mD,D_1=1 \Big] \ind\Big(D_1 = 1\Big)\Big] = 0.
\end{align}

In the same way,
\begin{align*}
     \Var \Big[ K_H(X_2-X_1) \Biggiven X_1,D_1=1,D_2=0 \Big] \le& \E \Big[ K_H^2(X_2-X_1) \Biggiven X_1,D_1=1,D_2=0 \Big]\\
    =& \int K_H^2(x-X_1) f_0(x) \d x\\
    =& \lvert H \rvert^{-1} \int K^2(H^{-1/2}(x-X_1)) f_0(x) \d x \\
    =& \lvert H \rvert^{-1/2} \int K^2(z) f_0(X_1+H^{1/2}z) \d z.
\end{align*}

Then from $\int K^2(z) \d z < \infty$ and $n \lvert H^{1/2} \rvert \to \infty$,
\begin{align*}
    & \E \Big[\E \Big[ \frac{n_0^2}{n_1^2} f_1^{-2}(X_1) \frac{1}{n_0} \Var \Big[ K_H(X_2-X_1) \Biggiven X_1,D_1=1,D_2=0 \Big] \Biggiven \mD,D_1=1 \Big] \ind\Big(D_1 = 1\Big)\Big]\\
    \lesssim & \frac{1}{n} \lvert H \rvert^{-1/2} \int K^2(z) f_0(x+H^{1/2}z) \d x \d z \\
    \lesssim & \frac{1}{n} \lvert H \rvert^{-1/2} \int K^2(z) \d z \\
    =& o(1).
\end{align*}

Then
\begin{align}\label{eq:kernel4}
    \lim_{n \to \infty} \E \Big[\frac{n_0^2}{n_1^2} \E \Big[  f_1^{-2}(X_1) \Var \Big[ K_H(X_2-X_1) \Biggiven X_1,D_1=1,D_2=0 \Big] \Biggiven \mD,D_1=1 \Big] \ind\Big(D_1 = 1\Big)\Big] = 0.
\end{align}

Plugging \eqref{eq:kernel3} and \eqref{eq:kernel4} into \eqref{eq:kernel2} yields
\begin{align*}
    \lim_{n \to \infty } \E \Big[\frac{n_0^2}{n_1^2} \E \Big[ R_2^2 \Biggiven \mD,D_1=1 \Big] \ind\Big(D_1 = 1\Big)\Big] = 0.
\end{align*}

{\bf Part II.} We decompose $R_3$ as
\begin{align*}
    &\E \Big[ R_3^2 \Biggiven X_1,\mD,D_1=1 \Big] \\
    = & \Big\{\E \Big[ R_3 \Biggiven X_1,\mD,D_1=1 \Big]\Big\}^2 + \Var \Big[ R_3 \Biggiven X_1,\mD,D_1=1 \Big]\\
    = & \Big\{\E \Big[ \Big(f_1^{-1}(X_2) - f_1^{-1}(X_1)\Big) K_H(X_2-X_1) \Biggiven X_1,D_1=1,D_2=0 \Big] \Big\}^2 \\
    & + \frac{1}{n_0} \Var \Big[ \Big(f_1^{-1}(X_2) - f_1^{-1}(X_1)\Big) K_H(X_2-X_1) \Biggiven X_1,D_1=1,D_2=0 \Big].
    \yestag\label{eq:kernel5}
\end{align*}

For the first term in \eqref{eq:kernel5},
\begin{align*}
    & \Big\{\E \Big[ \Big(f_1^{-1}(X_2) - f_1^{-1}(X_1)\Big) K_H(X_2-X_1) \Biggiven X_1,D_1=1,D_2=0 \Big] \Big\}^2\\
    = & \Big[ \int \Big(f_1^{-1}(x) - f_1^{-1}(X_1)\Big) K_H(x-X_1) f_0(x) \d x \Big]^2 \\
    = & \Big[ \lvert H \rvert^{-1/2}\int \Big(f_1^{-1}(x) - f_1^{-1}(X_1)\Big) K(H^{-1/2}(x-X_1)) f_0(x) \d x \Big]^2\\
    = & \Big[ \int \Big(f_1^{-1}(X_1+H^{1/2}z) - f_1^{-1}(X_1)\Big) K(z) f_0(X_1+H^{1/2}z) \d z \Big]^2\\
    = & \int \Big(f_1^{-1}(X_1+H^{1/2}z_1) - f_1^{-1}(X_1)\Big) \Big(f_1^{-1}(X_1+H^{1/2}z_2) - f_1^{-1}(X_1)\Big) K(z_1) K(z_2) \\
    & f_0(X_1+H^{1/2}z_1) f_0(X_1+H^{1/2}z_2) \d z_1 \d z_2
\end{align*}

In the same way as \eqref{eq:kernel3}, and then handling the second term of \eqref{eq:kernel5} in the same way as \eqref{eq:kernel4}, we obtain
\begin{align*}
    \lim_{n \to \infty } \E \Big[\frac{n_0^2}{n_1^2} \E \Big[ R_3^2 \Biggiven \mD,D_1=1 \Big] \ind\Big(D_1 = 1\Big)\Big] = 0.
\end{align*}

{\bf Part III.} We decompose $R_4$ as
\begin{align*}
    R_4 = & \frac{1}{n_0} \sum_{j:D_j=0} \Big[\Big(\frac{1}{n_1}\sum_{k:D_k=1} K_H(X_j-X_k)\Big)^{-1} - f_1^{-1}(X_j)\Big] K_H(X_j-X_1)\\
    =& \frac{1}{n_0} \sum_{j:D_j=0} f_1^{-2}(X_j)\Big[f_1(X_j) - \frac{1}{n_1}\sum_{k:D_k=1} K_H(X_j-X_k)\Big] K_H(X_j-X_1)\\
    &+ \frac{1}{n_0} \sum_{j:D_j=0} f_1^{-1}(X_j) \Big[f_1(X_j) - \frac{1}{n_1}\sum_{k:D_k=1} K_H(X_j-X_k)\Big] \\
    &~~\Big[\Big(\frac{1}{n_1}\sum_{k:D_k=1} K_H(X_j-X_k)\Big)^{-1} - f_1^{-1}(X_j)\Big] K_H(X_j-X_1).
\end{align*}

The first term is the dominant term. Notice that conditional on $\mX_1$, the expectation term is
\begin{align*}
    & \E \Big[\frac{1}{n_0} \sum_{j:D_j=0} f_1^{-2}(X_j)\Big[f_1(X_j) - \frac{1}{n_1}\sum_{k:D_k=1} K_H(X_j-X_k)\Big] K_H(X_j-X_1) \Biggiven \mD,D_1=1,\mX_1 \Big]\\
    =& \E \Big[f_1^{-2}(X_2)\Big[f_1(X_2) - \frac{1}{n_1}\sum_{k:D_k=1} K_H(X_2-X_k)\Big] K_H(X_2-X_1) \Biggiven \mD,D_1=1,D_2=0,\mX_1 \Big]\\
    =& \int \Big[f_1(x) - \frac{1}{n_1}\sum_{k:D_k=1} K_H(x-X_k)\Big] K_H(x-X_1) f_1^{-2}(x) f_0(x) \d x
\end{align*}
and the variance term is
\begin{align*}
    & n_0 \Var \Big[\frac{1}{n_0} \sum_{j:D_j=0} f_1^{-2}(X_j)\Big[f_1(X_j) - \frac{1}{n_1}\sum_{k:D_k=1} K_H(X_j-X_k)\Big] K_H(X_j-X_1) \Biggiven \mD,D_1=1,\mX_1 \Big]\\
    =& \Var \Big[f_1^{-2}(X_2)\Big[f_1(X_2) - \frac{1}{n_1}\sum_{k:D_k=1} K_H(X_2-X_k)\Big] K_H(X_2-X_1) \Biggiven \mD,D_1=1,D_2=1,\mX_1 \Big]\\
    \le&  \int \Big[f_1(x) - \frac{1}{n_1}\sum_{k:D_k=1} K_H(x-X_k)\Big]^2 K_H^2(x-X_1) f_1^{-4}(x) f_0(x) \d x.
\end{align*}

Then from the Cauchy-Schwarz inquality, the expectation term can be bounded by
\begin{align*}
    &\frac{1}{2} \E \Big[\Big(\int \Big[f_1(x) - \frac{1}{n_1}\sum_{k:D_k=1} K_H(x-X_k)\Big] K_H(x-X_1) f_1^{-2}(x) f_0(x) \d x\Big)^2 \Biggiven X_1, \mD,D_1=1 \Big]\\
    \le& \frac{1}{n_1^2} \E \Big[\Big(\int \Big[\sum_{k \neq 1:D_k=1} \Big(f_1(x) - K_H(x-X_k) \Big)\Big] K_H(x-X_1) f_1^{-2}(x) f_0(x) \d x\Big)^2 \Biggiven X_1, \mD,D_1=1 \Big]\\
    &+ \frac{1}{n_1^2} \Big(\int \Big[f_1(x) - K_H(x-X_1)\Big] K_H(x-X_1) f_1^{-2}(x) f_0(x) \d x\Big)^2.
\end{align*}

The first term above is the dominant term. We expand it as:
\begin{align*}
    &\E \Big[\Big(\int \Big[\sum_{k \neq 1:D_k=1} \Big(f_1(x) - K_H(x-X_k) \Big)\Big] K_H(x-X_1) f_1^{-2}(x) f_0(x) \d x\Big)^2 \Biggiven X_1, \mD,D_1=1 \Big]\\
    =& (n_1-1)(n_1-2) \E \Big[ \int \Big(f_1(x_1) - K_H(x_1-X_2) \Big) \Big(f_1(x_2) - K_H(x_2-X_3) \Big) \\
    & K_H(x_1-X_1) K_H(x_2-X_1) f_1^{-2}(x_1) f_1^{-2}(x_2) f_0(x_1) f_0(x_2) \d x_1 \d x_2 \Biggiven X_1,D_1=D_2=D_3=1\Big]\\
    &+ (n_1-1) \E \Big[ \int \Big(f_1(x_1) - K_H(x_1-X_2) \Big) \Big(f_1(x_2) - K_H(x_2-X_2) \Big) \\
    & K_H(x_1-X_1) K_H(x_2-X_1) f_1^{-2}(x_1) f_1^{-2}(x_2) f_0(x_1) f_0(x_2) \d x_1 \d x_2 \Biggiven X_1,D_1=D_2=1\Big].
\end{align*}

It suffices to consider the first term above since the second term is $O(n_1^{-1})$. For the first term,
\begin{align*}
    & \E \Big[ \int \Big(f_1(x_1) - K_H(x_1-X_2) \Big) \Big(f_1(x_2) - K_H(x_2-X_3) \Big) \\
    & K_H(x_1-X_1) K_H(x_2-X_1) f_1^{-2}(x_1) f_1^{-2}(x_2) f_0(x_1) f_0(x_2) \d x_1 \d x_2 \Biggiven D_1=D_2=D_3=1\Big]\\
    =& \int \Big(f_1(x_1) - \E[K_H(x_1-X_2) \given D_2=1] \Big) \Big(f_1(x_2) - \E[K_H(x_1-X_3) \given D_3=1] \Big) \\
    & \E[K_H(x_1-X_1) K_H(x_2-X_1) \given D_1=1] f_1^{-2}(x_1) f_1^{-2}(x_2) f_0(x_1) f_0(x_2) \d x_1 \d x_2.
\end{align*}

Then in the same way as establishing $R_2$, and handling the variance term in the same way, we obtain
\begin{align*}
    \lim_{n \to \infty } \E \Big[\E \Big[ R_4^2 \Biggiven \mD,D_1=1 \Big] \ind\Big(D_1 = 1\Big)\Big] = 0.
\end{align*}
This completes the proof.
\end{proof}

\subsection{Proof of Lemma~\ref{lemma:kernel2}}

\begin{proof}[Proof of Lemma~\ref{lemma:kernel2}]
We only consider $R_2$, and $R_3,R_4$ can be handled in a similar way.

For $R_2$, notice that
\begin{align*}
    & f_1(X_1) \E \Big[ \Big\lvert R_2 \Big\rvert \Biggiven X_1,\mD,D_1=1 \Big] \\
    \le & f_1(X_1) \Big\lvert \E \Big[ R_2 \Biggiven X_1,\mD,D_1=1 \Big]\Big\rvert + f_1(X_1) \E \Big[ \Big\lvert R_2 -  \E \Big[ R_2 \Biggiven X_1,\mD,D_1=1 \Big]\Big\rvert \Biggiven X_1,\mD,D_1=1 \Big] \\
    \le & \Big\lvert \E \Big[ K_H(X_2-X_1) \Biggiven X_1,D_1=1,D_2=0 \Big] - f_0(X_1)\Big\rvert + \Big\{\frac{1}{n_0} \Var \Big[ K_H(X_2-X_1) \Biggiven X_1,D_1=1,D_2=0 \Big] \Big\}^{1/2}.
\end{align*}

The second term after taking expectation is $O((n \lvert H^{1/2} \rvert)^{-1})$ as shown in the proof of Lemma~\ref{lemma:kernel2}. For the first term above, from the Lipchitz condition on $\cS$,
\begin{align*}
    & \Big\lvert \E \Big[ K_H(X_2-X_1) \Biggiven X_1,D_1=1,D_2=0 \Big] - f_0(X_1)\Big\rvert = \Big\lvert \int K(z) [f_0(X_1+H^{1/2}z) - f_0(X_1)] \d z \Big\rvert\\
    \le & \int K(z) \Big\lvert f_0(X_1+H^{1/2}z) - f_0(X_1)\Big\rvert \d z \\
    \le & \int K(z) \lVert H^{1/2}z \rVert \d z + \lVert f_0 \rVert_{\infty} \int K(z) \ind(X_1+H^{1/2}z \notin \cS) \d z\\
    \le & \lVert H^{1/2} \rVert_2 \int K(z) \lVert z \rVert \d z + \lVert f_0 \rVert_{\infty} \int K(z) \ind(X_1+H^{1/2}z \notin \cS) \d z.
\end{align*}

Then from the diameter and the surface area of $\cS$ are bounded, and $\int K(z) \lVert z \rVert \d z < \infty$,
\begin{align*}
    & \E \Big[\E \Big[ \frac{n_0}{n_1} f_1^{-1}(X_1) \Big\lvert \E \Big[ K_H(X_2-X_1) \Biggiven X_1,D_1=1,D_2=0 \Big] - f_0(X_1)\Big\rvert \Biggiven \mD,D_1=1 \Big] \ind\Big(D_1 = 1\Big)\Big]\\
    \le & \E \Big[ \frac{n_0}{n_1} \Big[\lVert H^{1/2} \rVert_2 \int K(z) \lVert z \rVert \d z \lambda(\cS) + \lVert f_0 \rVert_{\infty} \int K(z) \ind(x+H^{1/2}z \notin S) \d x \d z \Big] \ind\Big(D_1 = 1\Big)\Big]\\
    \lesssim & \lVert H^{1/2} \rVert_2.
\end{align*}
This completes the proof.
\end{proof}

\subsection{Proof of Lemma~\ref{lemma:moment,catch}}

\begin{proof}[Proof of Lemma~\ref{lemma:moment,catch}]

{\bf Part I.}
Let $Z$ be a copy from $\nu_1$ independent of the data. Then for any $m \in \zahl{M}$,
\begin{align*}
  & \E\big[\nu_1\big(a_m(x)\big)\big] = \P \Big(Z \in a_m(x)\Big) = \P\Big(\lVert \cX_{(m-1)}(Z) - Z \rVert \le \lVert x - Z \rVert \le \lVert \cX_{(m)}(Z) - Z \rVert\Big) \\
  =&  \P\Big( \nu_0(B_{Z,\lVert \cX_{(m-1)}(Z) - Z \rVert}) \le \nu_0(B_{Z,\lVert x - Z \rVert}) \le \nu_0(B_{Z,\lVert \cX_{(m)}(Z) - Z \rVert}) \Big).
  \yestag\label{eq:meancatch0}
\end{align*}
For any given $z \in \bR^d$, $[\nu_0(B_{z,\lVert X_i - z \rVert})]_{i=1}^{n_0}$ are i.i.d. from the uniform distribution on $[0,1]$ since $[X_i]_{i=1}^{n_0}$ are i.i.d. from $\nu_0$ and we use the probability integral transform. Then 
\[
\Big(\nu_0(B_{Z,\lVert \cX_{(m-1)}(Z) - Z \rVert}), \nu_0(B_{Z,\lVert \cX_{(m)}(Z) - Z \rVert})\Big)
\]
has the same distribution as $(U_{(m-1)},U_{(m)})$ and is independent of $Z$.

Let $W = \nu_0(B_{Z,\lVert x-Z \rVert})$. Denote the density of $W$ by $f_W$. In the same way as Lemma 6.3 in \citet{lin2021estimation}, from the almost everywhere continuity of $f_0$ and $f_1$, $f_W(0) = f_1(x)/f_0(x)$, and for any $\epsilon \in (0,1)$, there exists some $\delta = \delta_x>0$ such that for any $w$ with $0 < w \le \delta$, we have $(1-\epsilon) f_1(x)/f_0(x) < f_W(w) < (1+\epsilon) f_1(x)/f_0(x)$.

Let $\eta_n = 4\log(n_0/M)$. Since $M \log n_0 /n_0 \to 0$, we can take $n_0$ large enough so that $\eta_n < \delta$. Then for any $m \in \zahl{M}$,
\begin{align*}
    & \E\big[\nu_1\big(a_m(x)\big)\big] = \P\Big( \nu_0(B_{Z,\lVert \cX_{(m-1)}(Z) - Z \rVert}) \le W \le \nu_0(B_{Z,\lVert \cX_{(m)}(Z) - Z \rVert}) \Big) \\
    \le& \P \Big(\nu_0(B_{Z,\lVert \cX_{(m-1)}(Z) - Z \rVert}) \le W \le \nu_0(B_{Z,\lVert \cX_{(m)}(Z) - Z \rVert}) \le \eta_N \frac{M}{n_0}\Big) + \P\Big(\nu_0(B_{Z,\lVert \cX_{(m)}(Z) - Z \rVert}) > \eta_N \frac{M}{n_0}\Big)\\
    =&\E \Big[ \ind \Big(\nu_0(B_{Z,\lVert \cX_{(m-1)}(Z) - Z \rVert}) \le W \le \nu_0(B_{Z,\lVert \cX_{(m)}(Z) - Z \rVert}) \le \eta_N \frac{M}{n_0}\Big) \Big] + \P\Big(U_{(m)} > \eta_N \frac{M}{n_0}\Big)\\
    =& \E \Big[\int_{\nu_0(B_{Z,\lVert \cX_{(m-1)}(Z) - Z \rVert})}^{\nu_0(B_{Z,\lVert \cX_{(m)}(Z) - Z \rVert})} f_W(w) \ind \Big( w \le \eta_N \frac{M}{n_0}\Big) \d w \Big]  + \P\Big(U_{(m)} > \eta_N \frac{M}{n_0}\Big)\\
    \le& (1+\epsilon) f_1(x)/f_0(x) \E \Big[\nu_0(B_{Z,\lVert \cX_{(m)}(Z) - Z \rVert}) - \nu_0(B_{Z,\lVert \cX_{(m-1)}(Z) - Z \rVert}) \Big]  + \P\Big(U_{(m)} > \eta_N \frac{M}{n_0}\Big)\\
    =& (1+\epsilon) f_1(x)/f_0(x) \E \Big[U_{(m)} - U_{(m-1)} \Big]  + \P\Big(U_{(m)} > \eta_N \frac{M}{n_0}\Big).
\end{align*}

Notice that
\begin{align*}
    \E \Big[U_{(m)} - U_{(m-1)} \Big] = \frac{m}{n_0+1} - \frac{m-1}{n_0+1} = \frac{1}{n_0+1}.
\end{align*}

We then obtain
\begin{align}
    \limsup_{n_0\to\infty} n_0 \sum_{m=1}^M \gamma_{M,m} \E\big[\nu_1\big(a_m(x)\big)\big] \le (1+\epsilon) f_1(x)/f_0(x).
\end{align}

The lower bound can be established in a similar way.

{\bf Part II.}
Let $\tZ_1,\ldots,\tZ_p$ be $p$ independent copies that are drawn from $\nu_1$ independent of the data. Then

\begin{align*}
    &\E\Big[ \Big(\sum_{m=1}^M \gamma_{M,m} \nu_1\big(a_m(x)\big) \Big)^p\Big] = \E\Big[ \prod_{k=1}^p \Big(\sum_{m=1}^M \gamma_{M,m} \ind\big( \tZ_k \in a_m(x)\big) \Big)\Big]\\
    =& \E\Big[ \prod_{k=1}^p \Big(\sum_{m=1}^M \gamma_{M,m} \ind\big( \nu_0(B_{\tZ_k,\lVert \cX_{(m-1)}(\tZ_k) - \tZ_k \rVert}) \le \nu_0(B_{\tZ_k,\lVert x - \tZ_k \rVert}) \le \nu_0(B_{\tZ_k,\lVert \cX_{(m)}(\tZ_k) - \tZ_k \rVert}) \big) \Big)\Big]
\end{align*}

Let $W_k = \nu_0(B_{\tZ_k,\lVert x - \tZ_k \rVert})$ and $V_k^{(m)} = \nu_0(B_{\tZ_k,\lVert \cX_{(m)}(\tZ_k) - \tZ_k \rVert})$ for any $k \in \zahl{p}$ and $m \in \zahl{M}$. Then $[W_k]_{k=1}^p$ are i.i.d. since $[\tZ_k]_{k=1}^p$ are i.i.d.. For any $k \in \zahl{p}$ and $\tZ_k \in \bR^d$ given, $V_k^{(m)} \given \tZ_k $ has the same distribution as $U_{(m)}$. Then for any $k \in \zahl{p}$ and $m \in \zahl{M}$, $V_k^{(m)}$ has the same distribution as $U_{(m)}$, and $V_k^{(m)}$ is independent of $\tZ_k$.

Let $W_{\max} = \max_{k \in \zahl{p}} W_k$ and $V_{\max} = \max_{k \in \zahl{p}} V_k^{(M)}$. Then
\begin{align*}
    &\E\Big[ \Big(\sum_{m=1}^M \gamma_{M,m} \nu_1\big(a_m(x)\big) \Big)^p\Big]\\
    \le & \E\Big[ \prod_{k=1}^p \Big(\sum_{m=1}^M \gamma_{M,m} \ind\big( V_k^{(m-1)} \le W_k \le V_k^{(m)} \big) \Big) \ind\Big(V_{\max} \le \eta_n \frac{M}{n_0} \Big)\Big] + \P\Big( V_{\max} > \eta_n \frac{M}{n_0} \Big).
    \yestag\label{eq:momentcatch3}
\end{align*}

It is easy to check
\begin{align}\label{eq:momentcatch2}
    \lim_{n_0 \to \infty} n_0^p \P\Big( V_{\max} > \eta_n \frac{M}{n_0}\Big) = 0.
\end{align}

Notice that for any $k \in \zahl{p}$, for any $w$ with $0 < w \le \delta$, we have $(1-\epsilon) f_1(x)/f_0(x) < f_{W_k}(w) < (1+\epsilon) f_1(x)/f_0(x)$. Then
\begin{align*}
    & \E\Big[ \prod_{k=1}^p \Big(\sum_{m=1}^M \gamma_{M,m} \ind\big( V_k^{(m-1)} \le W_k \le V_k^{(m)} \big) \Big) \ind\Big(V_{\max} \le \eta_n \frac{M}{n_0} \Big)\Big]\\
    =& \int\E\Big[ \prod_{k=1}^p \Big(\sum_{m=1}^M \gamma_{M,m} \ind\big( V_k^{(m-1)} \le W_k \le V_k^{(m)} \big) \Big) \ind\Big(V_{\max} \le \eta_n \frac{M}{n_0} \Big) \Biggiven W_k=t_k,k\in\zahl{p}\Big] \\
    & f_{(W_1,\ldots,W_p)}(t_1,\ldots,t_p) \d t_1 \cdots \d t_p\\
    \le& \int\E\Big[ \prod_{k=1}^p \Big(\sum_{m=1}^M \gamma_{M,m} \ind\big( V_k^{(m-1)} \le W_k \le V_k^{(m)} \big) \Big) \Biggiven W_k=t_k,k\in\zahl{p}\Big] \\
    & f_{(W_1,\ldots,W_p)}(t_1,\ldots,t_p) \ind\Big(t_1,\ldots,t_p \le \eta_n \frac{M}{n_0}\Big) \d t_1 \cdots \d t_p\\
    =& \int\E\Big[ \prod_{k=1}^p \Big(\sum_{m=1}^M \gamma_{M,m} \ind\big( V_k^{(m-1)} \le W_k \le V_k^{(m)} \big) \Big) \Biggiven W_k=t_k,k\in\zahl{p}\Big] \\
    & f_{W_1}(t_1) \ldots f_{W_p}(t_p) \ind\Big(t_1,\ldots,t_p \le \eta_n \frac{M}{n_0}\Big) \d t_1 \cdots \d t_p\\
    \le& (1+\epsilon)^p \big[f_1(x)/f_0(x)\big]^p \int\E\Big[ \prod_{k=1}^p \Big(\sum_{m=1}^M \gamma_{M,m} \ind\big( V_k^{(m-1)} \le W_k \le V_k^{(m)} \big) \Big) \Biggiven W_k=t_k,k\in\zahl{p}\Big]  \d t_1 \cdots \d t_p.
\end{align*}

From H\"older's inequality,
\begin{align*}
    &\E\Big[ \prod_{k=1}^p \Big(\sum_{m=1}^M \gamma_{M,m} \ind\big( V_k^{(m-1)} \le W_k \le V_k^{(m)} \big) \Big) \Biggiven W_k=t_k,k\in\zahl{p}\Big]\\
    \le&  \prod_{k=1}^p \Big\{\E\Big[ \Big(\sum_{m=1}^M \gamma_{M,m} \ind\big( V_k^{(m-1)} \le W_k \le V_k^{(m)} \big) \Big)^p \Biggiven W_k=t_k,k\in\zahl{p}\Big] \Big\}^{1/p}\\
    =& \prod_{k=1}^p \Big\{\E\Big[ \Big(\sum_{m=1}^M \gamma_{M,m} \ind\big( V_k^{(m-1)} \le W_k \le V_k^{(m)} \big) \Big)^p \Biggiven W_k=t_k\Big] \Big\}^{1/p}\\
    =& \prod_{k=1}^p \Big\{\E\Big[ \Big(\sum_{m=1}^M \gamma_{M,m} \ind\big( U_{(m-1)} \le t_k \le U_{(m)} \big) \Big)^p \Big] \Big\}^{1/p}\\
    =& \prod_{k=1}^p \Big[\sum_{m=1}^M \gamma_{M,m}^p \P\Big( U_{(m-1)} \le t_k \le U_{(m)} \Big) \Big]^{1/p}.
\end{align*}

As long as
\begin{align*}
    \limsup_{n_0 \to \infty} n_0 \int_0^\infty \Big[\sum_{m=1}^M \gamma_{M,m}^p \P\Big( U_{(m-1)} \le t \le U_{(m)} \Big) \Big]^{1/p} \d t \le 1,
\end{align*}
we have
\begin{align}\label{eq:momentcatch4}
    &\lim_{n_0 \to \infty} n_0^p \E\Big[ \prod_{k=1}^p \Big(\sum_{m=1}^M \gamma_{M,m} \ind\big( V_k^{(m-1)} \le W_k \le V_k^{(m)} \big) \Big) \ind\Big(V_{\max} \le \eta_n \frac{M}{n_0} \Big)\Big] \notag\\
    \le& (1+\epsilon)^p \big[f_1(x)/f_0(x)\big]^p.
\end{align}

Plugging \eqref{eq:momentcatch2} and \eqref{eq:momentcatch4} intro \eqref{eq:momentcatch3} yields
\begin{align*}
    \limsup_{n_0 \to \infty} n_0^p \E\Big[ \Big(\sum_{m=1}^M \gamma_{M,m} \nu_1\big(a_m(x)\big) \Big)^p\Big] \le (1+\epsilon)^p \big[f_1(x)/f_0(x)\big]^p.
\end{align*}

By H\"older's inequality, 
\[
    n_0^p \E\Big[ \Big(\sum_{m=1}^M \gamma_{M,m} \nu_1\big(a_m(x)\big) \Big)^p\Big] \ge \Big[n_0 \E\Big[\sum_{m=1}^M \gamma_{M,m} \nu_1\big(a_m(x)\big)\Big]\Big]^p.
\]

From the result of Part I,
\begin{align*}
  \liminf_{n_0 \to \infty} n_0^p \E\Big[ \Big(\sum_{m=1}^M \gamma_{M,m} \nu_1\big(a_m(x)\big) \Big)^p\Big] \ge \big[f_1(x)/f_0(x)\big]^p.
\end{align*}

Then the proof is complete.
\end{proof}

\subsection{Proof of Lemma~\ref{lemma:lp,w}}

\begin{proof}[Proof of Lemma~\ref{lemma:lp,w}]
To consider $\sum_{j:D_j=1-D_1} \lvert w_{1\leftarrow j} \rvert$, we first consider $\sum_{j:D_j=1-D_1} w_{1\leftarrow j}^2$.

We reorder $[X_i]_{i=1}^n$ such that the first $n_0$ are with $D=0$ and the next $n_1$ are with $D=1$, and denote the new sequence by $[Z_1,\ldots,Z_{n_0},Z_1',\ldots,Z_{n_1}']$. We assume without loss of generality that $D_1=1$. Then $\mB_1 = (b_1,\ldots,b_{n_0})^\top$ with $b_i = (1,(Z_i-Z_1')^\top)$, and $\mW_1 = {\rm diag}(w_1,\ldots,w_{n_0})$ with $w_i = K_H(Z_i-Z_1')$ for $i \in \zahl{n_0}$. Then
\begin{align*}
    \sum_{j:D_j=1-D_1} w_{1\leftarrow j}^2 = e_1^\top (\mB_1^\top \mW_1 \mB_1)^{-1} \mB_1^\top \mW_1^2 \mB_1 (\mB_1^\top \mW_1 \mB_1)^{-1} e_1.
\end{align*}

Notice that the weights are scale invariant with respect to the covariates, i.e., minimizing
\begin{align*}
    \sum_{j:D_j=1-D_i} \Big[Y_j - \beta_0 - \beta^\top (X_j - X_i) \Big]^2 K_H(X_j-X_i)
\end{align*}
is equivalent to minimizing
\begin{align*}
    \sum_{j:D_j=1-D_i} \Big[Y_j - \beta_0 - \beta^\top \alpha(X_j - X_i) \Big]^2 K_{\alpha^2H}(\alpha(X_j-X_i))
\end{align*}
for any $\alpha>0$ since we are only interested in the solution for $\beta_0$. Then we can assume without loss of generality that the absolute values of all elements of $Z_i - Z_1',i \in \zahl{n_0}$ are greater than 1. Notice that $\sum_{j:D_j=1-D_1} w_{1\leftarrow j}^2$ is scale invariant with respect to $\mW_1$. Let $\mV_1 = {\rm diag}(v_1,\ldots,v_{n_0})$ with $v_i = K(H^{-1/2}(Z_i-Z_1'))$ for $i \in \zahl{n_0}$. Then using the fact that $\mV_1$ is diagonal with all elements nonnegative and bounded, we have
\begin{align*}
    & \sum_{j:D_j=1-D_1} w_{1\leftarrow j}^2 = e_1^\top (\mB_1^\top \mV_1 \mB_1)^{-1} \mB_1^\top \mV_1^2 \mB_1 (\mB_1^\top \mV_1 \mB_1)^{-1} e_1\\
    \le& \lVert K \rVert_\infty e_1^\top (\mB_1^\top \mV_1 \mB_1)^{-1} \mB_1^\top \mV_1 \mB_1 (\mB_1^\top \mV_1 \mB_1)^{-1} e_1 = \lVert K \rVert_\infty e_1^\top (\mB_1^\top \mV_1 \mB_1)^{-1} e_1\\
    \le& \lVert K \rVert_\infty \lambda_{\rm min}^{-1}(\mB_1^\top \mV_1 \mB_1),
    \yestag\label{eq:lp,w1}
\end{align*}
where $\lambda_{\rm min}(\cdot)$ is the smallest eigenvalue.

From the definition of the smallest eigenvalue,
\begin{align*}
    \lambda_{\rm min}(\mB_1^\top \mV_1 \mB_1) = \min_{u \in \bR^{d+1}} u^\top \mB_1^\top \mV_1 \mB_1 u = \min_{u \in \bR^{d+1}} \sum_{i=1}^{n_0} v_i (\mB_1 u)_i^2.
\end{align*}

Let $\cN := \{i:v_i>0\}$. Using the fact that $K(\cdot)$ is bounded away from zero and denoting the lower bound by $\underline{K}$, we have
\begin{align*}
    \lambda_{\rm min}(\mB_1^\top \mV_1 \mB_1) \ge \underline{K} \min_{u \in \bR^{d+1}} \sum_{i \in \cN} (\mB_1 u)_i^2.
\end{align*}

Without loss of generality, we assume $\lvert \cN \rvert$ is divisible by $d+1$ and then randomly split $\cN$ into $\bigcup_{k=1}^{\lvert \cN \rvert/(d+1)} \cN_k$ with $[\cN_k]_{k=1}^{\lvert \cN \rvert/(d+1)}$ disjoint and all of cardinality $d+1$. Let $\mB_{\cN_k}$ be the matrix constructed by extracting $\mB$'s rows that are indexed by $\cN_k$. We then have
\begin{align*}
    &\lambda_{\rm min}(\mB_1^\top \mV_1 \mB_1) \ge \underline{K} \min_{u \in \bR^{d+1}} \sum_{k=1}^{\lvert \cN \rvert/(d+1)} \sum_{i \in \cN_k} (\mB_1 u)_i^2 \ge \underline{K} \sum_{k=1}^{\lvert \cN \rvert/(d+1)} \min_{u \in \bR^{d+1}}  \sum_{i \in \cN_k} (\mB_1 u)_i^2\\
    =& \underline{K} \sum_{k=1}^{\lvert \cN \rvert/(d+1)} \lambda_{\rm min}(\mB_{\cN_k}^\top \mB_{\cN_k}) \gtrsim \underline{K} \lvert \cN \rvert/(d+1) \gtrsim \lvert \cN \rvert.
    \yestag\label{eq:lp,w2}
\end{align*}

Notice that $w_{1\leftarrow j}=0$ if and only if $K(H^{-1/2}(X_j-X_1))=0$. From the Cauchy-Schwarz inequality and combining \eqref{eq:lp,w1} and \eqref{eq:lp,w2}, we obtain
\begin{align*}
    \Big(\sum_{j:D_j=1-D_1} \lvert w_{1\leftarrow j} \rvert\Big)^2 \le \lvert \cN \rvert \sum_{j:D_j=1-D_1} w_{1\leftarrow j}^2 = O(1),
\end{align*}
and then the proof is complete.
\end{proof}

\subsection{Proof of Lemma~\ref{lemma:lp1}}

\begin{proof}[Proof of Lemma~\ref{lemma:lp1}]

Notice that
\begin{align*}
    & \E \Big[\frac{1}{n_0}  \sum_{j:D_j=0} \Big[\Big(n_1 e_1^\top (\mB_j^\top \mW_j \mB_j)^{-1} b_{j1} \Big) - f_1^{-1}(X_j)\Big] K_H(X_j-X_1) \Biggiven \mD,D_1=1,\mX_1 \Big]\\
    =& \E \Big[ \Big[\Big(n_1 e_1^\top (\mB_2^\top \mW_2 \mB_2)^{-1} b_{21} \Big) - f_1^{-1}(X_2)\Big] K_H(X_2-X_1) \Biggiven \mD,D_1=1,D_2=0,\mX_1 \Big]\\
    =& \int \Big[\Big(n_1 e_1^\top (\mB_x^\top \mW_x \mB_x)^{-1} b_x \Big) - f_1^{-1}(x)\Big] K_H(x-X_1) f_0(x) \d x,
\end{align*}
where for any $x \in \bR^d$, let $\mB_x \in \bR^{n_1 \times (d+1)}$ be the design matrix with the row corresponding to unit $k$ with $D_k=1$ to be $(1,(X_k-x)^\top)$, $\mW_x \in \bR^{n_1 \times n_1}$ be the diagonal matrix with the diagonal element corresponding to unit $k$ with $D_k=1$ to be $K_H(X_k-x)$, and $b_x$ be $(1,(X_1-x)^\top)$.

We can check that
\begin{align*}
    \frac{1}{n_1} \mB_x^\top \mW_x \mB_x = \begin{bmatrix}
        A_1 & A_2\\
        A_3 & A_4
    \end{bmatrix},
\end{align*}
where
\begin{align*}
    & A_1 = \frac{1}{n_1}\sum_{k:D_k=1} K_H(X_k-x), \\
    & A_2^\top = A_3 = \frac{1}{n_1}\sum_{k:D_k=1} K_H(X_k-x) (X_k-x),\\
    & A_4 = \frac{1}{n_1}\sum_{k:D_k=1} K_H(X_k-x) (X_k-x) (X_k-x)^\top.
\end{align*}

One then has
\begin{align*}
    n_1(\mB_x^\top \mW_x \mB_x)^{-1} = \begin{bmatrix}
        (A_1-A_2A_4^{-1}A_3)^{-1} & -(A_1-A_2A_4^{-1}A_3)^{-1}A_2A_4^{-1}\\
        -A_4^{-1}A_3(A_1-A_2A_4^{-1}A_3)^{-1} & (A_4-A_3A_1^{-1}A_2)^{-1}
    \end{bmatrix}.
\end{align*}

We then obtain
\begin{align*}
    & n_1 e_1^\top (\mB_x^\top \mW_x \mB_x)^{-1} b_x\\
    =& (A_1-A_2A_4^{-1}A_3)^{-1} -(A_1-A_2A_4^{-1}A_3)^{-1}A_2A_4^{-1} (X_1-x)^\top.
\end{align*}

We rewrite
\begin{align*}
    &\int \Big[\Big(n_1 e_1^\top (\mB_x^\top \mW_x \mB_x)^{-1} b_x \Big) - f_1^{-1}(x)\Big] K_H(x-X_1) f_0(x) \d x\\
    =& \int f_1^{-2}(x)\Big[f_1(x)-A_1 + A_2 A_4^{-1}A_3 - f_1(x) A_2 A_4^{-1} (X_1-x)^\top\Big] K_H(x-X_1) f_0(x) \d x\\
    &+ \int f_1^{-1}(x)\Big[f_1(x)-A_1+A_2A_4^{-1}A_3\Big]\Big[\Big(n_1 e_1^\top (\mB_x^\top \mW_x \mB_x)^{-1} b_x \Big) - f_1^{-1}(x)\Big] K_H(x-X_1) f_0(x) \d x.
\end{align*}
The first term is the dominant term. We decompose
\begin{align*}
    &\int f_1^{-2}(x)\Big[f_1(x)-A_1 + A_2 A_4^{-1}A_3 - f_1(x) A_2 A_4^{-1} (X_1-x)^\top\Big] K_H(x-X_1) f_0(x) \d x\\
    =& \int f_1^{-2}(x)\Big[f_1(x)-A_1\Big] K_H(x-X_1) f_0(x) \d x + \int f_1^{-2}(x)A_2 A_4^{-1}A_3 K_H(x-X_1) f_0(x) \d x\\
    &- \int f_1^{-1}(x)A_2 A_4^{-1} (X_1-x)^\top K_H(x-X_1) f_0(x) \d x.
\end{align*}

In the above $A_1$ is the kernel density estimation, and then the first term can be handled in the same way as Lemma~\ref{lemma:kernel1}. The third term and the second term can be handled similarly, and we consider the second term. Notice that $A_4$ is an estimate of $\mu_2(K)f_1(x)H$. Then it suffices to consider
\begin{align*}
    &\mu_2(K) \int f_1^{-2}(x)A_2 (\mu_2(K)f_1(x)H)^{-1}A_3 K_H(x-X_1) f_0(x) \d x\\
    =& \int \Big[\frac{1}{n_1}\sum_{k:D_k=1} K_H(X_k-x) (X_k-x)\Big]^\top H^{-1}  \Big[\frac{1}{n_1}\sum_{k:D_k=1} K_H(X_k-x) (X_k-x)\Big] K_H(x-X_1) f_1^{-3}(x) f_0(x) \d x.
\end{align*}

We consider
\begin{align*}
    \E \Big[\Big(\int f_1^{-2}(x)A_2 (\mu_2(K)f_1(x)H)^{-1}A_3 K_H(x-X_1) f_0(x) \d x \Big)^2 \Biggiven X_1, \mD,D_1=1 \Big],
\end{align*}
and then the dominant term is
\begin{align*}
    & \E \Big[\Big(\int \Big[K_H(X_2-x) (X_2-x)\Big]^\top H^{-1}  \Big[K_H(X_3-x) (X_3-x)\Big] K_H(x-X_1) f_1^{-3}(x) f_0(x) \d x \Big)^2 \\
    & \Biggiven X_1, D_1=D_2=D_3=1 \Big].
\end{align*}

From $\int z K(z) \d z = 0$, the above term is zero and then the proof is complete.
\end{proof}

\subsection{Proof of Lemma~\ref{lemma:rf1}}

\begin{proof}[Proof of Lemma~\ref{lemma:rf1}]
{\bf Part I.} For $S_1$, notice that
\begin{align*}
    & \E \Big[B^{-1} \sum_{b=1}^B \sum_{t\ge1} \sum_{j:D_j=0}  (\lvert \{k \in \cI^1_b:X_k \in L^1_{bt}\} \rvert)^{-1}\ind(1 \in \cI^1_b:X_1 \in L^1_{bt}) [\ind(X_j \in L^1_{bt}) - \nu_0(L^1_{bt})] \\
    & \Biggiven \mD,D_1=1,\mX_1,\{\cI^1_b\}_{b=1}^B,\{L^1_{bt}\}_{t \ge 1}\Big] = 0,
\end{align*}
since $[X_j]_{j:D_j=0}$ are independent of the selection of subsamples $\{\cI^1_b\}_{b=1}^B$ and the generation of trees $\{L^1_{bt}\}_{t \ge 1}$. Then
\begin{align*}
    & \E \Big[\Big(B^{-1} \sum_{b=1}^B \sum_{t\ge1} \sum_{j:D_j=0}  (\lvert \{k \in \cI^1_b:X_k \in L^1_{bt}\} \rvert)^{-1}\ind(1 \in \cI^1_b:X_1 \in L^1_{bt}) [\ind(X_j \in L^1_{bt}) - \nu_0(L^1_{bt})] \Big)^2 \\
    & \Biggiven \mD,D_1=1,\mX_1,\{\cI^1_b\}_{b=1}^B,\{L^1_{bt}\}_{t \ge 1}\Big]\\
    =& \Var \Big[B^{-1} \sum_{b=1}^B \sum_{t\ge1} \sum_{j:D_j=0}  (\lvert \{k \in \cI^1_b:X_k \in L^1_{bt}\} \rvert)^{-1}\ind(1 \in \cI^1_b:X_1 \in L^1_{bt}) \ind(X_j \in L^1_{bt})\\
    & \Biggiven \mD,D_1=1,\mX_1,\{\cI^1_b\}_{b=1}^B,\{L^1_{bt}\}_{t \ge 1}\Big]\\
    =& n_0 \Var \Big[B^{-1} \sum_{b=1}^B \sum_{t\ge1} (\lvert \{k \in \cI^1_b:X_k \in L^1_{bt}\} \rvert)^{-1}\ind(1 \in \cI^1_b:X_1 \in L^1_{bt}) \ind(X_2 \in L^1_{bt})\\
    & \Biggiven \mD,D_1=1,D_2=0,\mX_1,\{\cI^1_b\}_{b=1}^B,\{L^1_{bt}\}_{t \ge 1}\Big]\\
    \le& n_0 \E \Big[\Big(B^{-1} \sum_{b=1}^B \sum_{t\ge1} (\lvert \{k \in \cI^1_b:X_k \in L^1_{bt}\} \rvert)^{-1}\ind(1 \in \cI^1_b:X_1 \in L^1_{bt}) \ind(X_2 \in L^1_{bt})\Big)^2\\
    & \Biggiven \mD,D_1=1,D_2=0,\mX_1,\{\cI^1_b\}_{b=1}^B,\{L^1_{bt}\}_{t \ge 1}\Big]\\
    \le& n_0 B^{-1} \sum_{b=1}^B \E \Big[\Big(\sum_{t\ge1} (\lvert \{k \in \cI^1_b:X_k \in L^1_{bt}\} \rvert)^{-1}\ind(1 \in \cI^1_b:X_1 \in L^1_{bt}) \ind(X_2 \in L^1_{bt})\Big)^2\\
    & \Biggiven \mD,D_1=1,D_2=0,\mX_1,\{\cI^1_b\}_{b=1}^B,\{L^1_{bt}\}_{t \ge 1}\Big]\\
    =& n_0 B^{-1} \sum_{b=1}^B \sum_{t\ge1} \E \Big[ (\lvert \{k \in \cI^1_b:X_k \in L^1_{bt}\} \rvert)^{-2}\ind(1 \in \cI^1_b:X_1 \in L^1_{bt}) \ind(X_2 \in L^1_{bt})\\
    & \Biggiven \mD,D_1=1,D_2=0,\mX_1,\{\cI^1_b\}_{b=1}^B,\{L^1_{bt}\}_{t \ge 1}\Big]\\
    =& n_0 B^{-1} \sum_{b=1}^B \sum_{t\ge1} (\lvert \{k \in \cI^1_b:X_k \in L^1_{bt}\} \rvert)^{-2}\ind(1 \in \cI^1_b:X_1 \in L^1_{bt}) \nu_0(L^1_{bt})
\end{align*}

Notice that for $b \in \zahl{B}$, we have
\begin{align*}
    &\sum_{t\ge1} \sum_{i:D_i=1} (\lvert \{k \in \cI^1_b:X_k \in L^1_{bt}\} \rvert)^{-2}\ind(i \in \cI^1_b:X_i \in L^1_{bt}) \nu_0(L^1_{bt}) \\
    =& \sum_{t\ge1} (\lvert \{k \in \cI^1_b:X_k \in L^1_{bt}\} \rvert)^{-1} \nu_0(L^1_{bt}) \le (\min_{t \ge 1}\lvert \{k \in \cI^1_b:X_k \in L^1_{bt}\} \rvert)^{-1} \sum_{t\ge1} \nu_0(L^1_{bt})\\
    =& (\min_{t \ge 1}\lvert \{k \in \cI^1_b:X_k \in L^1_{bt}\} \rvert)^{-1}.
\end{align*}

Then
\begin{align*}
    &\E \Big[\E \Big[ S_1^2 \Biggiven \mD,D_1=1 \Big] \ind\Big(D_1 = 1\Big)\Big]\\
    \le& \E \Big[\E \Big[ n_0 B^{-1} \sum_{b=1}^B \sum_{t\ge1} (\lvert \{k \in \cI^1_b:X_k \in L^1_{bt}\} \rvert)^{-2}\ind(1 \in \cI^1_b:X_1 \in L^1_{bt}) \nu_0(L^1_{bt}) \Biggiven \mD,D_1=1 \Big]\Big]\\
    =&\E \Big[\frac{n_0}{n_1} \E \Big[\sum_{i:D_i=1} B^{-1} \sum_{b=1}^B \sum_{t\ge1} (\lvert \{k \in \cI^1_b:X_k \in L^1_{bt}\} \rvert)^{-2}\ind(i \in \cI^1_b:X_i \in L^1_{bt}) \nu_0(L^1_{bt}) \Biggiven \mD\Big]\Big]\\
    \le&\E \Big[\frac{n_0}{n_1} \E \Big[B^{-1} \sum_{b=1}^B (\min_{t \ge 1}\lvert \{k \in \cI^1_b:X_k \in L^1_{bt}\} \rvert)^{-1} \Biggiven \mD\Big]\Big].
\end{align*}

From Assumption~\ref{asp:rf,dr}\ref{asp:rf,dr,t}, we obtain
\begin{align*}
    \lim_{n \to \infty} \E \Big[\E \Big[ S_1^2 \Biggiven \mD,D_1=1 \Big] \ind\Big(D_1 = 1\Big)\Big] = 0.
\end{align*}

{\bf Part II.} For $S_2$, we decompose it as follows:
\begin{align*}
    & \E \Big[ \Big[n_0 B^{-1} \sum_{b=1}^B \sum_{t\ge1} (\lvert \{k \in \cI^1_b:X_k \in L^1_{bt}\} \rvert)^{-1}\ind(1 \in \cI^1_b:X_1 \in L^1_{bt}) \Big(\nu_0(L^1_{bt}) - \frac{f_0(X_1)}{f_1(X_1)} \nu_1(L^1_{bt})\Big)\Big]^2 \Biggiven \mD,D_1=1 \Big]\\
    =& \Big[\E \Big[ n_0 B^{-1} \sum_{b=1}^B \sum_{t\ge1} (\lvert \{k \in \cI^1_b:X_k \in L^1_{bt}\} \rvert)^{-1}\ind(1 \in \cI^1_b:X_1 \in L^1_{bt}) \Big(\nu_0(L^1_{bt}) - \frac{f_0(X_1)}{f_1(X_1)} \nu_1(L^1_{bt})\Big) \Biggiven \mD,D_1=1 \Big]\Big]^2\\
    +& \Var\Big[ n_0 B^{-1} \sum_{b=1}^B \sum_{t\ge1} (\lvert \{k \in \cI^1_b:X_k \in L^1_{bt}\} \rvert)^{-1}\ind(1 \in \cI^1_b:X_1 \in L^1_{bt}) \Big(\nu_0(L^1_{bt}) - \frac{f_0(X_1)}{f_1(X_1)} \nu_1(L^1_{bt})\Big) \Biggiven \mD,D_1=1 \Big].
\end{align*}

The variance term can be handled in the same way as $S_3$. It then suffices to consider the bias term.
Notice that
\begin{align*}
    & \E \Big[ n_0 B^{-1} \sum_{b=1}^B \sum_{t\ge1} (\lvert \{k \in \cI^1_b:X_k \in L^1_{bt}\} \rvert)^{-1}\ind(1 \in \cI^1_b:X_1 \in L^1_{bt}) \Big(\nu_0(L^1_{bt}) - \frac{f_0(X_1)}{f_1(X_1)} \nu_1(L^1_{bt})\Big) \Biggiven \mD,D_1=1 \Big]\\
    =& \E \Big[ n_0 \sum_{t\ge1} (\lvert \{k \in \cI^1:X_k \in L^1_{t}\} \rvert)^{-1}\ind(1 \in \cI^1:X_1 \in L^1_{t}) \Big(\nu_0(L^1_{t}) - \frac{f_0(X_1)}{f_1(X_1)} \nu_1(L^1_{t})\Big) \Biggiven \mD,D_1=1 \Big],
\end{align*}
where we shorthand $\cI^1_1$ as $\cI^1$ and $L^1_{1t}$ as $L^1_{t}$.

Consider any $\epsilon>0$. From Assumption~\ref{asp:rf,dr}\ref{asp:rf,dr,d} that $f_0$ and $f_1$ are continuous, bounded and bounded away from zero, and $\cS$ is compact, we have $f_0/f_1$ is uniformly continuous on $\cS$. Then there exists $\delta>0$ such that for any $x,z \in \cS$ with $\lVert x-z \rVert \le \delta$, we have 
\[
\lvert f_0(x)/f_1(x) - f_0(z)/f_1(z) \rvert \le \epsilon. 
\]
Let $\cE_t$ be the event 
\[
\Big\{{\rm diam}(L_t^1 \cap \cS) > \delta\Big\}. 
\]
Under $\cE^c_t$, we have for any $X_1 \in L_t^1$,
\begin{align*}
    &\Big\lvert \nu_0(L^1_{t}) - \frac{f_0(X_1)}{f_1(X_1)} \nu_1(L^1_{t}) \Big\rvert = \Big\lvert \int_{L^1_{t}} f_0(z) \d z - \frac{f_0(X_1)}{f_1(X_1)} \int_{L^1_{t}} f_1(z) \d z\Big\rvert\\
    =& \Big\lvert \int_{L^1_{t}} \Big(\frac{f_0(z)}{f_1(z)} - \frac{f_0(X_1)}{f_1(X_1)}\Big)f_1(z) \d z\Big\rvert \le \epsilon \int_{L^1_{t}} f_1(z) \d z = \epsilon \nu_1(L^1_{t}).
\end{align*}

Then
\begin{align*}
    &\Big\lvert \E \Big[ n_0 \sum_{t\ge1} (\lvert \{k \in \cI^1:X_k \in L^1_{t}\} \rvert)^{-1}\ind(1 \in \cI^1:X_1 \in L^1_{t}) \Big(\nu_0(L^1_{t}) - \frac{f_0(X_1)}{f_1(X_1)} \nu_1(L^1_{t})\Big) \ind_{\cE^c_t} \Biggiven \mD,D_1=1 \Big] \Big\rvert\\
    \le& \epsilon \E \Big[ n_0 \sum_{t\ge1} (\lvert \{k \in \cI^1:X_k \in L^1_{t}\} \rvert)^{-1}\ind(1 \in \cI^1:X_1 \in L^1_{t}) \nu_1(L^1_{t}) \Biggiven \mD,D_1=1 \Big] \\
    =& \frac{n_0}{n_1}\epsilon,
\end{align*}
by noticing that
\begin{align*}
    &n_1\E \Big[\sum_{t\ge1} (\lvert \{k \in \cI^1:X_k \in L^1_{t}\} \rvert)^{-1}\ind(1 \in \cI^1:X_1 \in L^1_{t}) \nu_1(L^1_{t}) \Biggiven \mD,D_1=1 \Big]\\
    =& \sum_{i:D_i=1} \E \Big[\sum_{t\ge1} (\lvert \{k \in \cI^1:X_k \in L^1_{t}\} \rvert)^{-1}\ind(i \in \cI^1:X_i \in L^1_{t}) \nu_1(L^1_{t}) \Biggiven \mD,D_1=1 \Big]\\
    =& \sum_{t\ge1} \nu_1(L^1_{t}) \\
    =&1.
\end{align*}

On the other hand,
\begin{align*}
    & \Big\lvert \E \Big[ n_0 \sum_{t\ge1} (\lvert \{k \in \cI^1:X_k \in L^1_{t}\} \rvert)^{-1}\ind(1 \in \cI^1:X_1 \in L^1_{t}) \Big(\nu_0(L^1_{t}) - \frac{f_0(X_1)}{f_1(X_1)} \nu_1(L^1_{t})\Big) \ind_{\cE_t} \Biggiven \mD,D_1=1 \Big] \Big\rvert\\
    \lesssim & \E \Big[ n_0 \sum_{t\ge1} (\lvert \{k \in \cI^1:X_k \in L^1_{t}\} \rvert)^{-1}\ind(1 \in \cI^1:X_1 \in L^1_{t}) \lambda(L^1_{t} \cap \cS) \ind_{\cE_t} \Biggiven \mD,D_1=1 \Big] \\
    =& \frac{n_0}{n_1} \E \Big[ \sum_{t\ge1} \lambda(L^1_{t} \cap \cS) \ind_{\cE_t} \Biggiven \mD \Big]\\ 
    =& \frac{n_0}{n_1} \E \Big[ \int_\cS \ind({\rm diam}(L_t^1(x) \cap \cS) > \delta) \d x \Biggiven \mD \Big] \\
    =& \frac{n_0}{n_1} \int_\cS \E \Big[ \ind({\rm diam}(L_t^1(x) \cap \cS) > \delta) \Biggiven \mD \Big] \d x \\
    \le& \frac{n_0}{n_1} \delta^{-1} \int_\cS \E \Big[ {\rm diam}(L_t^1(x) \cap \cS) \Biggiven \mD \Big] \d x.
\end{align*}

Using Assumption~\ref{asp:rf,dr}\ref{asp:rf,dr,t} and since the selection of $\epsilon$ is arbitrary, we obtain
\begin{align*}
    \lim_{n \to \infty} \E \Big[\E \Big[ S_2^2 \Biggiven \mD,D_1=1 \Big] \ind\Big(D_1 = 1\Big)\Big] = 0.
\end{align*}

{\bf Part III.} For $S_3$, it suffices to consider
\begin{align*}
    \E \Big[\Big(n_1 B^{-1} \sum_{b=1}^B \sum_{t\ge1} (\lvert \{k \in \cI^1_b:X_k \in L^1_{bt}\} \rvert)^{-1}\ind(1 \in \cI^1_b:X_1 \in L^1_{bt}) \nu_0(L^1_{bt}) - 1\Big)^2 \Biggiven \mD,D_1=1\Big].
\end{align*}

Notice that
\begin{align*}
    &\E \Big[n_1 B^{-1} \sum_{b=1}^B \sum_{t\ge1} (\lvert \{k \in \cI^1_b:X_k \in L^1_{bt}\} \rvert)^{-1}\ind(1 \in \cI^1_b:X_1 \in L^1_{bt}) \nu_0(L^1_{bt}) \Biggiven \mD,D_1=1\Big]\\
    =& \E \Big[\sum_{i:D_i=1} B^{-1} \sum_{b=1}^B \sum_{t\ge1} (\lvert \{k \in \cI^1_b:X_k \in L^1_{bt}\} \rvert)^{-1}\ind(i \in \cI^1_b:X_i \in L^1_{bt}) \nu_0(L^1_{bt}) \Biggiven \mD\Big] \\
    =& 1.
\end{align*}
Then
\begin{align*}
    &\E \Big[\Big(n_1 B^{-1} \sum_{b=1}^B \sum_{t\ge1}  (\lvert \{k \in \cI^1_b:X_k \in L^1_{bt}\} \rvert)^{-1}\ind(1 \in \cI^1_b:X_1 \in L^1_{bt}) \nu_0(L^1_{bt}) - 1\Big)^2 \Biggiven \mD,D_1=1\Big]\\
    =& \Var \Big[n_1 B^{-1} \sum_{b=1}^B \sum_{t\ge1}  (\lvert \{k \in \cI^1_b:X_k \in L^1_{bt}\} \rvert)^{-1}\ind(1 \in \cI^1_b:X_1 \in L^1_{bt}) \nu_0(L^1_{bt}) \Biggiven \mD,D_1=1\Big]\\
    =& \Var\Big[\E\Big[n_1 B^{-1} \sum_{b=1}^B \sum_{t\ge1}  (\lvert \{k \in \cI^1_b:X_k \in L^1_{bt}\} \rvert)^{-1}\ind(1 \in \cI^1_b:X_1 \in L^1_{bt}) \nu_0(L^1_{bt}) \Biggiven \mD,D_1=1,\mX_1\Big] \Biggiven \mD,D_1=1 \Big]\\
    &+ \E\Big[\Var\Big[n_1 B^{-1} \sum_{b=1}^B \sum_{t\ge1}  (\lvert \{k \in \cI^1_b:X_k \in L^1_{bt}\} \rvert)^{-1}\ind(1 \in \cI^1_b:X_1 \in L^1_{bt}) \nu_0(L^1_{bt}) \Biggiven \mD,D_1=1,\mX_1\Big] \Biggiven \mD,D_1=1 \Big].
    \yestag\label{eq:rf1}
\end{align*}

Conditional on $\mX_1$, the randomness comes from the subsampling and the construction of trees. For any $\beta = (\beta_1,\ldots,\beta_s)^\top \in \bR^s$, let $L^\beta$ be the tree constructed based on samples $\{X_{\beta_r}\}_{r=1}^s$ with leaves $\{L^\beta_t\}_{t\ge1}$, and $L^\beta(x)$ be the leaf containing $x$ for any test point $x \in \bR^d$. For the first term in \eqref{eq:rf1},
\begin{align*}
    & \E\Big[n_1 B^{-1} \sum_{b=1}^B \sum_{t\ge1}  (\lvert \{k \in \cI^1_b:X_k \in L^1_{bt}\} \rvert)^{-1}\ind(1 \in \cI^1_b:X_1 \in L^1_{bt}) \nu_0(L^1_{bt}) \Biggiven \mD,D_1=1,\mX_1\Big]\\
    =& n_1 \binom{n_1}{s}^{-1}\sum_{\substack{1 \le \beta_1 < \cdots < \beta_s \le n\\ D_{\beta_1}=\cdots=D_{\beta_s}=1}}\E\Big[\sum_{t\ge1}  (\lvert \{k \in \beta:X_k \in L^\beta_t\} \rvert)^{-1}\ind(1 \in \beta:X_1 \in L^\beta_t) \nu_0(L^\beta_t) \Biggiven \mD,D_1=1,\mX_1\Big]\\
    =& n_1 \binom{n_1}{s}^{-1}\sum_{\substack{2 \le \beta_2 < \cdots < \beta_s \le n,\beta_1=1\\ D_{\beta_2}=\cdots=D_{\beta_s}=1}}\E\Big[\sum_{t\ge1}  (\lvert \{k \in \beta:X_k \in L^\beta_t\} \rvert)^{-1}\ind(X_1 \in L^\beta_t) \nu_0(L^\beta_t) \Biggiven \mD,D_1=1,\mX_1\Big]\\
    =& n_1 \binom{n_1}{s}^{-1}\sum_{\substack{2 \le \beta_2 < \cdots < \beta_s \le n,\beta_1=1\\ D_{\beta_2}=\cdots=D_{\beta_s}=1}}\E\Big[(\lvert \{k \in \beta:X_k \in L^\beta(X_1)\} \rvert)^{-1} \nu_0(L^\beta(X_1)) \Biggiven \mD,D_1=1,\mX_1\Big].
\end{align*}

We then apply the Efron-Stein inequality \cite[Theorem 3.1]{boucheron2013concentration}. For any $\ell \in \zahl{n}$ such that $D_\ell=1$, define $\mX_1^\ell$ to be the vector replacing $X_\ell$ by $\tX_\ell$, where $[\tX_\ell]_{\ell:D_\ell=1}$ are independent copies of $[X_\ell]_{\ell:D_\ell=1}$ from $X\given D=1$. For $\ell\ge2$, as long as $\beta_2,\ldots,\beta_s \neq \ell$, we have
\begin{align*}
    &\E\Big[(\lvert \{k \in \beta:X_k \in L^\beta(X_1)\} \rvert)^{-1} \nu_0(L^\beta(X_1)) \Biggiven \mD,D_1=1,\mX_1\Big] \\
    =& \E\Big[(\lvert \{k \in \beta:X_k \in L^\beta(X_1)\} \rvert)^{-1} \nu_0(L^\beta(X_1)) \Biggiven \mD,D_1=1,\mX_1^\ell\Big],
\end{align*}
since the subsamples are the same and then the construction of trees does not depend on $X_\ell$ and $\tX_\ell$. Then from the Efron-Stein inequality and the Cauchy-Schwarz inquality,
\begin{align*}
    &2\Var\Big[\E\Big[n_1 B^{-1} \sum_{b=1}^B \sum_{t\ge1}  (\lvert \{k \in \cI^1_b:X_k \in L^1_{bt}\} \rvert)^{-1}\ind(1 \in \cI^1_b:X_1 \in L^1_{bt}) \nu_0(L^1_{bt}) \Biggiven \mD,D_1=1,\mX_1\Big] \Biggiven \mD,D_1=1 \Big]\\
    \le& \E \Big[\Big[n_1 \binom{n_1}{s}^{-1}\sum_{\substack{2 \le \beta_2 < \cdots < \beta_s \le n,\beta_1=1\\ D_{\beta_2}=\cdots=D_{\beta_s}=1}}\Big(\E\Big[(\lvert \{k \in \beta:X_k \in L^\beta(X_1)\} \rvert)^{-1} \nu_0(L^\beta(X_1)) \Biggiven \mD,D_1=1,\mX_1\Big]\\
    &- \E\Big[(\lvert \{k \in \beta:X_k \in L^\beta(X_1)\} \rvert)^{-1} \nu_0(L^\beta(X_1)) \Biggiven \mD,D_1=1,\mX_1^1\Big]\Big) \Big]^2 \Biggiven \mD,D_1=1 \Big]\\
    &+ \sum_{\ell \neq 1:D_\ell=1}\E \Big[\Big[n_1 \binom{n_1}{s}^{-1}\sum_{\substack{2 \le \beta_3 < \cdots < \beta_s \le n\\ \beta_1=1,\beta_2=\ell,\beta_3,\cdots,\beta_s \neq \ell \\ D_{\beta_3}=\cdots=D_{\beta_s}=1}}\Big(\E\Big[(\lvert \{k \in \beta:X_k \in L^\beta(X_1)\} \rvert)^{-1} \nu_0(L^\beta(X_1)) \Biggiven \mD,D_1=1,\mX_1\Big]\\
    &- \E\Big[(\lvert \{k \in \beta:X_k \in L^\beta(X_1)\} \rvert)^{-1} \nu_0(L^\beta(X_1)) \Biggiven \mD,D_1=1,\mX_1^\ell\Big]\Big)\Big]^2 \Biggiven \mD,D_1=1 \Big]\\
    \le& 4n_1^2 \binom{n_1}{s}^{-2}\Big[\binom{n_1-1}{s-1}^{2} + (n_1-1)\binom{n_1-2}{s-2}^{2}\Big] \\
    &\E \Big[\Big(\E\Big[(\lvert \{k \in \beta:X_k \in L^\beta(X_1)\} \rvert)^{-1} \nu_0(L^\beta(X_1)) \Biggiven \mD,D_1=1,\mX_1\Big]\Big)^2\Biggiven \mD,D_1=1\Big]\\
    \lesssim& \Big(s^2 + \frac{s^4}{n_1}\Big) \E \Big[\Big(\E\Big[(\lvert \{k \in \beta:X_k \in L^\beta(X_1)\} \rvert)^{-1} \nu_0(L^\beta(X_1)) \Biggiven \mD,D_1=1,\mX_1\Big]\Big)^2\Biggiven \mD,D_1=1\Big].
\end{align*}
Notice that
\begin{align*}
    &\E \Big[\Big(\E\Big[(\lvert \{k \in \beta:X_k \in L^\beta(X_1)\} \rvert)^{-1} \nu_0(L^\beta(X_1)) \Biggiven \mD,D_1=1,\mX_1\Big]\Big)^2\Biggiven \mD,D_1=1\Big]\\
    \le& \E \Big[\Big(\E \Big[(\min_{t\ge1}\lvert L_t^\beta \rvert)^{-1}\Biggiven \mD,D_1=1,\mX_1\Big]\Big)\\
    &\Big(\E\Big[(\lvert \{k \in \beta:X_k \in L^\beta(X_1)\} \rvert)^{-1} \nu_0(L^\beta(X_1)) \Biggiven \mD,D_1=1,\mX_1\Big]\Big)\Biggiven \mD,D_1=1\Big]\\
    =& \frac{1}{n_1} \E \Big[(\min_{t\ge1}\lvert L_t^\beta \rvert)^{-1}\Biggiven \mD,D_1=1\Big].
\end{align*}
Then from Assumption~\ref{asp:rf,dr}\ref{asp:rf,dr,t},
\begin{align*}
    &\Var\Big[\E\Big[n_1 B^{-1} \sum_{b=1}^B \sum_{t\ge1}  (\lvert \{k \in \cI^1_b:X_k \in L^1_{bt}\} \rvert)^{-1}\ind(1 \in \cI^1_b:X_1 \in L^1_{bt}) \nu_0(L^1_{bt}) \Biggiven \mD,D_1=1,\mX_1\Big] \Biggiven \mD,D_1=1 \Big]\\
    =& o\Big(\frac{s^2}{n_1} + \frac{s^4}{n_1^2}\Big).
    \yestag\label{eq:rf2}
\end{align*}

For the second term in \eqref{eq:rf1}, from the independence of the trees and shorthanding $\cI^1_1$ as $\cI^1$ and $L^1_{1t}$ as $L^1_{t}$,
\begin{align*}
    &\E\Big[\Var\Big[n_1 B^{-1} \sum_{b=1}^B \sum_{t\ge1}  (\lvert \{k \in \cI^1_b:X_k \in L^1_{bt}\} \rvert)^{-1}\ind(1 \in \cI^1_b:X_1 \in L^1_{bt}) \nu_0(L^1_{bt}) \Biggiven \mD,D_1=1,\mX_1\Big] \Biggiven \mD,D_1=1 \Big]\\
    =& B^{-1} \E\Big[\Var\Big[n_1 \sum_{t\ge1}  (\lvert \{k \in \cI^1:X_k \in L^1_{t}\} \rvert)^{-1}\ind(1 \in \cI^1:X_1 \in L^1_{t}) \nu_0(L^1_{t}) \Biggiven \mD,D_1=1,\mX_1\Big] \Biggiven \mD,D_1=1 \Big]\\
    \le& B^{-1} \E\Big[\E\Big[\Big(n_1 \sum_{t\ge1}  (\lvert \{k \in \cI^1:X_k \in L^1_{t}\} \rvert)^{-1}\ind(1 \in \cI^1:X_1 \in L^1_{t}) \nu_0(L^1_{t}) \Big)^2 \Biggiven \mD,D_1=1,\mX_1\Big] \Biggiven \mD,D_1=1 \Big]\\
    =& B^{-1} \E\Big[ n_1^2 \E\Big[\sum_{t\ge1}  (\lvert \{k \in \cI^1:X_k \in L^1_{t}\} \rvert)^{-2}\ind(1 \in \cI^1:X_1 \in L^1_{t}) \nu_0^2(L^1_{t}) \Biggiven \mD,D_1=1,\mX_1\Big] \Biggiven \mD,D_1=1 \Big]\\
    =& B^{-1} \E\Big[ n_1 \E\Big[\sum_{t\ge1}  (\lvert \{k \in \cI^1:X_k \in L^1_{t}\} \rvert)^{-1} \nu_0^2(L^1_{t}) \Biggiven \mD,D_1=1,\mX_1\Big] \Biggiven \mD,D_1=1 \Big]\\
    \le& B^{-1} \E\Big[ n_1 \E\Big[( \min_{t\ge1}\lvert \{k \in \cI^1:X_k \in L^1_{t}\} \rvert)^{-1} \Biggiven \mD,D_1=1,\mX_1\Big] \Biggiven \mD,D_1=1 \Big] = o(B^{-1}n_1).
    \yestag\label{eq:rf3}
\end{align*}

Plugging \eqref{eq:rf2} and \eqref{eq:rf3} to \eqref{eq:rf1} and using Assumption~\ref{asp:rf,dr}\ref{asp:rf,dr,t}, the proof is complete.
\end{proof}

\end{document}